\documentclass[11pt]{article}
\usepackage[utf8]{inputenc}
\usepackage{times}
\usepackage{amsmath, amssymb, amsthm}
\usepackage{graphicx}
\usepackage[letterpaper, margin=0.75in]{geometry}
\usepackage{subcaption}
\usepackage{xparse}
\usepackage{float}
\usepackage{xcolor}
\usepackage{enumerate}
\usepackage[colorlinks=true, linkcolor=blue, citecolor=blue, urlcolor=blue]{hyperref}
\usepackage[none]{hyphenat}
\usepackage{dsfont}
\usepackage{booktabs}
\usepackage{tikz}
\usepackage{cite}

\numberwithin{equation}{section}
\newtheorem{lemma}{Lemma}[section]
\newtheorem{theorem}{Theorem}[section]

\newtheorem{remark}{Remark}
\newtheorem{corollary}{Corollary}[section]

\allowdisplaybreaks

\title{Domain decomposition methods for the Stokes-Biot model of fluid-poroelastic structure interaction}
\author{Manraj Ghumman\thanks{Department of Mathematics, University of Pittsburgh, Pittsburgh, PA 15260, USA, email: msg85@pitt.edu.}
  \and Manu Jayadharan\thanks{Engineering Sciences \& Applied Mathematics, Northwestern University, 2145 Sheridan Road, Evanston, IL 60208, email: manu.jayadharan@northwestern.edu.}
\and Ivan Yotov\thanks{Department of Mathematics, University of Pittsburgh, Pittsburgh, PA 15260, USA, email: yotov@math.pitt.edu. Partially supported by the US National Science Foundation grant DMS-2410686 and the Alexander von Humboldt Foundation via the Humboldt Research Award.}}
\date{\today}

\begin{document}

\newcommand{\bu}{\mathbf{u}}
\newcommand{\bv}{\mathbf{v}}
\newcommand{\bF}{\mathbf{f}}
\newcommand{\bn}{\mathbf{n}}
\newcommand{\fp}{{fp}}
\newcommand{\bD}{\mathbf{D}}
\newcommand{\bg}{\mathbf{g}}
\newcommand{\bsigma}{\boldsymbol{\sigma}}
\newcommand{\bI}{\boldsymbol{\mathrm{I}}}
\newcommand{\bK}{\mathbf{K}}
\newcommand{\disp}{\boldsymbol{\eta}}
\newcommand{\bt}{\mathbf{t}}
\renewcommand{\ij}{{i,j}}
\newcommand{\bjs}{\scriptscriptstyle\mathrm{BJS}}
\newcommand{\bV}{\mathbf{V}}
\newcommand{\rW}{\mathrm{W}}
\newcommand{\fh}{{fh}}
\newcommand{\ph}{{ph}}
\newcommand{\bH}{\mathbf{H}}
\newcommand{\rL}{\mathrm{L}}
\newcommand{\btau}{\boldsymbol{\tau}}
\newcommand{\bL}{\mathbf{L}}
\renewcommand{\dh}{{dh}}
\newcommand{\bxi}{\boldsymbol{\xi}}
\newcommand{\bchi}{\boldsymbol{\chi}}
\newcommand{\bgamma}{\boldsymbol{\gamma}}
\newcommand{\blambda}{\boldsymbol{\lambda}}
\newcommand{\bLambda}{\boldsymbol{\Lambda}}
\newcommand{\bmu}{\boldsymbol{\mu}}
\newcommand{\br}{\mathbf{r}}
\newcommand{\bq}{\mathbf{q}}
\newcommand{\ff}{{ff}}
\newcommand{\pp}{{pp}}
\newcommand{\fhi}{{fh,i}}
\newcommand{\rphi}{{ph,i}}
\newcommand{\rphj}{{ph,j}}
\newcommand{\hij}{{h,i,j}}
\newcommand{\fhj}{{fh,j}}
\newcommand{\ee}{{ee}}
\newcommand{\rH}{\mathrm{H}}
\newcommand{\kai}{\text{ker }a_{f,i}}
\newcommand{\bpsi}{\boldsymbol{\psi}}
\newcommand{\bvarphi}{\boldsymbol{\varphi}}

\newcommand{\<}{\langle}
\renewcommand{\>}{\rangle}

\maketitle

\begin{abstract}
  We develop a non-overlapping domain decomposition method for the numerical solution of the Stokes-Biot model of fluid-poroelastic structure interaction in a mixed form. The model is based on a velocity-pressure formulation for the free fluid, a three-field stress-displacement-rotation formulation with weakly symmetric stress for the solid deformation, and a Darcy velocity-pressure formulation for the fluid in the poroelastic media. Mass conservation, balance of stress, and the Beavers-Joseph-Saffman slip with friction condition are imposed on the interface. The interface conditions are incorporated through Lagrange multipliers modeling the traces of the displacement and the Darcy pressure. The system is discretized using stable mixed finite element spaces for Stokes flow, elasticity, and Darcy flow. The domain is decomposed into a union of subdomains of either Stokes or Biot type with three types of interfaces: Stokes-Stokes, Biot-Biot, and Stokes-Biot. On the Stokes-Stokes interfaces, a normal stress Lagrange multiplier is introduced to impose weakly velocity continuity, while the Biot-Biot and Stokes-Biot interfaces are equipped with displacement and pressure Lagrange multipliers to impose weakly continuity of normal stress and normal velocity, respectively.  The global problem is reduced via Schur complement to an interface problem for the Lagrange multipliers, which is solved by GMRES. Each iteration requires the solution of local Stokes or Biot problems, which can be performed in parallel. We show that the resulting interface operator is positive definite and analyze the convergence of the GMRES iteration through fields-of-value analysis. Numerical experiments are presented to illustrate the performance of the method.
\end{abstract}

\section{Introduction}
In this paper we develop a non-overlapping domain decomposition method for the numerical solution of the Stokes- Biot model of fluid-poroelastic structure interaction (FPSI). The model problem occurs in a wide range of applica- tions, including in biomedicine, the geosciences, and industry, such as fluid flows in the eye, the brain, or the lungs, arterial flows, interaction of surface and subsurface hydrological systems, groundwater flow through deformable or fractured aquifers, and design of industrial filters. In the mathematical model, the free fluid is governed by the Stokes equations, while the flow in the deformable porous media is modeled by the Biot system of poroelasticity. The two regions are coupled through conservation of mass and momentum interface conditions, as well as the Beavers--Joseph--Saffman slip with friction condition. The model exhibits features of the widely studied coupled Stokes-Darcy flows \cite{LSY,DMQ,ErvJenSun,GalSar,ry2005,gos2011,Changqing} and fluid-structure interaction \cite{galdi2010fundamental,bazilevs2013computational,bungartz2006fluid,richter2017fluid}.The FPSI model has attracted increased attention in recent years. Mathematical analysis can be found in \cite{show2005,cesmelioglu2017analysis,ambartsumyan2019nonlinear,Stokes-Biot-eye,fpsi-mixed-elast,fpsi-msfmfe,Bociu-etal-2021,wy2022,augmented}. For the development of discretization methods, splitting methods, and preconditioning techniques we refer the reader to in
\cite{AKYZ,bukavc2015operator,bukavc2015partitioning,Buk-Yot-Zun-fracture,ambartsumyan2019nonlinear,Stokes-Biot-eye,fpsi-mixed-elast,fpsi-msfmfe,Cesm-Chid,HDG-SB,Boon-precond-SB,hyper-SB,augmented,CLR-NSB-HDG,wy2022,badia2009coupling,fpsi-robin-robin}.

Due to the complexity of the FPSI model, the efficient solution of the resulting coupled algebraic system is key for its use in large scale applications. There has been some work in this direction, including operator-splitting methods in \cite{bukavc2015operator,Bukac-JCP}, Nitsche's coupling in \cite{bukavc2015partitioning},
an optimization-based decoupling method in \cite{Cesm-etal-optim}, a second order in time split schemes in \cite{Kunwar-etal,pb2024}, parameter-robust preconditioning in \cite{Boon-precond-SB}, and Robin-Robin partitioned methods in \cite{badia2009coupling,hyper-SB,pb2024,fpsi-robin-robin,SB-robin-robin-parallel}. 

In this paper we develop a new approach for the solution of the algebraic system arising in the discretization of the Stokes-Biot model, which is based on non-overlapping domain decomposition \cite{Toselli-Widlund,QV}. To the best of our knowledge, this approach has not been studied in the literature for FPSI. Domain decomposition methods split the computational domain into multiple non-overlapping subdomains. The underlying partial differential equation (PDE) model is locally discretized in each subdomain and interface continuity conditions are enforced through Lagrange multipliers. The global problem is reduced via Schur complement to an interface problem for the Lagrange multipliers, which can be solved by a Krylov space iterative method. Each iteration requires the solution of local subdomain problems of lower complexity that are easier to solve. This approach naturally leads to scalable algorithms on massively parallel computers with distributed memory. Previous works on domain decomposition methods for mixed finite element discretizations of PDEs that have informed the developments in this paper include methods for Darcy flow \cite{glowinski1988domain,cowsar1995balancing,arbogast2000mixed}, elasticity \cite{EldarElast}, poroelasticity \cite{ManuBiot,ManuMortarBiot}, and Stokes-Darcy flow \cite{Changqing}. We also mention other relevant domain decomposition and preconditioning works on elasticity \cite{Kim-elast-BDDC}, poroelasticity \cite{Ahmed-FS-JCAM,Ahmed-apost-CMAME,FlorezDD,girault2011domain,Boon-precond-Biot}, and Stokes-Darcy flow \cite{Disc-Quart-Valli-2007,Hu-precond-SD,Boon-SD,Boon-etal-robust-SD}.

In this work we consider a mixed formulation for the Stokes-Biot model, which is based on a velocity-pressure Stokes formulation, a three-field stress-displacement-rotation elasticity formulation with weakly symmetric stress, and a velocity-pressure Darcy formulation \cite{fpsi-mixed-elast}. The advantages of this formulation at the discrete level include local momentum and mass conservation in the poroelastic region, robustness with respect to the physical parameters, and locking-free behavior in the almost incompressible regime. The two regions are coupled at the interface via continuity of flux, force and momentum balance, and the Beavers-Joseph-Saffman slip with friction condition. The system is discretized using stable mixed finite element spaces for Stokes flow, elasticity, and Darcy flow. We allow for non-matching grids on the Stokes-Biot interfaces, but for simplicity focus on matching grids on the Stokes-Stokes and Biot-Biot interfaces. The domain is decomposed into a union of subdomains of either Stokes or Biot type, resulting in three types of interfaces: Stokes-Stokes, Biot-Biot, and Stokes-Biot. A normal stress Lagrange multiplier is introduced on the Stokes-Stokes interfaces to impose weakly velocity continuity. On the Biot-Biot and Stokes-Biot interfaces, displacement and pressure Lagrange multipliers are introduced to impose weakly continuity of normal stress and normal velocity, respectively. We use Nitsche's method to impose weakly the Dirichlet boundary conditions in the Stokes region. This allows us to define the normal stress Lagrange multiplier space on the Stokes-Stokes interfaces as the trace of the Stokes velocity space, without the need to modify it at the intersection points with the outer Dirichlet boundary. The displacement and pressure Lagrange multiplier spaces on the Biot-Biot and Stokes-Biot interfaces are chosen to be the normal traces of the solid stress and the Darcy velocity spaces, respectively. 
The global system is reduced to an interface problem for the Lagrange multipliers by forming the Schur complement. The interface operator involves solving Stokes or Biot subdomain problems, which can be performed in parallel. It can be considered as a generalized Steklov-Poincar\'e operator. The Stokes subdomain problems take normal fluid stress Neumann data on the interfaces and return the computed velocity, while the Biot subdomain problems take displacement and pressure Dirichlet data and return the computed normal solid stress and Darcy flux. The interface operator gives the jump in velocity on the Stokes-Stokes interfaces, the jump in normal solid stress and normal flux on the Biot-Biot interfaces, as well as the jump in flux and momentum and the tangential slip on the Stokes-Biot interfaces. When a Krylov space method is employed for the the solution of the interface problem, each iteration requires computing the action of the interface operator, i.e., the solution of Stokes or Biot subdomain problems. We prove that the interface operator is positive definite, despite being non-symmetric, which allows us to employ GMRES for its solution. Furthermore, we derive lower and upper spectral bounds and analyze the convergence of the GMRES iteration through fields-of-value estimates. The analysis is based on characterizing the norm induced by the interface bilinear form in terms of norms of the discrete extensions computed through the subdomain problems and establishing discrete extension and trace inequalities. The use of Nitsche's method for the Stokes Dirichlet boundary conditions presents a significant technical challenge, due to the presence of boundary terms scaled by $h^{-1}$. To deal with this, we use a modified scaled norm for the Stokes velocity induced by the Stokes bilinear form and construct a special discrete Stokes subdomain extension that vanishes on the Dirichlet boundary. 

The rest of the paper is organized as follows. The model problem is presented in Section~\ref{sec:model}, followed by its monolithic mixed finite element discretization in Section~\ref{sec:mfe}. Section~\ref{sec:dd} is devoted to the development and analysis of the domain decomposition algorithm. Numerical experiments illustrating the performance of the method are presented in Section~\ref{sec:numer}. We end with conclusions and possible extensions in Section~\ref{sec:concl}. 

\section{Model problem}\label{sec:model}

\begin{minipage}[c]{0.65\textwidth}
We use the Stokes equations to model the fluid motion in the free fluid region and the steady-state Biot equations to model the poroelastic region. The analysis extends directly to the unsteady quasistatic Biot model, since its backward Euler discretization yields the steady-state model. Let $\Omega$ be an open and bounded domain in $\mathbb{R}^d$, $d = 2\text{ or } 3$, and $\overline\Omega = \overline\Omega_f\cup\overline\Omega_p$ where $\Omega_f$ is the fluid region and $\Omega_p$ is the poroelastic region. Let $\Gamma_{\fp} = \partial\Omega_f\cap\partial\Omega_p$ denote the interface between the Stokes and Biot regions. Let $\partial\Omega = \Gamma_f\cup\Gamma_p$ where $\Gamma_f = \partial\Omega_f\setminus\Gamma_{\fp}$ denotes the external Stokes boundary and $\Gamma_p = \partial\Omega_p\setminus\Gamma_{\fp}$ denotes the external boundary for the poroelastic region. The computational domain is shown in Figure~\ref{fig:domain}.
\end{minipage}
\hfill
\begin{minipage}[c]{0.35\textwidth}
  \centering
    \includegraphics[width=.9\textwidth]{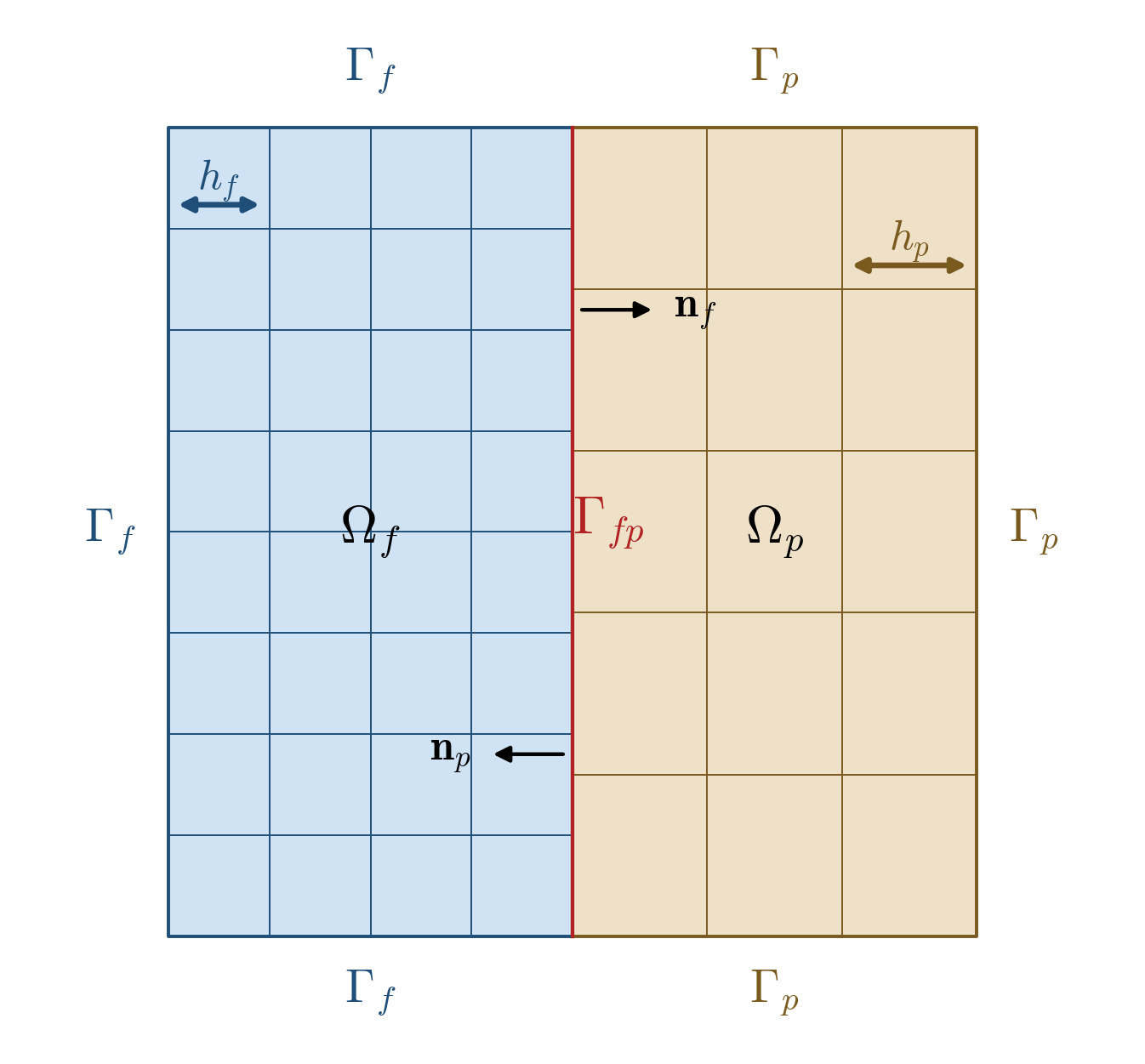}
  \captionof{figure}{Computational domain.}
\label{fig:domain}
\end{minipage}

\subsection{Stokes problem}
The free fluid in $\Omega_f$ is governed by the Stokes equations
\begin{subequations}\label{e2.1}
\begin{align}
    -\nabla\cdot \bsigma_f &= \bF_f,\quad \nabla\cdot \bu_f = q_f\quad \text{in } \Omega_f,\label{e2.1a}\\
    \bu_f &= \bg_f^D\quad \text{on }\Gamma_f^D,\label{e2.1b}\\
    \bsigma_f\bn_f &= \bg_f^N\quad \text{on }\Gamma_f^N,\label{e2.1c}
\end{align}
\end{subequations}
where $\bu_f$ is the fluid velocity, $p_f$ is the fluid pressure, $\bsigma_f = -p_f\bI+2\mu\bD(\bu_f)$ is the stress tensor. Here $\bD(\bv) = \frac{1}{2}(\nabla \bv + \nabla\bv^\mathrm{T})$ and $\mu > 0$ is the fluid viscosity. In addition, $\bF_f$ is a fluid body force, $q_f$ is an external source or sink, and $\bg_f^D$, $\bg_f^N$ are the Dirichlet and Neumann boundary conditions on the external Stokes boundaries $\Gamma_f^D$ and $\Gamma_f^N$, respectively, with $\Gamma_f = \Gamma_f^D\cup \Gamma_f^N$. To avoid non-uniqueness issues, we assume that $|\Gamma_f^D|> 0$.

\subsection{Biot problem}
The poroelastic region is governed by the steady-state Biot system 
\begin{subequations}\label{e2.2}
\begin{align}
    -\nabla \cdot \bsigma_p = \bF_p,\quad \mu\bK^{-1}&\bu_p +\nabla p_p = \boldsymbol{0},\quad s_0 p_p + \alpha \nabla \cdot \disp_p+\nabla \cdot \bu_p = q_p\quad \text{in }\Omega_p,\label{e2.2a}\\
    &p_p = g_p^{D_p}\quad \text{on }\Gamma_p^{D_p}, \quad \bu_p\cdot \bn_p = 0\quad \text{on }\Gamma_p^{N_v},\label{e2.2b}\\
    &\disp_p = \bg_p^{D_d}\quad \text{on }\Gamma_p^{D_d},\quad \bsigma_p\bn_p = \boldsymbol{0}\quad \text{on }\Gamma_p^{N_s}\label{e2.2c}.
\end{align}
\end{subequations}
Here $\bu_p$ is the Darcy velocity, $p_p$ is the Darcy pressure, $\disp_p$ is the displacement, $s_0> 0$ is the mass storativity coefficient, $0< \alpha\leq 1$ is the Biot-Willis constant, and $\bsigma_p$ is the poroelastic stress tensor, with
\begin{align}\label{e2.3}
    \bsigma_p = \bsigma_e - \alpha p_p\bI,\quad A\bsigma_e = \bD(\disp_p),
\end{align}
where $\boldsymbol{\sigma}_e$ is the elastic stress tensor and $A:\mathbb{S}\rightarrow\mathbb{S}$ is the compliance tensor. We assume that $A$ is a symmetric, uniformly positive definite, and bounded operator on the space $\mathbb{S}$ of $d\times d$ symmetric matrices, and that it can be extended to the space $\mathbb{M}$ of general $d\times d$ matrices. The first equation in \eqref{e2.3} models the poroelastic stress due to fluid pressure and the elastic body, and the second equation characterizes the stress-strain relation of an elastic body. In addition, $\bK$ is a symmetric and uniformly positive definite and bounded rock permeability tensor. Finally, $\bF_p$ is a structure body force, $q_p$ is a source or sink term, $g_p^{D_p}$ is the Dirichlet boundary data for the Darcy equations on $\Gamma_p^{D_p}$, with 
$\Gamma_p = \Gamma_p^{D_p}\cup\Gamma_p^{N_v}$, and $\bg_p^{D_d}$ is the Dirichlet boundary data for the elasticity equations on $\Gamma_p^{D_d}$, with $\Gamma_p = \Gamma_p^{D_d}\cup\Gamma_p^{N_s}$. Furthermore, to avoid non-uniqueness issues, we assume that $|\Gamma_p^{D_p}|,\ |\Gamma_p^{D_d}|\ > 0$.

\subsection{Stokes-Biot coupling}
The Stokes and Biot equations are coupled through interface condition on the fluid-poroelastic structure interface $\Gamma_{\fp}$. They are mass conservation, balance of normal components of the stresses, conservation of momentum, and the Beavers–Joseph–Saffman (BJS) condition modeling slip with friction:
\begin{subequations}\label{e2.4}
\begin{align}
    \bu_f\cdot \bn_f\ +\left(\disp_p + \bu_p\right)\cdot \bn_p = 0,\quad -(\bsigma_f\bn_f)\cdot\bn_f = p_p\quad \text{on } \Gamma_{\fp},\label{e2.4a}\\
    \bsigma_f\bn_f + \bsigma_p\bn_p = 0,\quad -(\bsigma_f\bn_f)\cdot\bt_{f,l} = \mu\alpha_{\bjs}\sqrt{\bK^{-1}_l}\left(\bu_f-\disp_p\right)\cdot \bt_{f,l}\quad \text{on }\Gamma_{\fp},\label{e2.4b}
\end{align}
\end{subequations}
where $\bn_f$ and $\bn_p$ are the outward unit normal vectors to $\partial\Omega_f$ and $\partial\Omega_p$ respectively, $\bt_{f,l}$, $l=1,\cdots,d-1$ is an orthonormal set of tangential vectors on $\Gamma_{\fp}$, $\bK_l = \left(\bK\bt_{f,l}\right)\cdot\bt_{f,l}$, and $\alpha_{\bjs}$ is a friction coefficient. 

\section{Mixed finite element method}\label{sec:mfe}
We derive a mixed finite element method for the model problem \eqref{e2.1}, \eqref{e2.2}, \eqref{e2.4}. Assume that $\Omega_f$ and $\Omega_p$ are polygonal ($d=2$) or polyhedral ($d=3$) subdomains. Let $\mathcal{T}^f_{h_f}$ and $\mathcal{T}^p_{h_p}$ be shape regular affine finite element partitions of $\Omega_f$ and $\Omega_p$, respectively, which may be non-matching along the interface $\Gamma_\fp$. Here $h_f,\ h_p$ are the maximum diameters in $\Omega_f,\ \Omega_p$, respectively, and we denote $h = \max\{h_f,\ h_p\}$. We also assume that the trace of $\mathcal{T}^f_{h_f}$ on $\Gamma_f^D$ is quasiuniform, see \cite{ciarlet}.

For a domain $G$ in $\mathbb{R}^d$, the $\rL^2(G)$ inner product and norm for scalar, vector-valued, and tensor-valued functions are denoted by $(\cdot,\cdot)_G$ and $\|\cdot\|_G$, respectively. The norm and seminorm of the Hilbert spaces $\rH^k(G)$ are denoted by $\|\cdot\|_{k,G}$ and $|\cdot|_{k,G}$, respectively. For a section of the domain or element boundary $S\subset\mathbb{R}^{d-1}$ we write $\langle\cdot,\cdot\rangle_S$ and $\|\cdot\|_S$ for the $\rL^2(G)$ inner product and norm, respectively. For a scalar-valued space $H$, we use $\bH$ and $\mathbb{H}$ to denote the corresponding vector-valued and tensor-valued spaces, respectively. We also let $C$ denote a generic constant independent of $h$.

\subsection{Stokes variational formulation}\label{s3.1}

In the fluid region $\Omega_f$, let
\begin{align*}
    \bV_f := \bH^1(\Omega_f),\quad \rW_f = \rL^2(\Omega_f).
\end{align*}
Let $\bV_\fh\times \rW_\fh\subset \bV_f\times \rW_f$ be any Stokes mixed finite element spaces satisfying the inf-sup condition
\begin{align}\label{e3.1}
    C\sup_{\bv_\fh\in \bV^{0,\Gamma_f^D}_\fh\setminus \{\boldsymbol{0}\}} \frac{(w_\fh, \nabla\cdot \bv_\fh)_{\Omega_f}}{\|\bv_\fh\|_{1,\Omega_f}}\geq \|w_\fh\|_{\Omega_f},\quad \forall\ w_\fh\in \rW_\fh,
\end{align}
where
\begin{align*}
    \bV_\fh^{0, \Gamma_f^D} = \left\{\bv_\fh\in\bV_\fh\ :\ \bv_\fh = \mathbf{0}\ \text{on }\Gamma_f^D \right\}.
\end{align*}
We note that $\bv_\fh$ in \eqref{e3.1} can be restricted to $\bV_\fh^{0, \Gamma_f^D}$, since part of $\partial\Omega_f$ is $\Gamma_\fp$. Multiplying \eqref{e2.1a} by $\bv_\fh\in\bV_\fh$, integrating by parts, and substituting the normal fluid stress on $\Gamma_f$ from \eqref{e2.1c}, we get
\begin{align*}
     &2\mu(\bD(\bu_\fh),\bD(\bv_\fh))_{\Omega_f}- \langle \bsigma_\fh \bn_f, \bv_\fh\rangle_{\partial\Omega_f\setminus{\Gamma_f^N}}-(p_\fh,\nabla\cdot \bv_\fh)_{\Omega_f}= (\bF_f, \bv_f)_{\Omega_f}+\langle\bg_f^N,\bv_\fh\rangle_{\Gamma_f^N},
\end{align*}
where $\bsigma_\fh = -p_\fh\bI+2\mu\bD(\bu_\fh)$. We impose the Dirichlet boundary condition weakly via Nitsche's method, obtaining that $(\bu_{fh},p_{fh} \in \bV_{fh}\times W_{fh}$ satisfy, for all $(\bv_\fh,w_\fh)\in\bV_\fh\times\rW_\fh$
\begin{subequations}\label{e3.2}
\begin{align}
    &2\mu(\bD(\bu_\fh),\bD(\bv_\fh))_{\Omega_f} -2\mu\langle \bD(\bu_\fh) \bn_f, \bv_\fh\rangle_{\Gamma_f^D}-2\mu\langle \bu_\fh, \bD(\bv_\fh) \bn_f\rangle_{\Gamma_f^D} + \frac{\gamma}{h}\langle \bu_\fh, \bv_\fh\rangle_{\Gamma_f^D}\nonumber\\ 
  &\qquad-(p_\fh,\nabla\cdot \bv_\fh)_{\Omega_f}
  +\langle p_\fh, \bv_\fh\cdot \bn_f\rangle_{\Gamma_f^D}
  -\langle \bsigma_\fh\bn_f, \bv_\fh\rangle_{\Gamma_\fp}
  \nonumber\\
  &\quad =(\bF_\fh, \bv_\fh)_{\Omega_f} -2\mu\langle \bg_f^D, \bD(\bv_\fh) \bn_f\rangle_{\Gamma_f^D} +\frac{\gamma}{h}\langle \bg_f^D, \bv_\fh\rangle_{\Gamma_f^D}
  +\langle \bg_f^N, \bv_\fh\rangle_{\Gamma_f^N},\label{e3.2a}
    \\
    &-(\nabla\cdot \bu_\fh, w_\fh)_{\Omega_f} + \langle \bu_\fh\cdot \bn_f, w_\fh\rangle_{\Gamma_f^D}=- (q_f, w_\fh)_{\Omega_f} + \langle \bg_f^D\cdot \bn_f, w_\fh\rangle_{\Gamma_f^D},\label{e3.2b}
\end{align}
\end{subequations}
where $\gamma> 0$ is the Nitsche penalty parameter. We will later see that we need $\gamma$ large enough for the well posedness of the global problem and also for the Stokes subproblems. Note that the term $\langle \bsigma_\fh\bn_f, \bv_\fh\rangle_{\Gamma_\fp}$ in \eqref{e3.2a} will be used for the coupling with the Biot system.
\subsection{Biot variational formulation}
In the poroelastic region $\Omega_p$. Let
\[
\begin{alignedat}{2}
\bV_p&:=\{\bv_p\in\bH(\mathrm{div};\Omega_p):\bv_p\cdot\bn_p = 0 \text{ on }\Gamma_p^{N_v}\},\quad
& \rW_p &:= \rL^2(\Omega_p),\\
\mathbb{X}_p &:= \{\btau_p\in\mathbb{H}(\mathrm{div};\Omega_p)
 : \btau_p\bn_p = \boldsymbol{0} \ \text{on } \Gamma_p^{N_s}\},\quad
& \mathbb{Q}_p &= \mathbb{L}^2(\Omega_p,\mathbb{N}),\\
\bV_d &= \bL^2(\Omega_p),\quad
& && 
\end{alignedat}
\]
where $\mathbb{N}$ denotes the space of $d\times d$ skew-symmetric matrices and
\begin{align*}
    \bH(\mathrm{div};\Omega)=\{\bv\in\rL^2(\Omega,\mathbb{R}^d)\ :\ \mathrm{div}\ \bv\in\rL^2(\Omega) \},\quad \mathbb{H}(\mathrm{div};\Omega)=\{\btau\in\rL^2(\Omega,\mathbb{M})\ :\ \mathrm{div}\ \btau\in\rL^2(\Omega,\mathbb{R}^d) \},
\end{align*}
equipped with the norms
\begin{align*}
    \|\bv\|_{\bH(\mathrm{div};\Omega)} = \left(\|\bv\|^2_{\Omega}+\|\mathrm{div}\ \bv\|^2_{\Omega}\right)^{1/2},\quad \|\btau\|_{\mathbb{H}(\mathrm{div};\Omega)} = \left(\|\btau\|^2_{\Omega}+\|\mathrm{div}\ \btau\|^2_{\Omega}\right)^{1/2}.
\end{align*}
Let $\bV_\ph\times\rW_\ph\subset\bV_p\times\rW_p$ be any stable Darcy mixed finite element pair satisfying the inf-sup condition
\begin{align}\label{e3.3}
    C\sup_{\bv_p\in \bV_\ph\setminus \{\boldsymbol{0}\}} \frac{(w_\ph, \nabla\cdot \bv_\ph)_{\Omega_p}}{\|\bv_\ph\|_{\bH(\mathrm{div};\Omega_p)}}\geq \|w_\ph\|_{\Omega_p},\quad \forall\ w_\ph\in \rW_\ph.
\end{align}
Let $\mathbb{X}_\ph\times\bV_\dh\times\mathbb{Q}_\ph\subset\mathbb{X}_p\times\bV_d\times\mathbb{Q}_p$ be any stable finite element triple for mixed elasticity with weak stress symmetry satisfying the inf-sup condition
\begin{align}\label{e3.4}
\begin{split}
     C\sup_{\btau_\ph\in \mathbb{X}_\ph^{0,\Gamma_\fp}\setminus \{0\}} \frac{(\bxi_\ph,\nabla\cdot\btau_\ph)_{\Omega_p}+(\bchi_\ph,\btau_\ph)_{\Omega_p}}{\|\btau_\ph\|_{\mathbb{H}(\mathrm{div};\Omega_p)}}&\geq  \|\bxi_\ph\|_{\Omega_p}+\|\bchi_\ph\|_{\Omega_p},\quad\forall\ \bxi_\ph\in\bV_\dh,\ \bchi_\ph\in\mathbb{Q}_\ph,
     \end{split}
\end{align}
where
\begin{align*}
    \mathbb{X}_\ph^{0,\Gamma_\fp}=\{\btau_\ph\in\mathbb{X}_\ph\ :\ \btau_\ph\bn_p = 0\ \text{on }\Gamma_\fp\}.
\end{align*}
We note that $\btau_\ph$ in \eqref{e3.4} can be restricted to $\mathbb{X}_\ph^{0,\Gamma_\fp}$, since $|\Gamma_p^{D_d}|>0$. We next note that 
\begin{align*}
    \nabla\cdot\disp_p = \text{tr}(\bD(\disp_p))=\text{tr}(A\bsigma_e)=\text{tr}(A(\bsigma_p+\alpha p_p\bI)),
\end{align*}
where $\text{tr}$ is the trace operator.
From the stress-strain relation in \eqref{e2.3} we have
\begin{align}\label{e3.5}
    A(\bsigma_p+\alpha p_p\bI) = \bD(\disp_p) = \nabla\disp_p - \bgamma_p,
\end{align}
where $\bgamma_p = \frac{1}{2}(\nabla\disp_p - \nabla\disp_p^\mathrm{T})$. We test equation \eqref{e3.5} with $\btau_\ph\in \mathbb{X}_\ph$ and the second equation in \eqref{e2.2a} with $\bv_\ph\in\bV_\ph$ and integrate by parts. In \eqref{e2.2a} we test the first equation with $\bxi_\ph\in \mathbb{Q}_\ph$ and the third equation with $w_\ph\in\rW_\ph$. Using the Lagrange multiplier $\bchi_\ph\in\mathbb{Q}_\ph$, we weakly impose the symmetry of $\bsigma_\ph\in\mathbb{X}_\ph$. We obtain that 
$(\bsigma_\ph,\disp_\ph,\bgamma_\ph,\bu_\ph,p_\ph)\in \mathbb{X}_\ph\times\bV_{dh}\times\mathbb{Q}_\ph\times\bV_\ph\times\rW_\ph$ satisfy, 
for all $(\btau_\ph,\bxi_\ph,\bchi_\ph,\bv_\ph,w_\ph)\in \mathbb{X}_\ph\times\bV_{dh}\times\mathbb{Q}_\ph\times\bV_\ph\times\rW_\ph$,
\begin{subequations}\label{e3.6}
\begin{align}
  &(A(\bsigma_\ph + \alpha p_\ph\bI),\btau_\ph)_{\Omega_p}+ (\disp_\ph, \nabla\cdot\btau_\ph)_{\Omega_p}+(\bgamma_\ph, \btau_\ph)_{\Omega_p}
  -\langle \disp_\ph, \btau_\ph\bn_p\rangle_{\Gamma_\fp}
  =\langle \bg_p^{D_d}, \btau_\ph\bn_p\rangle_{\Gamma_p^{D_d}},\label{e3.6b}\\
    &(\nabla\cdot\bsigma_\ph,\bxi_\ph)_{\Omega_p}=-(\bF_p,\bxi_\ph)_{\Omega_p},\label{e3.6a}\\
    &(\bsigma_\ph,\bchi_\ph)_{\Omega_p} = 0,\label{e3.6c}\\
  &(\mu\bK^{-1}\bu_\ph,\bv_\ph)_{\Omega_p}-(p_\ph,\nabla\cdot\bv_\ph)_{\Omega_p}
  +\langle p_\ph, \bv_\ph\cdot \bn_p\rangle_{\Gamma_{fp}}
  =-\langle g_p^{D_p},\bv_\ph\cdot\bn_p\rangle_{\Gamma_p^{D_p}},\label{e3.6d}\\
    &(s_0 p_\ph,w_\ph)_{\Omega_p} + \alpha (A(\bsigma_\ph+\alpha p_\ph\bI),w_\ph\bI)_{\Omega_p} +(\nabla \cdot \bu_\ph,w_\ph)_{\Omega_p} = (q_\ph,w_\ph)_{\Omega_p}.\label{e3.6e}
\end{align}    
\end{subequations}
Note that the terms $\langle \disp_\ph, \btau_\ph\bn_p\rangle_{\Gamma_\fp}$ in \eqref{e3.6a} and $\langle p_\ph, \bv_\ph\cdot \bn_p\rangle_{\Gamma_{fp}}$ in \eqref{e3.6d} will be used for the coupling with the Stokes problem.

\subsection{Stokes-Biot mixed finite element method}
We introduce two Lagrange multipliers on the Stokes-Biot interface:
\begin{equation}\label{LM}
    \lambda^p_h = p_\ph,\quad \blambda^d_h = \disp_\ph\quad \text{on }\Gamma_{\fp}.
\end{equation}
The first one is standard in Stokes–Darcy and Stokes–Biot models with a mixed Darcy formulation and it is used to impose weakly continuity of flux. The second one is needed in the mixed elasticity formulation, since the trace of $\disp_p$ on $\Gamma_\fp$ is not well defined for $\disp_p\in\bL^2(\Omega_p)$. It is used to impose weakly the continuity of normal stress condition $\bsigma_f\bn_f + \bsigma_p\bn_p=0$ and the BJS condition, \textit{cf}. \eqref{e2.4b}. For the spaces we choose non-conforming approximations:
\begin{align}\label{e3.7}
\begin{split}
  &\Lambda^p_h := \bV_\ph\cdot\bn_p|_{\Gamma_\fp},\quad \bLambda^d_h := \mathbb{X}_\ph\bn_p|_{\Gamma_\fp}, \quad \text{with norms} \quad
  \|\bmu^d_h\|_{ \bLambda^d_h}=\|\bmu^d_h\|_{\Gamma_\fp}, \quad
  \|\mu^p_h\|_{ \Lambda^p_h}=\|\mu^p_h\|_{\Gamma_\fp}.
    \end{split}
\end{align}
Using \eqref{LM} and the second equalities in \eqref{e2.4a} and \eqref{e2.4b}, the term $-\langle \bsigma_\fh\bn_f, \bv_\fh\rangle_{\Gamma_\fp}$ in \eqref{e3.2a} can be rewritten as
\begin{equation*}
  \begin{split}
    -\langle \bsigma_\fh\bn_f, \bv_\fh\rangle_{\Gamma_\fp} & =
    -\langle \bsigma_\fh\bn_f\cdot\bn_f, \bv_\fh\cdot\bn_f\rangle_{\Gamma_\fp}
    -\sum_{l=1}^{d-1}\langle \bsigma_\fh\bn_f\cdot\bt_{f,l}, \bv_\fh\cdot\bt_{f,l}\rangle_{\Gamma_\fp} \\
    & = \<\lambda^p_h,\bv_\fh\cdot\bn_f\>_{\Gamma_\fp}
    + \sum_{l=1}^{d-1}\Big\<\mu\alpha_{\bjs}\sqrt{\bK^{-1}_l}\big(\bu_f-\blambda_h^d\big)\cdot \bt_{f,l},\bv_\fh\cdot\bt_{f,l}\Big\>_{\Gamma_{fp}}.
  \end{split}
\end{equation*}
The variational formulation for the Stokes-Biot problem reads: Find $(\bu_\fh,p_\fh,\bsigma_\ph,\disp_\ph,\bgamma_\ph,\bu_\ph,p_\ph,\blambda^d_h,\lambda^p_h)\in\bV_\fh\times\rW_\fh\times\mathbb{X}_\ph\times\bV_\dh\times\mathbb{Q}_\ph\times\bV_\ph\times\rW_\ph\times\bLambda^d_h\times\Lambda^p_h$ such that for all $(\bv_\fh,w_\fh,\btau_\ph,\bxi_\ph,\bchi_\ph,\bv_\ph,w_\ph,\bmu^d_h,\mu^p_h)\in\bV_\fh\times\rW_\fh\times\mathbb{X}_\ph\times\bV_d\times\mathbb{Q}_\ph\times\bV_\ph\times\rW_\ph\times\bLambda^d_h\times\Lambda^p_h$, 
\begin{subequations}\label{e3.8}
\begin{align}    &2\mu(\bD(\bu_\fh),\bD(\bv_\fh))_{\Omega_f} -2\mu\langle \bD(\bu_\fh) \bn_f, \bv_\fh\rangle_{\Gamma_f^D} -2\mu\langle \bu_\fh, \bD(\bv_\fh) \bn_f\rangle_{\Gamma_f^D}+ \frac{\gamma}{h}\langle \bu_\fh, \bv_\fh\rangle_{\Gamma_f^D}\nonumber\\
  &\quad +\sum_{l=1}^{d-1}\Big\<\mu\alpha_{\bjs}\sqrt{\bK^{-1}_l}(\bu_\fh - \blambda_h^d)\cdot\bt_{f,l},\bv_\fh\cdot\bt_{f,l}\Big\>_{\Gamma_\fp} -(p_\fh,\nabla\cdot \bv_\fh)_{\Omega_f}+\langle p_\fh, \bv_\fh\cdot \bn_f\rangle_{\Gamma_f^D}
  \nonumber\\
  & \quad + \langle \lambda^p_h,\bv_\fh\cdot\bn_f\rangle_{\Gamma_\fp}
  =(\bF_f, \bv_\fh)_{\Omega_f}-2\mu\langle \bg_f^D, \bD(\bv_\fh) \bn_f\rangle_{\Gamma_f^D}+\frac{\gamma}{h}\langle \bg_f^D, \bv_\fh\rangle_{\Gamma_f^D}+\langle \bg_f^N, \bv_\fh\rangle_{\Gamma_f^N}, \label{e3.8a}\\
    -&(\nabla\cdot \bu_f, w_\fh)_{\Omega_f} + \langle \bu_\fh\cdot \bn_f, w_\fh\rangle_{\Gamma_f^D}= -(q_f, w_\fh)_{\Omega_f} + \langle \bg_f^D\cdot \bn_f, w_\fh\rangle_{\Gamma_f^D},\label{e3.8b}\\
    &(A(\bsigma_\ph + \alpha p_\ph\bI),\btau_\ph)_{\Omega_p}+ (\disp_\ph, \nabla\cdot\btau_\ph)_{\Omega_p}+(\bgamma_\ph, \btau_\ph)_{\Omega_p}
    -\langle \blambda^d_h, \btau_\ph\bn_p\rangle_{\Gamma_\fp}
    =\langle \bg_p^{D_d}, \btau_\ph\bn_p\rangle_{\Gamma_p^{D_d}},\label{e3.8c}\\
    &(\nabla\cdot\bsigma_\ph,\bxi_\ph)_{\Omega_p}=-(\bF_p,\bxi_\ph)_{\Omega_p},\label{e3.8d}\\
    &(\bsigma_\ph,\bchi_\ph)_{\Omega_p} = 0,\label{e3.8e}\\ &(\mu\bK^{-1}\bu_\ph,\bv_\ph)_{\Omega_p}-(p_\ph,\nabla\cdot\bv_\ph)_{\Omega_p}
    +\langle \lambda^p_h, \bv_\ph\cdot \bn_p\rangle_{\Gamma_\fp}
    =-\langle g_p^{D_p},\bv_\ph\cdot\bn_p\rangle_{\Gamma_p^{D_p}},\label{e3.8f}\\
    &(s_0 p_\ph,w_\ph)_{\Omega_p} + \alpha (A(\bsigma_\ph+\alpha p_\ph\bI),w_\ph\bI)_{\Omega_p} +(\nabla \cdot \bu_\ph,w_\ph)_{\Omega_p} = (q_p,w_\ph)_{\Omega_p},\label{e3.8g}\\
    &\langle\bu_\fh\cdot\bn_f +\blambda^d_h\cdot\bn_p
    +\bu_\ph\cdot\bn_p,\mu^p_h\rangle_{\Gamma_\fp}=0,\label{e3.8h}\\
    &\langle\lambda^p_h,\bmu^d_h\cdot\bn_p\rangle_{\Gamma_\fp}
    -\sum_{l=1}^{d-1}\Big\<\mu\alpha_{\bjs}\sqrt{\bK^{-1}_l}
(\bu_\fh-\blambda^d_h)\cdot\bt_{f,l},\bmu^d_h\cdot\bt_{f,l}\Big\>_{\Gamma_\fp}+\langle\bsigma_\ph\bn_p,\bmu^d_h\rangle_{\Gamma_\fp}=0.\label{e3.8i}
\end{align}
\end{subequations}
Here \eqref{e3.8h} enforces weakly the conservation of mass across the interface $\Gamma_\fp$ -- the first equation in \eqref{e2.4a}, while \eqref{e3.8i} enforces weakly the conservation of momentum -- the first equation in \eqref{e2.4b}. The second equations in \eqref{e2.4a} and \eqref{e2.4b} have been implicitly incorporated in \eqref{e3.8i}.

To simplify the notation and make the structure of the system more evident, we introduce the bilinear forms
\begin{subequations}\label{e3.9}
\begin{align}
    &a_f(\bu_\fh,\bv_\fh) =  2\mu(\bD(\bu_\fh),\bD(\bv_\fh))_{\Omega_f} -2\mu\langle \bD(\bu_\fh)\cdot \bn_f, \bv_\fh\rangle_{\Gamma_f^D}\nonumber\\
  &\hspace{3cm}-2\mu\langle \bu_\fh, \bD(\bv_\fh)\cdot \bn_f\rangle_{\Gamma_f^D}+ \frac{\gamma}{h}\langle \bu_\fh, \bv_\fh\rangle_{\Gamma_f^D},\label{e3.9a}\\ &a_{\bjs}(\bu_\fh,\blambda^d_h;\bv_\fh,\bmu^d_h)
  =\sum_{l=1}^{d-1}\Big\<\mu\alpha_{\bjs}\sqrt{\bK^{-1}_l} (\bu_\fh-\blambda_h^d)\cdot\bt_{f,l},(\bv_\fh-\bmu_h^d)\cdot\bt_{f,l}\Big\>_{\Gamma_\fp}, \\
  &b_f(\bv_\fh,w_\fh)=-(\nabla\cdot \bv_f, w_\fh)_{\Omega_f}
  + \langle \bv_\fh\cdot \bn_f, w_\fh\rangle_{\Gamma_f^D}, \\
  &a_e(\bsigma_\ph,p_\ph;\btau_\ph,w_\ph)= (A(\bsigma_\ph+\alpha p_\ph\bI),\btau_\ph+\alpha w_\ph\bI)_{\Omega_p}, \\
&   a_p^p(p_\ph,w_\ph) = (s_0p_\ph,w_\ph)_{\Omega_p}, \quad a_p(\bu_\ph,\bv_\ph)=(\mu\bK^{-1}\bu_\ph,\bv_\ph)_{\Omega_p}, \\ 
  &b_p(\bv_\ph,w_\ph)=(\nabla\cdot\bv_\ph,w_\ph)_{\Omega_p},\quad b_e(\btau_\ph;\bxi_\ph,\bchi_\ph)=(\nabla\cdot\btau_\ph,\bxi_\ph)_{\Omega_p}
  +(\btau_\ph,\bchi_\ph)_{\Omega_p}, \\    &b_{\Gamma_1}(\bu_\fh,\bu_\ph,\blambda^d_h;\mu^p_h)=\langle\bu_\fh\cdot\bn_f+\blambda^d_h\cdot\bn_p+\bu_\ph\cdot\bn_p,\mu^p_h\rangle_{\Gamma_\fp}, \\
  \begin{split}
    &b_{\Gamma_2}(\bsigma_\ph,\bmu^d_h)
    =    \langle\bsigma_\ph\bn_p,\bmu^d_h\rangle_{\Gamma_\fp},
    \end{split}
\end{align}
\end{subequations}
and the functionals
\begin{align*}
    &f_f(\bv_\fh) = (\bF_\fh, \bv_\fh)_{\Omega_f}-2\mu\langle \bg_f^D, \bD(\bv_\fh)\cdot \bn_f\rangle_{\Gamma_f^D}+\frac{\gamma}{h}\langle \bg_f^D, \bv_\fh\rangle_{\Gamma_f^D}+\langle \bg_f^N, \bv_\fh\rangle_{\Gamma_f^N},\\
    &q_f(w_\fh) = -(q_f,w_\fh)_{\Omega_f}+\langle \bg_f^D\cdot\bn_f, w_\fh\rangle_{\Gamma_f^D},\quad g_s(\btau_\ph) = \langle \bg_p^{D_p},\btau_\ph\bn_p\rangle_{\Gamma_p^{D_d}},\\
    &g_p(\bv_\ph) = -\langle g_p^{D_p},\bv_\ph\cdot\bn_p\rangle_{\Gamma_p^{D_p}}, \quad f_p(\bxi_\ph) = -(\bF_p,\bxi_\ph)_{\Omega_p}, \quad q_p(w_\ph) = (q_p,w_\ph)_{\Omega_p}.
\end{align*}
We can now rewrite the above variational formulation as follows: Find $(\bu_\fh,p_\fh,\bsigma_\ph,\disp_\ph,\bgamma_\ph,\bu_\ph,p_\ph,\blambda^d_h,\lambda^p_h)\in\bV_\fh\times\rW_\fh\times\mathbb{X}_\ph\times\bV_d\times\mathbb{Q}_\ph\times\bV_\ph\times\rW_\ph\times\bLambda^d_h\times\Lambda^p_h$ such that for all $(\bv_\fh,w_\fh,\btau_\ph,\bxi_\ph,\bchi_\ph,\bv_\ph,w_\ph,\bmu^d_h,\mu^p_h)\in\bV_\fh\times\rW_\fh\times\mathbb{X}_\ph\times\bV_\dh\times\mathbb{Q}_\ph\times\bV_\ph\times\rW_\ph\times\bLambda^d_h\times\Lambda^p_h$,
\begin{subequations}\label{e3.10}
\begin{flalign}
  &a_f(\bu_\fh,\bv_\fh) + a_{\bjs}(\bu_\fh,\blambda^d_h;\bv_\fh,\mathbf{0}) + b_f(\bv_\fh,p_\fh)
  + b_{\Gamma_1}(\bv_\fh,\mathbf{0},\mathbf{0};\lambda^p_h)
  = f_f(\bv_\fh) 
  \label{e3.10a}&&\\
  &b_f(\bu_\fh,w_\fh)=q_f(w_\fh),\label{e3.10b}&&\\   &a_e(\bsigma_\ph,p_\ph;\btau_\ph,0)+b_e(\btau_\ph;\disp_\ph,\bgamma_\ph)
  - b_{\Gamma_2}(\btau_\ph,\blambda^d_h)
  =g_s(\btau_\ph),\label{e3.10c}&&\\
    &b_e(\bsigma_\ph;\bxi_\ph,\bchi_\ph)_{\Omega_p} = f_p(\bxi_\ph),\label{e3.10d}&&\\
  &a_p(\bu_\ph,\bv_\ph)-b_p(\bv_\ph,p_\ph) + b_{\Gamma_1}(\mathbf{0},\bv_\ph,\mathbf{0};\lambda^p_h)
  =g_p(\bv_\ph),\label{e3.10e}&&\\
    &a_p^p(p_\ph,w_\ph) + a_e(\bsigma_\ph,p_\ph;\mathbf{0},w_\ph)+b_p(\bu_\ph,w_\ph) = q_p(w_\ph),\label{e3.10f}&&\\
&b_{\Gamma_1}(\bu_\fh,\bu_\ph,\blambda^d_h;\mu^p_h)=0,\label{e3.10g}&&\\
  &a_{\bjs}(\bu_\fh,\blambda^d_h;\mathbf{0},\blambda_h^d) + b_{\Gamma_1}(\mathbf{0},\mathbf{0},\bmu_h^d;\lambda^p_h)
+ b_{\Gamma_2}(\bsigma_\ph,\bmu^d_h)=0.\label{e3.10h}&&
\end{flalign}
\end{subequations}

We next show that the system \eqref{e3.10} has a solution. Recall Korn's inequality \cite{Toselli-Widlund}
\begin{align}\label{e3.11}
    \|\bD(\bv_\fh)\|_{\Omega_f}^2+\frac{1}{h}\|\bv_\fh\|_{\Gamma_f^D}^2\geq C_{K}\|\bv_\fh\|_{1,\Omega_f}^2\quad \forall\ \bv_\fh\in\bH^1(\Omega_f),
\end{align}
where $C_{K}$ depends on the shape of $\Omega_f$ and dimension $d$. We also recall the inverse estimate \cite{ciarlet}, for all $e \in \mathcal{T}_{h_f}^f|_{\Gamma_f^D}$, $e \subset \partial E$,
\begin{align}\label{e3.12}
    h\|\varphi\|_{e}^2\leq C_{I}\|\varphi\|_{E}^2,\quad \forall\text{ polynomial valued }\varphi \text{ on } E,
\end{align}
where $C_I$ depends on the shape regularity of the mesh $\mathcal{T}_{h_f}^f$ and the quasiuniformity of its trace on $\Gamma_f^D$. We next show that $a_f$ is coercive.

\begin{lemma}\label{lemma3.1}
  Assume that $|\Gamma_f^D| > 0$ and $\gamma > 2\mu C_I^2$. Then there exists a constant $C_f$ that depends on $C_I$, $\gamma$, and $\mu$, such that
    \begin{align}\label{e3.13}
    a_f(\bu_\fh,\bu_\fh)\geq C_f\|\bu_\fh\|_{1,\Omega_f}.
\end{align}
\end{lemma}
\begin{proof}
Recalling the definition of $a_f(\cdot,\cdot)$, cf. \eqref{e3.9a}, using the inverse estimate \eqref{e3.12}, and the Cauchy-Schwarz inequality, we obtain
\begin{align*}
    a_f(\bu_\fh,\bu_\fh)&\geq 2\mu\|\bD(\bu_\fh)\|_{\Omega_f}^2-\frac{4\mu C_I}{\sqrt{h}}\| \bD(\bu_\fh)\|_{\Omega_f} \|\bu_\fh\|_{\Gamma_f^D}+ \frac{\gamma}{h}\|\bu_\fh\|_{\Gamma_f^D}^2\\
    &\geq 2\mu\left( 1-\frac{4\mu}{\delta}\right)\|\bD(\bu_\fh)\|_{\Omega_f}^2+\frac{1}{h}\left(\gamma-\frac{\delta C_I^2}{2}\right)\|\bu_\fh\|_{\Gamma_f^D}^2,
\end{align*}
where in the second step, we used Young's inequality 
\begin{align*}
    ab \le \frac{a^2}{2\delta}+\frac{\delta b^2}{2}.
\end{align*}
Now, since $\gamma > 2\mu C_I^2$ we can choose $\delta$ such that $\frac{\gamma}{C_I^2}>\frac{\delta}{2} > 2\mu$ and using Korn's inequality \eqref{e3.11}, we obtain \eqref{e3.13}.
\end{proof}

\begin{lemma}
 Under the assumptions of Lemma~\ref{lemma3.1}, there exists a unique solution to the problem \eqref{e3.10}.
\end{lemma}
\begin{proof}
This is a finite dimensional linear problem. To show well posedness it is enough to show that the solution is unique. We consider the homogeneous system \eqref{e3.10},
take $\bv_\fh = \bu_\fh,\ w_\fh = p_\fh,\ \btau_{ph}=\bsigma_{ph},\ \bxi_{ph}=\disp_{ph},\ \bchi_{ph}=\gamma_{ph},\ \bv_{ph}=\bu_{ph},\ w_{ph}=p_{ph},\ \bmu^d_h = \blambda_h^d,$ and $\mu_h^p = \lambda_h^p$, and combine the equations, obtaining
\begin{align*}
\begin{split}
&a_f(\bu_\fh,\bu_\fh)+a_e(\bsigma_\ph,p_\ph;\bsigma_\ph,p_\ph)+a_p(\bu_\ph,\bu_\ph)+a_p^p(p_\ph,p_\ph)+a_{\bjs}(\bu_\fh,\blambda_h^d;\bu_\fh,\blambda_h^d)=0.
    \end{split}
\end{align*}
Using the coercivity of $a_f$, Lemma \ref{lemma3.1}, and the definitions of the bilinear forms \eqref{e3.9}, we obtain
\begin{align*}
    &C_f\|\bu_\fh\|^2_{1,\Omega_f}+\|A^{1/2}(\bsigma_{ph}+\alpha p_{ph}\bI)\|_{\Omega_p}^2+\|\bK^{-1/2}\bu_{ph}\|_{\Omega_p}^2+\|s_0^{1/2}p_{ph}\|_{\Omega_p}^2\\
  &\qquad +\sum_{l=1}^{d-1}\mu\alpha_{\bjs}\sqrt{\bK^{-1}_l}
  \|(\bu_\fh - \blambda^d_h)\cdot\bt_{f,l}\|_{\Gamma_\fp}^2 \ge 0.
\end{align*}
It follows that $\bu_\fh = \mathbf{0},\ \bsigma_\ph = \mathbf{0},\ \bu_\ph=\mathbf{0},\ p_\ph = 0$, and $(\bu_\fhi-\blambda^d_h)\cdot\bt_{f,l} = 0$ for $l = 1,\cdots,d-1$. Using this in the homogeneous equations \eqref{e3.10a}, \eqref{e3.10c}, and \eqref{e3.10e} yields
\begin{align}
    &b_f(\bv_\fh,p_\fh)+\langle \lambda^p_h,\bv_\fh\cdot\bn_f\rangle_{\Gamma_\fp} = 0,\quad \forall\ \bv_\fh\in\bV_\fh,\label{e3.15}\\
    &b_e(\btau_\ph;\disp_\ph,\bgamma_\ph)-\langle \blambda^d_h, \btau_\ph\bn_p\rangle_{\Gamma_\fp} = 0,\quad \forall\ \btau_\ph\in\mathbb{X}_\ph,\label{e3.16}\\
    &\langle \lambda^p_h, \bv_\ph\cdot \bn_p\rangle_{\Gamma_\fp} = 0,\quad \forall\ \bv_\ph\in \bV_\ph.\label{e3.17}
\end{align}
Since $\lambda_h^p\in\Lambda_h^p = \bV_\ph\cdot\bn_p|_{\Gamma_\fp}$, we can take $\bv_\ph\in\bV_\ph$ such that $\bv_\ph\cdot\bn_p|_{\Gamma_\fp} = \lambda_h^p$ in \eqref{e3.17}, implying that $\lambda_h^p = 0$. Next, for $\bv_\fh\in\bV^{0,\Gamma_f^D}_\fh$, it holds that
$b_f(\bv_\fh,w_\fh) = -(\nabla\cdot\bv_\fh,w_\fh)_{\Omega_f}$. 
Using the Stokes inf-sup condition \eqref{e3.1} and \eqref{e3.15}, we obtain 
\begin{align*}
    \|p_\fh\|_{\Omega_f}&\leq  C\sup_{\bv_\fh\in \bV^{0,\Gamma_f^D}_\fh\setminus \{\boldsymbol{0}\}} \frac{b_f(\bv_\fh, p_\fh)_{\Omega_f}}{\|\bv_\fh\|_{1,\Omega_f}}\leq C\sup_{\bv_\fh\in \bV^{0,\Gamma_f^D}_\fh\setminus \{\boldsymbol{0}\}} \frac{\langle \lambda^p_h,\bv_\fh\cdot\bn_f\rangle_{\Gamma_\fp}}{\|\bv_\fh\|_{1,\Omega_f}} = 0,
\end{align*}
hence $p_\fh = 0$.
Next, using the inf-sup condition \eqref{e3.4} and \eqref{e3.16} yields $\disp_{ph} = 0$ and $\bgamma_{\ph} = 0$. Finally, taking $\btau_{ph}\bn_p = \blambda_h^d$ in \eqref{e3.16}, we get $\blambda_h^d = 0$.
\end{proof}

\section{Non-overlapping domain decomposition}\label{sec:dd}
We now decompose $\Omega_f$ and $\Omega_p$ into $N_f$ and $N_p$ non-overlapping Lipschitz polyhedral subdomains, respectively, aligned with the partitions $\mathcal{T}_{h_f}^f$ and $\mathcal{T}_{h_p}^p$:
\begin{align*}
    \Omega_f = \bigcup_{i = 1}^{N_f}\Omega_i,\quad \Omega_p = \bigcup_{i = N_f+1}^{N}\Omega_i,\quad N = N_f+N_p.
\end{align*}
For $1\leq i\leq N$, let $\bn_i$ be the outward unit normal vector to subdomain $\Omega_i$. The exterior boundary of $\Omega_i$, possibly with zero measure, is denoted by $\Gamma_i$:
\begin{align*}
    \Gamma_i = \partial \Omega_i\cap\partial\Omega,\quad 1\leq i\leq N.
\end{align*}
The corresponding Dirichlet and Neumann parts of the exterior boundary for $\Gamma_i$ in the fluid region are denoted by 
\begin{align*}
    \Gamma_i = \Gamma_i^D\cup\Gamma_i^N,\quad \Gamma_i^D = \partial \Omega_i \cap \Gamma_f^D,\quad \Gamma_i^N = \partial\Omega_i\cap \Gamma_f^N,\quad 1\leq i\leq N_f.
\end{align*}
To ensure that the Stokes subdomain problems are well posed, we assume that $|\Gamma_i^D|>0$ for $1\leq i\leq N_f$; see Remark \ref{r2} for the general case.

Similarly, for the poroelastic region we have for the Darcy and the elasticity equations
\begin{align*}
    \Gamma_i &= \Gamma_i^{D_p}\cup\Gamma_i^{N_v},\quad \Gamma_i^{D_p} = \partial \Omega_i \cap \Gamma_p^{D_p},\quad \Gamma_i^{N_v} = \partial\Omega_i\cap \Gamma_p^{N_v},\quad N_f+1\leq i\leq N,\\
    \Gamma_i &= \Gamma_i^{D_d}\cup\Gamma_i^{N_s},\quad \Gamma_i^{D_d} = \partial \Omega_i \cap \Gamma_p^{D_d},\quad \Gamma_i^{N_s} = \partial\Omega_i\cap \Gamma_p^{N_s},\quad N_f+1\leq i\leq N.
\end{align*}
Let $\Gamma_{\ij}$ be the interfaces between the subdomains, again possibly with measure zero:
\begin{align*}
    \Gamma_{\ij} = \partial\Omega_i \cap \partial\Omega_j,\quad 1\leq i<j\leq N.
\end{align*}
We introduce the following notation to represent the union of the interfaces between the subdomains of the same type:
\begin{align*}
    \Gamma_{\ff} = \bigcup_{1\leq i< j\leq N_f}\Gamma_\ij,\quad \Gamma_{\pp} = \bigcup_{N_f+1\leq i< j\leq N}\Gamma_\ij;\quad\text{hence } \Gamma_{\fp}=\left(\bigcup_{1\leq i<j\leq N}\Gamma_{i,j}\right)\setminus(\Gamma_{\ff}\cup\Gamma_{\pp}).
\end{align*}
We also denote the type of interface for each subdomain $1\leq i\leq N$ as
\begin{align*}
    \Gamma_\ff^i = \partial\Omega_i\cap\Gamma_\ff,\quad \Gamma_\pp^i = \partial\Omega_i\cap\Gamma_\pp,\quad \Gamma_\fp^i = \partial\Omega_i\cap\Gamma_\fp.
\end{align*}

\subsection{Domain decomposition variational formulation}
For all $1\leq i\leq N_f$, we define 
\begin{align*}
    \bV_\fhi=\bV_\fh|_{\Omega_i},\quad \rW_\fhi=\rW_{fh}|_{\Omega_i}.
\end{align*} 
The pair $\bV_\fhi\times\rW_\fhi$ satisfies the inf-sup condition
\begin{align}\label{e4.1}
    C\sup_{\bv_\fhi\in \bV^{0,\gamma_i}_\fhi\setminus \{\boldsymbol{0}\}} \frac{(w_\fhi, \nabla\cdot \bv_\fhi)_{\Omega_i}}{\|\bv_\fhi\|_{1,\Omega_i}}\geq \|w_\fhi\|_{\Omega_i},\quad \forall\ w_\fhi\in \rW_\fhi,
\end{align}
where $\gamma_i\subset\partial\Omega_i$, $\gamma_i\neq \partial\Omega_i$, and 
\begin{align*}
    \bV^{0,\gamma_i}_\fhi=\{\bv_\fhi\in\bV_\fhi\ :\ \bv_\fhi = \boldsymbol{0}\quad \text{on }\gamma_i\}.
\end{align*}
In \eqref{e4.1} the constant $C$ is taken to be the maximum of the local inf-sup constants.

For all $N_f+1\leq i\leq N$, we define
\begin{align*}
    \mathbb{X}_{ph,i} = \mathbb{X}_{ph}|_{\Omega_i},\ \bV_{dh,i} = \bV_{dh}|_{\Omega_i},\ \mathbb{Q}_{ph,i}=\mathbb{Q}_{ph}|_{\Omega_i},\ \bV_{ph,i}=\bV_{ph}|_{\Omega_i},\ \rW_{ph,i}=\rW_{ph}|_{\Omega_i}.
\end{align*}
The spaces $\bV_{ph,i}\times\rW_{ph,i}$ satisfy the inf-sup condition
\begin{align}\label{e4.2}
     C\sup_{\bv_{ph,i}\in \bV_{ph,i}\setminus \{\boldsymbol{0}\}} \frac{(w_{ph,i}, \nabla\cdot \bv_{ph,i})_{\Omega_i}}{\|\bv_{ph,i}\|_{\bH(\mathrm{div};\Omega_i)}}\geq \|w_{ph,i}\|_{\Omega_p},\quad \forall\ w_{ph,i}\in \rW_{ph,i}.
\end{align}
and the spaces $\mathbf{X}_{ph,i}\times\bV_{dh,i}\times\mathbb{Q}_{ph,i}$ satisfy the inf-sup condition
\begin{align}\label{e4.3}
\begin{split}
    C\sup_{\btau_{ph,i}\in \mathbb{X}^{0,\gamma_i}_{ph,i}\setminus \{0\}} \frac{(\bxi_{ph,i},\nabla\cdot\btau_{ph,i})_{\Omega_i}+(\bchi_{ph,i},\btau_{ph,i})_{\Omega_i}}{\|\btau_{ph,i}\|_{\bH(\mathrm{div};\Omega_i)}}&\geq  \|\bxi_{ph,i}\|_{\Omega_i}+\|\bchi_{ph,i}\|_{\Omega_i},\quad \forall\ \bxi_{ph,i}\in\bV_{dh,i},\ \bchi_\ph\in\mathbb{Q}_{ph,i},
    \end{split}
\end{align}
where $\gamma_i\subset\partial\Omega_i$, $\gamma_i\cup\Gamma_i^{N_s}\neq \partial \Omega_i$, and
\begin{align*}
    \mathbb{X}^{0,\gamma_i}_{ph,i} = \{\btau_{ph,i}\in\mathbb{X}_{ph,i}\ :\ \btau_{ph,i}\bn_i = 0\ \text{on }\gamma_i\}.
\end{align*}
Note that there is no inter-subdomain continuity in our global spaces. That is, $H^1-$conformity for fluid velocity and $H(\mathrm{div})-$conformity for poroelastic stress and Darcy velocity. To emphasize this, we denote these global spaces as
\begin{gather*}
    \widetilde\bV_\fh= \{\bv_\fh\in\bL^2(\Omega_f)\ :\ \bv_\fh|_{\Omega_i}\in\bV_\fhi\},\quad \widetilde\bV_\ph = \{\bv_\ph\in\bL^2(\Omega_p)\ :\ \bv_\ph|_{\Omega_i}\in\bV_{ph,i}\},\\
    \widetilde{\mathbb{X}}_\ph = \{\btau_\ph\in \mathbb{L}^2(\Omega_p)\ :\ \btau_\ph|_{\Omega_i}\in\mathbb{X}_{ph,i}\}.
\end{gather*}
For a bilinear form $c(\cdot,\cdot)$ on $\Omega_\star$, $\star \in \{f,p\}$, or $\Gamma_{fp}$ we denote its local contribution on $\Omega_i$ as $c_i(\cdot,\cdot)$. We define the ``broken'' global bilinear forms
\begin{align*}
  &\widetilde{a}_f(\bu_\fh,\bv_\fh) = \sum_{i = 1}^{N_f}a_{f,i}(\bu_\fhi,\bv_\fhi),\quad \widetilde{b}_f(\bv_\fh,w_\fh) = \sum_{i = 1}^{N_f}b_{f,i}(\bv_\fhi,w_\fhi),\\
  &\widetilde b_{p}(\bv_\ph,p_\ph)=\sum_{i=N_f+1}^N b_{p,i}(\bv_{ph,i},p_{ph,i}),\quad \widetilde b_e(\btau_{ph};\bxi_{ph},\bchi_{ph})=\sum_{i=N_f+1}^N b_{e,i}(\btau_{ph,i};\bxi_{ph,i},\bchi_{ph,i}).
\end{align*}
We will enforce weakly the $H^1$-conformity for the fluid velocity and the $H(\mathrm{div})$-conformity for the poroelastic stress and Darcy velocity via Lagrange multipliers, see \eqref{e4.4g} below, which will allow us to reduce the hybrid global problem to a problem on the interface in Section \ref{s4.2}. We introduce the Lagrange multipliers $(\blambda_{\ij}^f,\blambda_{\ij}^d,\lambda_{\ij}^p)$:
\begin{align*}
    \blambda_{\ij}^f &:= -\bsigma_{\fh,i}\bn_\ij= -\bsigma_{\fh,j}\bn_\ij\ \text{on }\Gamma_{\ij}\subset\Gamma_{\ff},\\
    \lambda_{\ij}^p&:=p_{\rphi}=p_{\rphj},\ \text{on }\Gamma_{\ij}\subset \Gamma_{\pp},\quad \lambda_{\ij}^p:=-(\bsigma_{\fh,i}\bn_i)\cdot\bn_i=p_{\rphj}\ \text{on }\Gamma_{\ij}\subset\Gamma_{\fp},\\
    \blambda_{\ij}^d&:=\disp_{\rphi}=\disp_{\rphj}\quad \text{on }\Gamma_{\ij}\subset \Gamma_{\pp}\cup\Gamma_{\fp},
\end{align*}
where $\bn_\ij$ is the unit normal vector on $\Gamma_\ij$ outward to $\Omega_i,\ i<j$.
The Lagrange multiplier $\blambda^{f}_\ij$ approximates the normal stress and weakly enforces the continuity of the flux on the Stokes-Stokes interface $\Gamma_{\ij}\subset\Gamma_{\ff}$. The Lagrange multipliers $\lambda_{\ij}^p,\ \blambda_{\ij}^d$ approximate, respectively, the pressure and the displacement on the Biot-Biot interface $\Gamma_{\ij}\subset\Gamma_{\pp}$, and weakly enforce the continuity of the normal component of the Darcy velocity and the normal poroelastic stress. We define the Lagrange multiplier space on $\Gamma_{\ij}$ as follows:
\begin{align*}
    & \boldsymbol{\Lambda}_{\hij}^f =\bV_{\fh}|_{\Gamma_{\ij}}\quad\text{on } \Gamma_{\ij}\subset\Gamma_{\ff},\quad \bLambda^f_h=\Big\{\bmu_h^f\in\bL^2(\Gamma_\ff) \ :\ \bmu_h^f|_{\Gamma_\ij}\in\boldsymbol{\Lambda}_{\hij}^f\Big\},\\
    & \Lambda_{\hij}^p =\bV_{\ph}\cdot\bn_{\ij}|_{\Gamma_{\ij}},\quad \bLambda_{\hij}^d=\mathbb{X}_{\ph}\bn_{\ij}|_{\Gamma_{\ij}}\quad\text{on } \Gamma_{\ij}\subset \Gamma_{\pp}\cup\Gamma_{\fp},\\
    & \Lambda^p_h = \Big\{\mu_h^p\in\rL^2(\Gamma_\pp\cup\Gamma_\fp)\ :\ \mu_h^p|_{\Gamma_\ij}\in\Lambda_{\hij}^p\Big\}, \\
    & \bLambda^d_h =\Big\{\bmu_h^d\in\bL^2(\Gamma_\pp\cup\Gamma_\fp)\ :\ \bmu_h^d|_{\Gamma_\ij}\in\bLambda^d_\hij\Big\}.
\end{align*}
We will henceforth denote
\begin{align*}
    \blambda_h = (\blambda^f_h,\blambda^d_h,\lambda_h^p),\quad\bmu_h=(\bmu^f_h,\bmu^d_h,\mu_h^p),\quad\bLambda_h = \bLambda_h^f\times\bLambda_h^d\times\Lambda_h^p.
\end{align*}
Furthermore, we denote the jumps on fluid and poroelastic interfaces for $1\leq i< j\leq N$ as
\begin{align*}
    &[\bv_{\fh}]=\bv_{\fhi}-\bv_{\fhj},\quad \Gamma_{\ij}\subset\Gamma_{\ff},\\
    &[\bv_{\ph}\cdot\bn]=\bv_{\rphi}\cdot\bn_i+\bv_{\rphj}\cdot\bn_j,\quad [\btau_{\ph}\bn]=\boldsymbol{\tau}_{\rphi}\bn_i+\btau_{\rphj}\bn_j,\quad \Gamma_{\ij}\subset\Gamma_{\pp}.
\end{align*}
We also define the following bilinear forms on the interfaces $\Gamma_\ff$ and $\Gamma_\pp$:
\begin{align*}
    &b_\ff(\bv_\fh,\bmu^f_h):= \sum_{\Gamma_\ij\subset\Gamma_\ff}\langle[\bv_\fh],\bmu^f_h\rangle_{\Gamma_\ij},\quad b_\pp(\bv_\ph,\mu^p_h)=\sum_{\Gamma_\ij\subset\Gamma_\pp}\langle[\bv_\ph\cdot\bn],\mu_h^p\rangle_{\Gamma_\ij},\\
    &b_\ee(\btau_\ph,\bmu_h^d)=\sum_{\Gamma_\ij\subset\Gamma_\pp}\langle[\btau_\ph\bn],\bmu_h^d\rangle_{\Gamma_\ij}.
\end{align*}
%

The domain decomposition formulation is as follows: Find $(\bu_\fh,p_\fh,\bsigma_\ph,\disp_\ph,\bgamma_\ph,\bu_\ph,p_\ph,\blambda_h)\in\widetilde\bV_\fh\times\rW_\fh\times\widetilde{\mathbb{X}}_\ph\times\bV_{dh}\times\mathbb{Q}_\ph\times\widetilde{\bV}_\ph\times\rW_\ph\times\bLambda_h$ such that for all $(\bv_\fh,w_\fh,\btau_\ph,\bxi_\ph,\bchi_\ph,\bv_\ph,w_\ph,\bmu_h)\in\widetilde{\bV}_\fh\times\rW_\fh\times\widetilde{\mathbb{X}}_\ph\times\bV_{dh}\times\mathbb{Q}_\ph\times\widetilde{\bV}_\ph\times\rW_\ph\times\bLambda_h$,
\begin{subequations}\label{e4.4}
\begin{flalign}
  &\widetilde a_f(\bu_\fh,\bv_\fh) + a_{\bjs}(\bu_\fh,\blambda_h^d;\bv_\fh,\mathbf{0}) +\widetilde b_f(\bv_\fh,p_\fh)
  +b_\ff(\bv_\fh,\blambda^f_h)
    + b_{\Gamma_1}(\bv_\fh,\mathbf{0},\mathbf{0};\lambda^p_h)
  =f_f(\bv_\fh),&&\label{e4.4a}\\
    &\widetilde b_f(\bu_\fh,w_\fh)=q_f(w_\fh),\label{e4.4b}\\
  &a_e(\bsigma_\ph,p_\ph;\btau_\ph,0)+\widetilde b_e(\btau_\ph;\disp_\ph,\bgamma_\ph)
  - b_\ee(\btau_\ph,\blambda_h^d) - b_{\Gamma_2}(\btau_\ph,\blambda^d_h)
  = g_s(\btau_\ph),\label{e4.4c}\\
    &\widetilde b_e(\bsigma_\ph;\bxi_\ph,\bchi_\ph) = f_p(\bxi_\ph),\label{e4.4d}\\ 
  &a_p(\bu_\ph,\bv_\ph)-\widetilde b_p(\bv_\ph,p_\ph)
  +b_\pp(\bv_\ph,\lambda^p_h) + b_{\Gamma_1}(\mathbf{0},\bv_\ph,\mathbf{0};\lambda^p_h)
  =g_p(\bv_\ph),\label{e4.4e}\\
    &a_p^p(p_\ph,w_\ph) + a_e(\bsigma_\ph,p_\ph;\mathbf{0},w_\ph)+\widetilde b_p(\bu_\ph,w_\ph) = q_p(w_\ph),\label{e4.4f}\\
  &b_\ff(\bu_\fh,\bmu_h^f)+b_\pp(\bu_\ph,\mu_h^p)-b_\ee(\bsigma_\ph,\bmu_h^d) = 0,\label{e4.4g}\\
&b_{\Gamma_1}(\bu_\fh,\bu_\ph,\blambda^d_h;\mu^p_h)=0,\label{e4.4h}&&\\
  &a_{\bjs}(\bu_\fh,\blambda^d_h;\mathbf{0},\bmu_h^d) + b_{\Gamma_1}(\mathbf{0},\mathbf{0},\bmu_h^d;\lambda^p_h)
+ b_{\Gamma_2}(\bsigma_\ph,\bmu^d_h)=0.\label{e4.4i}&&
\end{flalign}
\end{subequations}
Equations \eqref{e4.4a}--\eqref{e4.4f} reduce to subdomain equations, due to the discontinuity of the test functions across the interfaces. Equation
\eqref{e4.4g} enforces weakly the continuity of the fluid velocity $\bu_\fh$ along $\Gamma_\ff$ as well as the continuity of the normal component of the Darcy velocity $\bu_\ph\cdot\bn_p$ and the normal stress $\bsigma_\ph\bn_p$ along the Biot-Biot interface $\Gamma_\pp$. Equations \eqref{e4.4h} and \eqref{e4.4i} enforce weakly the conservation of mass and momentum along the Stokes-Biot interface $\Gamma_\fp$.

\begin{remark}\label{r1}
   Standard arguments for hybrid mixed finite element methods \cite{BrezziFortin} imply that \eqref{e4.4} has a unique solution. Moreover, \eqref{e4.4} is equivalent to \eqref{e3.10} in the sense that the solution to \eqref{e3.10} is the corresponding component of the solution to \eqref{e4.4}.
\end{remark}

\subsection{Reduction to an interface problem}\label{s4.2}
In this section we show that the algebraic system \eqref{e4.4} can be reduced to a positive definite problem on the interface $\Gamma_\ff\cup\Gamma_\pp\cup\Gamma_\fp$. We introduce two sets of complementary subdomain problems to explicitly separate the dependence on the interface from the dependence on the source terms and the external boundary data. Given, $\blambda_h\in\bLambda_h$ we introduce the following problems. For the Stokes subdomains, $1\leq i\leq N_f$, the problems in the first set read: Find $(\bu_\fhi^*(\blambda_h),p^*_\fhi(\blambda_h))\in\bV_\fhi\times\rW_\fhi$ such that for all $(\bv_\fhi,w_\fhi)\in\bV_\fhi\times\rW_\fhi$, 
\begin{subequations}\label{e4.5}
  \begin{flalign}    &a_{f,i}(\bu_\fhi^*(\blambda_h),\bv_\fhi)
    +a_{\bjs,i}(\bu_\fhi^*(\blambda_h),\mathbf{0};\bv_\fhi,\mathbf{0})
    +b_{f,i}(\bv_\fhi,p_\fhi^*(\blambda_h))\nonumber\\
    &\quad =-\langle\blambda^f_h,\chi\bv_\fhi\rangle_{\Gamma_\ff^i}-\langle \lambda^p_h,\bv_\fhi\cdot\bn_f\rangle_{\Gamma_\fp^i}
    +a_{\bjs,i}(\mathbf{0},\blambda^d_h;\bv_\fhi,\mathbf{0})
    ,&&\label{e4.5a}\\
    &b_{f,i}(\bu_\fhi^*(\blambda_h),w_\fhi)=0,\label{e4.5b}
    \end{flalign}
\end{subequations}
where $\chi$ is either 1 or $-1$ on an interface $\Gamma_{ij}$ depending on the chosen orientation of the normal vector $\bn_{i,j}$.

For the Biot subdomains, $N_f+1\leq i\leq N$, these problems read: Find 
$(\bsigma_{ph,i}^*(\blambda_h),\disp_{ph,i}^*(\blambda_h),\bchi_{ph,i}^*(\blambda_h),
\linebreak \bu_{ph,i}^*(\blambda_h), 
p_{ph,i}^*(\blambda_h)) \in \mathbb{X}_{ph,i}\times\bV_{dh,i}\times\mathbb{Q}_{ph,i}\times\bV_{ph,i}\times\rW_{ph,i}$
such that for all $(\btau_{ph,i},\bxi_{ph,i},\bchi_{ph,i},\bv_{ph,i},w_{ph,i})\in \mathbb{X}_{ph,i}\times\bV_{dh,i}\times\mathbb{Q}_{ph,i}\times\bV_{ph,i}\times\rW_{ph,i}$,
\begin{subequations}\label{e4.6}
    \begin{flalign}
        &a_{e,i}(\bsigma_{ph,i}^*(\blambda_h),p_{ph,i}^*(\blambda_h);\btau_{ph,i},0)+b_{e,i}(\disp_{ph,i}^*(\blambda_h),\bgamma_{ph,i}^*(\blambda_h);\btau_{ph,i})\nonumber\\
        &\quad=\langle\btau_{ph,i}\bn_p,\blambda_h^d\rangle_{\Gamma_\pp^i}+\langle\btau_{ph,i}\bn_p,\blambda^d_h\rangle_{\Gamma_\fp^i},&&\label{e4.6a}\\
    &b_{e,i}(\bxi_{ph,i},\bchi_{ph,i};\bsigma_{ph,i}^*(\blambda_h)) = 0,\label{e4.6b}\\
        &a_{p,i}(\bu_{ph,i}^*(\blambda_h),\bv_{ph,i})-b_{p,i}(\bv_{ph,i},p_{ph,i}^*(\blambda_h))=-\langle\lambda^p_h,\bv_{ph,i}\cdot\bn_p\rangle_{\Gamma_\pp^i}-\langle\lambda^p_h,\bv_{ph,i}\cdot\bn_p\rangle_{\Gamma_\fp^i},\label{e4.6c} \\
    &a_{p,i}^p(p_{ph,i}^*(\blambda_h),w_{ph,i}) + a_{e,i}(\bsigma_{ph,i}^*(\blambda_h),p_{ph,i}^*(\blambda_h);\mathbf{0},w_{ph,i})+b_{p,i}(\bu_{ph,i}^*(\blambda_h),w_{ph,i}) =0.\label{e4.6d}
    \end{flalign}
\end{subequations}

Next, we define the complementary set of problems. For the Stokes subdomains, $1\leq i\leq N_f$, these read: Find $(\bar\bu_\fhi,\bar p_\fhi)\in\bV_\fhi\times\rW_\fhi$ such that for all $(\bv_\fhi,w_\fhi)\in\bV_\fhi\times\rW_\fhi$,
\begin{subequations}\label{e4.7}
\begin{flalign}
    &a_{f,i}(\bar\bu_\fhi,\bv_\fhi)+a_{\bjs,i}(\bar\bu_\fh,\mathbf{0};\bv_\fh,\mathbf{0})+b_{f,i}(\bv_\fhi,\bar p_\fhi)=f_f(\bv_\fhi),\label{e4.7a}&&\\
    &b_{f,i}(\bar\bu_\fhi,w_\fhi)=q_f(w_\fhi).\label{e4.7b}&&
    \end{flalign}
\end{subequations}
For the Biot subdomains, $N_f+1\leq i\leq N$, these read: Find 
$(\bar\bsigma_{ph,i},\bar\disp_{ph,i},\bar\bchi_{ph,i},\bar\bu_{ph,i},\bar p_{ph,i})\in \mathbb{X}_{ph,i}\times\bV_{dh,i}\times\mathbb{Q}_{ph,i}\times\bV_{ph,i}\times\rW_{ph,i}$ such that for all $(\btau_{ph,i},\bxi_{ph,i},\bchi_{ph,i},\bv_{ph,i},w_{ph,i})\in \mathbb{X}_{ph,i}\times\bV_{ph,i}\times\mathbb{Q}_{ph,i}\times\bV_{ph,i}\times\rW_{ph,i}$, 
\begin{subequations}\label{e4.8}
    \begin{flalign}
        &a_{e,i}(\bar\bsigma_{ph,i},\bar p_{ph,i};\btau_{ph,i},0)+b_{e,i}(\bar\disp_{ph,i},\bar\bgamma_{ph,i};\btau_{ph,i})=g_s(\btau_{ph,i}),\label{e4.8a}&&\\
    &b_{e,i}(\bxi_{ph,i},\bchi_{ph,i};\bar\bsigma_{ph,i}) = f_p(\bxi_{ph,i}),\label{e4.8b}&&\\ 
        &a_{p,i}(\bar\bu_{ph,i},\bv_{ph,i})-b_{p,i}(\bv_{ph,i},\bar p_{ph,i})=g_p(\bv_{ph,i}),\label{e4.8c}&&\\
    &a_{p,i}^p(\bar p_{ph,i},w_{ph,i}) + a_{e,i}(\bar\bsigma_{ph,i},\bar p_{ph,i};\mathbf{0},w_{ph,i})+b_{p,i}(\bar\bu_{ph,i},w_{ph,i}) = q_p(w_{ph,i}).\label{e4.8d}&&
    \end{flalign}
\end{subequations}
We will show that these problems are well posed in Section \ref{s4.3}. We can now reduce the problems \eqref{e4.4} to an interface problem. We first define the bilinear form $s_\lambda:\bLambda_h\times\bLambda_h\rightarrow \mathbb{R}$ and the functional $g:\bLambda_h\rightarrow\mathbb{R}$ as
\begin{subequations}
\begin{align}\label{e4.9}
  s_\lambda(\blambda_h,\bmu_h) & = -b_\ff(\bu_\fh^*(\blambda_h),\mu_h^f)-b_\pp(\bu_\ph^*(\blambda_h),\mu_h^p)+b_\ee(\bsigma_\ph^*(\blambda_h),\bmu_h^d)
  \nonumber\\
  &\quad
  -b_{\Gamma_1}(\bu_\fh^*(\blambda_h),\bu_\ph^*(\blambda_h),\blambda_h^d;\mu_h^p)
  + a_{\bjs}(\bu_\fh^*(\blambda_h),\blambda^d_h;\mathbf{0},\bmu_h^d) \nonumber \\
  & \quad + b_{\Gamma_1}(\mathbf{0},\mathbf{0},\bmu_h^d;\lambda^p_h)
+ b_{\Gamma_2}(\bsigma_\ph^*(\blambda_h),\bmu^d_h),\\
g(\bmu) &= b_\ff(\bar\bu_\fh,\mu_h^f)+b_\pp(\bar\bu_\ph,\mu_h^p)-b_\ee(\bar\bsigma_\ph,\bmu_h^d) \nonumber \\
&\qquad +b_{\Gamma_1}(\bar\bu_\fh,\bar\bu_\ph,\boldsymbol{0};\mu_h^p)
-b_{\Gamma_2}(\bar\bsigma_\ph,\bmu_h^d).
\end{align}
\end{subequations}
It is easy to see that problem \eqref{e4.4} is equivalent to solving the interface problem: Find $\blambda_h= (\blambda_h^f,\blambda_h^d,\lambda_p^h)\in\bLambda_h$ such that 
\begin{align}\label{e4.10}
    s_\lambda(\blambda_h,\bmu_h) = g(\bmu_h),\quad \forall\ \bmu_h\in\bLambda_h,
\end{align}
and recovering the solution to \eqref{e4.4} as 
\begin{align*}
    \bu_\fh &= \bu_\fh^*(\blambda_h) + \bar\bu_\fh,\quad p_\fh = p_\fh^*(\blambda_h)+\bar p_\fh,\\
    \bsigma_\ph &= \bsigma_\ph^*(\blambda_h) + \bar\bsigma_\ph,\quad \disp_\ph = \disp_\ph^*(\blambda_h)+\bar\disp_\ph,\quad \bgamma_\ph = \bgamma^*(\blambda_h)+\bar\bgamma,\\
    \bu_\ph &= \bu_\ph^*(\blambda_h) + \bar\bu_\ph,\quad p_\ph = p_\ph^*(\blambda_h)+\bar p_\ph.
\end{align*}
Introducing the operator $S_\lambda:\bLambda_h\rightarrow\bLambda_h$ defined by
\begin{align*}
    (S_\lambda\blambda_h,\bmu_h) = s_\lambda(\blambda_h,\bmu_h),\quad \forall\ \blambda_h,\ \bmu_h\in \bLambda_h
\end{align*}
we can rewrite \eqref{e4.10} in the operator form
\begin{align}\label{e4.11}
    S_\lambda\blambda_h = g_h,
\end{align}
where $g_h\in\bLambda_h$ is defined by $\langle g_h,\bmu_h\rangle_{\Gamma_\fp} = g(\bmu_h),\ \forall\ \bmu_h\in\bLambda_h$.

\subsubsection{Algebraic formulation}
We will now give an algebraic interpretation of this system and show that \eqref{e4.11} is the Schur complement of the system \eqref{e4.4}. Define the operators 
\begin{align*}
    &(A_f\bu_{fh},\bv_{fh})=\widetilde a_f(\bu_{fh},\bv_{fh})+a_{\bjs}(\bu_\fh,\mathbf{0}; \bv_\fh,\mathbf{0}),\quad (A_p\bu_{ph},\bv_{ph})=a_p(\bu_{ph},\bv_{ph}),\\
  &(B_f\bu_{fh},w_{fh})=\widetilde b_f(\bu_{fh},w_{fh}),\quad (B_{\sigma\eta}\bsigma_{ph},\bxi_\ph)=\widetilde b_e(\bsigma_\ph;\bxi_{ph},\boldsymbol{0}), \quad(B_{\sigma\gamma}\bsigma_{ph},\bchi_{ph}) =
  \widetilde b_e(\bsigma_{ph};\boldsymbol{0},\bchi_{ph}),\\
    &(B_p\bu_{ph},w_{ph})=-\widetilde b_p(\bu_{ph},w_{ph}),\quad (A_e^p p_\ph, w_\ph)= a_e(\boldsymbol{0}, p_\ph; \boldsymbol{0}, w_\ph),\quad (A^s_e\bsigma_\ph, \btau_\ph) = a_e(\bsigma_\ph, 0; \btau_\ph, 0),\\
    &(A_p^p p_\ph,w_\ph) = a_p^p(p_\ph, w_\ph),\quad A_e^{sp}= a_e(\bsigma_\ph, 0; \boldsymbol{0}, w_\ph)\quad A_e^{ps}= a_e(\boldsymbol{0}, p_\ph; \btau_\ph, 0),
\end{align*}
where $A_f$ is the only operator with a contribution from the interface $\Gamma_\fp$. We also define the following operators on the degrees of freedom on the interface
\begin{align*}
     &(B_\ff\bu_\fh, \bmu_h^f)=b_\ff(\bu_\fh, \bmu_h^f),\quad (B_\pp\bu_\ph, \mu_h^p)=b_\pp(\bu_\ph, \mu_h^p),\quad (B_\ee\bsigma_\ph,\bmu_h^d)=-b_\ee(\bsigma_\ph,\bmu_h^d),\\
  &(A_{\bjs}^{fd}\bu_\fh, \bmu_h^d)= a_{\bjs}(\bu_\fh,\mathbf{0};\mathbf{0},\bmu_h^d),
  \quad (B_{\Gamma_2}\bsigma_\ph, \bmu_h^d)=b_{\Gamma_2}(\bsigma_\ph,\bmu_h^d), \\
  &(A_{\bjs}^d\blambda_h^d,\bmu_h^d)=a_{\bjs}(\mathbf{0},\blambda_h^d;\mathbf{0},\bmu_h^d),
  \quad (B_{\Gamma_1}^d\blambda_h^d, \mu_h^p)=b_{\Gamma_1}(\boldsymbol{0},\boldsymbol{0},\blambda_h^d;\mu_h^p),\\
     & (B_{\Gamma_1}^p\bu_\ph,\mu_h^p)=b_{\Gamma_1}(\boldsymbol{0},\bu_\ph,\boldsymbol{0};\mu_h^p),\quad (B_{\Gamma_1}^f\bu_\fh,\bmu_h^f)=b_{\Gamma_1}(\bu_\fh,\boldsymbol{0},\boldsymbol{0};\mu_h^p).
    \end{align*}
We can now express the equations \eqref{e4.4} as follows. Multiply \eqref{e4.4i} by a negative sign. For
\begin{align*}
    X = \begin{pmatrix}
    \bu_\fh& p_\fh& \bsigma_\ph&\disp_\ph&\bgamma_\ph&\bu_\ph&p_\ph
\end{pmatrix}^T,\quad \blambda_h = \begin{pmatrix}
    \blambda_h^f&\blambda_h^d&\lambda_h^p
\end{pmatrix}^T,
\end{align*}
we get
\begin{align}\label{e4.12} 
    \begin{pmatrix}
        M & L^T\\
        L & D
    \end{pmatrix}\begin{pmatrix}
        X\\
        \blambda
    \end{pmatrix}= \begin{pmatrix}
        F\\
        0
    \end{pmatrix}
\end{align}
where $F$ is a vector incorporating the source terms and boundary data,
\begin{align}
    M &= \begin{pmatrix}
        M_f & 0\\
        0 & M_p
    \end{pmatrix},\quad M_f = \begin{pmatrix}
        A_f & B_f^T\\
        B_f & 0
    \end{pmatrix},\nonumber\\
    M_p &= E+H = \begin{pmatrix}
        A^s_e & 0 & 0 & 0 & A_e^{ps}\\
        0 & 0 & 0 & 0 & 0\\
         0 & 0 & 0 & 0 & 0\\
          0 & 0 & 0 & 0 & 0\\
           A_e^{sp} & 0 & 0 & 0 & A_p^p+A_e^p\\
    \end{pmatrix}+\begin{pmatrix}
        0 & B_{\sigma \eta}^T & B_{\sigma\gamma}^T & 0 & 0\\
        B_{\sigma \eta} & 0 & 0 & 0 & 0 \\
       B_{\sigma\gamma} & 0 & 0 & 0 & 0 \\
       0 & 0 & 0 & A_p & B_p^T \\
       0 & 0 & 0 & -B_p & 0
    \end{pmatrix},\nonumber\\
    L &= \begin{pmatrix}
        B_\ff & 0 & 0 & 0 & 0 & 0 & 0\\
        -A_{\bjs}^{fd} & 0 & B_{ee}-B_{\Gamma_2} & 0 & 0 & 0 & 0\\
        B_{\Gamma_1}^f & 0 &  0 & 0 & 0 & B_\pp + B_{\Gamma_1}^p & 0
    \end{pmatrix},\quad D = \begin{pmatrix}
        0 & 0 & 0\\
        0 & -A_{\bjs}^d & -(B_{\Gamma_1}^d)^T\\
        0 & B_{\Gamma_1}^d & 0
    \end{pmatrix}.\nonumber
\end{align}
We will show in Section \ref{s4.3} that the subdomain problems are well posed, i.e., the matrix $M$ is invertible. The interface problem \eqref{e4.11} corresponds to the Schur complement
of \eqref{e4.12}:
\begin{align}\label{e4.13}
    S_\lambda\blambda_h= (LM^{-1}L^T-D)\blambda_h = LM^{-1}F.
\end{align}
We will establish in Section~\ref{s4.4} that the interface bilinear form $s_\lambda(\cdot,\cdot)$, while nonsymmetric, is positive definite. This allows us to utilize GMRES to solve the problem \eqref{e4.13} where each iteration requires evaluating the action of 
\begin{align*}
    M^{-1} = \begin{pmatrix}
        M^{-1}_1&&\\
        & \ddots & \\
        & & M_N^{-1}
    \end{pmatrix},
\end{align*}
with $M_i$ the system matrix corresponding to a Stokes or Biot subdomain problem on $\Omega_i$.

\subsection{Well-posedness of the subdomain problems}\label{s4.3}
We denote the mesh on a subdomain $\Omega_i$ as $\mathcal{T}_{h,i}$, $1\leq i\leq N$. We assume that the subdomain partitions $\mathcal{T}_{h,i}$ are shape-regular, in the sense that there exists a constant $\sigma > 0$ independent of $N$ such that
\begin{align}\label{e4.14}
    \forall\ 1\leq i\leq N,\quad \forall\ j:\ |\Gamma_\ij|>0,\quad \frac{D_i}{\rho_i}\leq \sigma,\quad \frac{D_i}{\rho_{ij}}\leq \sigma,
\end{align}
where $\rho_i$, respectively $\rho_{ij}$, is the diameter of the largest ball contained in $\Omega_i$, respectively $\Gamma_\ij$, and $D_i = \text{diam}(\Omega_i)$. We assume quasiuniformity for the traces of the subdomain grids on the interfaces. 
We also assume that there exists positive constants $K_{\min},\ K_{\max}, A_{\min},\ A_{\max}$ such that
\begin{align}\label{e4.16}
\begin{split}
   K_{\min}|\xi|^2\leq \xi^T\bK(\boldsymbol{x})\xi\leq K_{\max}|\xi|^2,\quad A_{\min}|\xi|^2&\leq \xi^TA(\boldsymbol{x})\xi\leq A_{\max}|\xi|^2\quad \forall\  \boldsymbol{x}\in \Omega_p,\ \forall\ \xi\in \mathbb{R}^d.
    \end{split}
\end{align}
Korn's inequality \eqref{e3.11} and \eqref{e4.14} imply
\begin{align}\label{e4.17}
    \|\bD(\bv_\fhi)\|_{\Omega_i}^2+\frac{1}{h}\|\bv_\fhi\|_{\Gamma_i^D}^2\geq C_{K}\|\bv_\fhi\|_{1,\Omega_i}^2\quad \forall\ \bv_\fhi\in\bH^1(\Omega_i),
\end{align}
where $C_{K}=C(\sigma,d)>0$.

\begin{lemma}\label{lemma4.1}
    Let $|\Gamma_i^D| > 0$ for $1\leq i\leq N_f$ and $\gamma > 2\mu C_I^2$, then
    \begin{align*}
    a_{f,i}(\bu_\fhi,\bu_\fhi)\geq C\|\bu_\fhi\|^2_{1,\Omega_i},
\end{align*}
where $C=C(\sigma,d,\mu)$.
\end{lemma}

\begin{proof}
Using \eqref{e4.17}, the proof proceeds as in Lemma \ref{lemma3.1}.
\end{proof}
\begin{remark}\label{r2}
The above Lemma proves that ker $ a_{f,i}$ is trivial for $1\leq i\leq N_f$ if $|\Gamma_i^D| > 0$, where
\begin{align*}
    \kai=\{\bv_\fhi\in\bV_\fhi\ |\ a_{f,i}(\bv_\fhi,\bv_\fhi) = 0\}.
\end{align*}
It is possible to have non trivial $\kai$ for a subdomain $\Omega_i$ if $|\Gamma_i^D| = 0$. In this case, the Stokes matrices $M_i$, $1\leq i\leq N_f$ would be singular and the Stokes velocity $\bv_\fhi$ would be determined up to an element of $\kai$. For a complete characterization of the $\kai$, see \cite[Lemma 4.1]{Changqing}. Furthermore, the appropriate representative in the quotient space $\bV_\fhi/\kai$ can be chosen using a one-level FETI method \cite{feti}. This involves solving a coarse space problem, which projects the interface onto a subspace orthogonal to the kernel of the subdomain operators, see \cite[Section 4.3]{Changqing}. We assume for the sake of simplicity that $|\Gamma_i^D| > 0$ so that $\kai = \{\boldsymbol{0}\}$. For the case of subdomain problems with nontrivial kernel, the arguments can be easily extended to $\bV_\fhi/\kai$ as shown in \cite{Changqing}.
\end{remark}

\begin{theorem}\label{theorem4.1}
Under the assumptions of Lemma~\ref{lemma4.1}, the Stokes subdomain problems \eqref{e4.5} and \eqref{e4.7} have a unique solution.
\end{theorem}
\begin{proof}
Since $|\Gamma_i^D|>0$, the inf-sup condition \eqref{e4.1} holds with $\gamma_i = \Gamma_i^D$. Moreover, for $\bv_\fhi\in\bV_\fhi^{0,\Gamma_i^D}$, it holds that $b_{f,i}(\bv_\fhi,w_\fhi) = -(\nabla\cdot \bv_\fhi,w_\fhi)_{\Omega_i}$. Then, the existence and uniqueness of a solution to \eqref{e4.5} and \eqref{e4.7} follows from the classical saddle point problem theory \cite{BrezziFortin}, using the coercivity of $a_{f,i}$, cf. Lemma \ref{lemma4.1}, and the inf-sup condition \eqref{e4.1} with $\gamma_i = \Gamma_i^D$.
\end{proof}

\begin{theorem}
    The Biot subdomain problems \eqref{e4.6} and \eqref{e4.8} have a unique solution. \label{theorem4.2}
\end{theorem}
\begin{proof}
To show the well posedness of the Biot subdomain problems \eqref{e4.6}, \eqref{e4.8} it is enough to show that the solution to the linear system is unique. For the homogeneous problem
\begin{subequations}\label{e4.18}
    \begin{flalign}
        &a_{e,i}(\bsigma_{ph,i}, p_{ph,i};\btau_{ph,i},0)+b_{e,i}(\disp_{ph,i},\bgamma_{ph,i};\btau_{ph,i})=0,\label{e4.18a}&&\\
    &b_{e,i}(\bxi_{ph,i},\bchi_{ph,i};\bsigma_{ph,i}) = 0,\label{e4.18b}&&\\ 
        &a_{p,i}(\bu_{ph,i},\bv_{ph,i})-b_{p,i}(\bv_{ph,i}, p_{ph,i})=0,\label{e4.18c}&&\\
    &a_{p,i}^p(p_{ph,i},w_{ph,i}) + a_{e,i}(\bsigma_{ph,i}, p_{ph,i};\mathbf{0},w_{ph,i})+b_{p,i}(\bu_{ph,i},w_{ph,i}) = 0,\label{e4.18d}&&
    \end{flalign}
\end{subequations}
we take $\btau_{ph,i}=\bsigma_{ph,i},\ \bxi_{ph,i}=\disp_{ph,i},\ \bchi_{ph,i}=\bgamma_{ph,i},\ \bv_{ph,i}=\bu_{ph,i}, \ w_{ph,i}=p_{ph,i}$, add \eqref{e4.18a}, \eqref{e4.18c}, \eqref{e4.18d}, and using \eqref{e4.18b}, we obtain
\begin{align*}
     &\|A^{1/2}(\bsigma_{ph,i}+\alpha p_{ph,i})\bI\|_{\Omega_i}^2+\|\bK^{-1/2}\bu_{ph,i}\|_{\Omega_i}^2 +\|s_0^{1/2}p_{ph,i}\|_{\Omega_i}^2= 0.
\end{align*}
From \eqref{e4.16}, it follows that $\bsigma_{ph,i}= \mathbf{0},\ \bu_{ph,i}= \mathbf{0},\ p_{ph,i} = 0$ on $\Omega_i$, and using the elasticity inf-sup condition \eqref{e4.3} we get $\disp_{ph,i}= \mathbf{0},\ \bgamma_{ph,i} = \mathbf{0}$.  
\end{proof}

\subsection{Analysis of the interface operator}\label{s4.4}

\begin{theorem}\label{theorem4.3}
    The bilinear form $s_\lambda(\cdot,\cdot)$ is positive definite on $\bLambda_h\times\bLambda_h$.
\end{theorem}

\begin{proof}
We take $\bv_\fhi = \bu^*_\fhi(\blambda_h),\ w_\fhi = p_\fhi^*(\blambda_h),\ \btau_{ph,i}=\bsigma_{ph,i}^*(\blambda_h),\ \bxi_{ph,i}=\disp_{ph,i}^*(\blambda_h), \bchi_{ph,i}=\gamma_{ph,i}^*(\blambda_h),\linebreak \bv_{ph,i}=\bu_{ph,i}^*(\blambda_h)$, and $w_{ph,i}=p_{ph,i}^*(\blambda_h)$ in problems \eqref{e4.5} and \eqref{e4.6}. Adding \eqref{e4.5a}, \eqref{e4.6a}, \eqref{e4.6c}, \eqref{e4.6d} and using \eqref{e4.5b}, \eqref{e4.6b}, \eqref{e4.9}, and \eqref{e4.10}, we obtain 
\begin{align}\label{e4.23}
  \begin{split}     s_\lambda(\blambda_h,\blambda_h)
    & = \sum_{i=1}^{N_f}\Big(a_{f,i}(\bu_\fhi^*(\blambda_h),\bu_\fhi^*(\blambda_h))
    + a_{\bjs,i}(\bu_\fhi^*(\blambda_h)-\blambda^d_h, \bu_\fhi^*(\blambda_h)-\blambda^d_h)\Big)\\
    & \qquad +\sum_{i=N_f+1}^{N}\Big(a_{e,i}(\bsigma_{ph,i}^*(\blambda_h), p_{ph,i}^*(\blambda_h);\bsigma_{ph,i}^*(\blambda_h), p_{ph,i}^*(\blambda_h))
    + a_{p,i}(\bu_{ph,i}^*(\blambda_h),\bu_{ph,i}^*(\blambda_h))\\
& \qquad\qquad +a_{p,i}^p(p_{ph,i}^*(\blambda_h),p_{ph,i}^*(\blambda_h))\Big)\\
    & = \sum_{i=1}^{N_f} \Big(a_{f,i}(\bu_\fhi^*(\blambda_h), \bu_\fhi^*(\blambda_h))
    +    \sum_{l=1}^{d-1}\mu\alpha_{\bjs}\sqrt{\bK^{-1}_l}
    \|(\bu_\fhi^*(\blambda_h)-\blambda^d_h)\cdot\bt_{f,l}\|_{\Gamma_\fp^i}^2\Big)\\
       &\quad +\sum_{i=N_f+1}^N\Big(\|A^{1/2}(\bsigma_{ph,i}^*(\blambda_h)+\alpha p_{ph,i}^*(\blambda_h)\bI)\|_{\Omega_i}^2+\|\bK^{-1/2}\bu_{ph,i}^*(\blambda_h)\|_{\Omega_i}^2+\|s_0^{1/2}p_{ph,i}^*(\blambda_h)\|_{\Omega_i}^2\Big),\\
        \end{split}
    \end{align}
    which implies that $s_\lambda(\blambda_h,\blambda_h)\geq 0$. We now show that if $s_\lambda(\blambda_h,\blambda_h) = 0$ then $\blambda_h=0$. Assume that $s_\lambda(\blambda_h,\blambda_h) = 0$. From \eqref{e4.23} and using the coercivity of $a_{f,i}$ (see Lemma \ref{lemma4.1}), it follows that 
    \begin{align*}
        &\bu_\fhi^*(\blambda_h) = \mathbf{0} \quad \text{on }\Omega_i,\ 1\leq i\leq N_f,\\
        &\bu_{ph,i}^*(\blambda_h)=\mathbf0,\ \bsigma_{ph,i}^*(\blambda_h)=\mathbf0,\ p_{ph,i}^*(\blambda_h)  = 0\quad \text{on }\Omega_i,\ N_f+1\leq i\leq N,\\
        &(\bu_\fhi^*(\blambda_h)-\blambda^d_h)\cdot\bt_{f,l}=0\quad\text{on }\Gamma_\fp^i,\quad N_f+1\leq i\leq N.
    \end{align*}
Using this in \eqref{e4.5a}, \eqref{e4.6a}, \eqref{e4.6c}, it follows that
\begin{subequations}
\begin{align}\label{e4.24a}
  0 &=b_{f,i}(\bv_\fhi,p^*_\fhi(\blambda_h))+\langle\blambda^f_h,\chi\bv_\fhi\rangle_{\Gamma_\ff^i}
  +\langle \lambda^p_h,\bv_\fhi\cdot\bn_f\rangle_{\Gamma_\fp^i},\\
    0&=-b_{e,i}(\disp_{ph,i}^*(\blambda_h),\bgamma_{ph,i}^*(\blambda_h);\btau_{ph,i})+\langle\btau_{ph,i}\bn_p,\blambda_h^d\rangle_{\Gamma_\pp^i}+\langle\btau_{ph,i}\bn_p,\blambda^d_h\rangle_{\Gamma_\fp^i},\label{e4.24b}\\ 
        0&=\langle\lambda^p_h,\bv_{ph,i}\cdot\bn_p\rangle_{\Gamma_\pp^i}+\langle\lambda^p_h,\bv_{ph,i}\cdot\bn_p\rangle_{\Gamma_\fp^i}.  \label{e4.24c}
    \end{align}
    \end{subequations}
From the definition of $\Lambda_h^p$ it follows that there exists $\bv_{ph,i}\in\bV_{ph,i}$ such that $\bv_{ph,i}\cdot\bn_p = \lambda_h^p$ on $\Gamma_\pp^i\cup\Gamma_\fp^i$. Using this in \eqref{e4.24c} it follows that $\lambda_h^p = 0$.

We next show that $\blambda_h^d = \mathbf0$. Let $\Omega_i$, $N_f+1\leq i\leq N$, be such that $|\Gamma_i^{D_d}|>0$. From the inf-sup condition \eqref{e4.3} with $\gamma_i = \Gamma_\pp^i\cup\Gamma_\fp^i$ and \eqref{e4.24b} we get $\disp^*_{ph,i}(\blambda_h)= \mathbf{0}$ and $\bgamma^*_{ph,i}(\blambda_h)=\mathbf0$. Now, taking $\btau_{ph,i}\bn_p = \blambda_h^d$ in \eqref{e4.24b} yields $\blambda_h^d = \mathbf0$ on $\Gamma_\pp^i\cup\Gamma_\fp^i$.

    Next, consider a subdomain $\Omega_j$, $N_f+1\leq j\leq N$, adjacent to $\Omega_i$, i.e., $|\Gamma_{i,j}|>0$. From the inf-sup condition \eqref{e4.3} with $\gamma_j = \partial\Omega_j\setminus\Gamma_{i,j}$ and \eqref{e4.24b} we obtain $\disp^*_{ph,j}(\blambda_h)= \mathbf{0}$ and $\bgamma^*_{ph,j}(\blambda_h)=\mathbf0$. Now, taking $\btau_{ph,j}\bn_p = \blambda_h^d$ in \eqref{e4.24b} yields $\blambda_h^d = 0$ on $\Gamma_\pp^j\cup\Gamma_\fp^j$. Iterating over all subdomains yields $\blambda_h^d = \mathbf0$.    

    We proceed with showing that $\blambda_h^f = \mathbf0$. 
    Let $\Omega_i$, $1\leq i\leq N_f$ a Stokes subdomain such that $|\Gamma_\fp^i|>0$. Using $\gamma_i = \partial\Omega_i\setminus\Gamma_\fp$ in the inf-sup condition \eqref{e4.1}, $b_{f,i}(\bv_\fhi,w_\fhi) = -(\nabla\cdot\bv_\fhi,w_\fhi)_{\Omega_i}$ for $\bv_\fhi\in\bV_\fhi^{0,\gamma_i}$, and \eqref{e4.24a}, we obtain
    \begin{align*}
        \|p_\fhi^*(\blambda_h)\|_{\Omega_i}&\leq  C\sup_{\bv_\fhi\in \bV^{0,\gamma_i}_\fhi\setminus \{\boldsymbol{0}\}} \frac{b_{f,i}(\bv_\fhi,p_\fhi^*(\blambda_h))}{\|\bv_\fhi\|_{1,\Omega_i}}\\
        &=C\sup_{\bv_\fhi\in \bV^{0,\gamma_i}_\fhi\setminus \{\boldsymbol{0}\}} \frac{\langle \lambda^p_h,\bv_\fhi\cdot\bn_f\rangle_{\Gamma_\fp^i}}{\|\bv_\fhi\|_{1,\Omega_i}} = 0,
    \end{align*}
    where we used that $\bu_\fhi^*(\blambda_h) = 0$ on $\Omega_i$, and that $\lambda_h^p = 0$ on $\Gamma_\fp^i$. Taking $\bv_\fhi\in\bV_\fhi$ such that $\bv_\fhi = \blambda_h^f$ on $\Gamma_\ff^i$ in \eqref{e4.24a}, we conclude that $\blambda_h^f = 0$ on $\Gamma_\ff^i$.

    Next, consider a subdomain $\Omega_j$, $1\leq j\leq N_f$ adjacent to $\Omega_j$, i.e., $|\Gamma_\ij|>0$. Letting $\gamma_j = \partial\Omega_j\setminus\Gamma_\ij$ in the inf-sup condition \eqref{e4.1}, we obtain as above that $p_{fh,j}^*(\blambda_h) = 0$. Taking $\bv_\fhj\in\bV_\fhj$ such that $\bv_\fhj = \blambda_h^f$ on $\Gamma_\ff^j$ in \eqref{e4.24a}, it follows that $\blambda_h^f = 0$ on $\Gamma_\ff^j$. Iterating over all subdomains $\Omega_i$, $1\leq i\leq N_f$ yields $\blambda_h^f = \mathbf0$..
\end{proof}

\begin{remark}\label{r3}
Since $\bLambda_h^f = \bV_h^f|_{\Gamma_{ff}}$, it is possible to directly choose an extension $\bv_\fhi\in\bV_\fhi$ such that $\bv_\fhi = \blambda_h^f$ on $\Gamma_\ff^i$ as shown in the proof above. The choice of $\bLambda_h^f$ is possible, since the Dirichlet boundary condition for the the Stokes system is imposed weakly using Nitsche's method. If the Dirichlet boundary condition was imposed essentially, we would need to ensure that $\bv_\fhi\in\bV_\fhi\subset\bH^1_{0,\Gamma_i^D}(\Omega_i)$, which requires a modified choice for $\bLambda_h^f$. We refer the reader to \cite{StokesDarcyMortar} for the analysis in this case. The work in \cite{StokesDarcyMortar} allows for non-matching grids across all interfaces via the use of mortar finite element spaces. Our analysis can be extended to this case. Imposing the boundary condition weakly, as done in the current paper, would help simplify the constructions in \cite{StokesDarcyMortar} .
\end{remark}

Next, we establish spectral bounds for the interface operator. We write the interface bilinear form $s_\lambda$ as a sum of subdomain contributions,
\begin{align*}
    s_\lambda(\blambda_h,\bmu_h) &= \sum_{i = 1}^{N}s_{\lambda,i}(\blambda_h,\bmu_h),
\end{align*}
    where for $1\leq i\leq N_f$,
    \begin{align}
    \begin{split}\label{e4.25}
      s_{\lambda,i}(\blambda_h,\bmu_h) &= -\langle\chi\bu_\fhi^*(\blambda_h),\bmu_h^f\rangle_{\Gamma_\ff^i}
      -\langle\bu_\fhi^*(\blambda_h)\cdot\bn_f,\mu_h^p\rangle_{\Gamma_\fp^i}\\
   &\quad-\sum_{l = 1}^{d-1}\Big\<\mu\alpha_{\bjs}\sqrt{\bK^{-1}_l}(\bu_\fhi^*(\blambda_h)-\blambda_h^d)\cdot\bt_{f,l},\bmu_h^d\cdot\bt_{f,l}\Big\>_{\Gamma_\fp^i},
   \end{split}
   \end{align}
   and for $N_f+1\leq i\leq N$, 
   \begin{align}
   \begin{split}\label{e4.26}
     s_{\lambda,i}(\blambda_h,\bmu_h) &= \langle \bsigma^*_{ph,i}(\blambda_h)\bn_p,\bmu_h^d\rangle_{\Gamma_\pp^i\cup\Gamma_\fp^i}
     -\langle \bu^*_{ph,i}(\blambda_h)\cdot\bn_p,\mu_h^p\rangle_{\Gamma_\pp^i\cup\Gamma_\fp^i}
     +\langle\bmu_h^d\cdot\bn_p,\lambda_h^p\rangle_{\Gamma_\fp^i}-\langle\blambda_h^d\cdot\bn_p,\mu_h^p\rangle_{\Gamma_\fp^i}.
   \end{split}
\end{align}

We begin with a bound for the Biot subproblems.
   
\begin{lemma}\label{lemma4.2}
There exist positive constants $C_{B,1},\ C_{B,2}$ independent of $h$ and $s_0$ such that for all $(\blambda_h^d,\lambda_h^p)\in\bLambda_h^d\times\Lambda_h^p$
\begin{align}\label{Biot-bound}
  \begin{split}      C_{B,1}\left(\|\blambda_h^d\|_{\Gamma_\fp\cup\Gamma_{pp}}^2+\|\lambda_h^p\|_{\Gamma_\fp\cup\Gamma_{pp}}^2\right)
    \leq \sum_{i = N_f+1}^{N}s_{\lambda,i}(\blambda_h,\blambda_h) \leq \frac{C_{B,2}}{h}\left(\|\blambda_h^d\|_{\Gamma_\fp\cup\Gamma_{pp}}^2+\|\lambda_h^p\|_{\Gamma_\fp\cup\Gamma_{pp}}^2\right),
\end{split}
\end{align}
where $C_{B,1},C_{B,2}=C(\sigma,\nu,d,\alpha,\bK,A)$.
\end{lemma}
\begin{proof}
    Take $\btau_{ph,i}=\bsigma^*(\blambda_h)_{ph,i},\ \bxi_{ph,i}=\disp^*(\blambda_h)_{ph,i},\ \bchi_{ph,i}=\gamma^*(\blambda_h)_{ph,i}, \bv_{ph,i}=\bu^*(\blambda_h)_{ph,i}$, and $\ w_{ph,i}=p^*(\blambda_h)_{ph,i}$ in \eqref{e4.6}. Adding \eqref{e4.6a}, \eqref{e4.6c}, and \eqref{e4.6d}, and using \eqref{e4.6b} and \eqref{e4.26}, we obtain, for $N_f+1\leq i\leq N$,
    \begin{align}
    \begin{split}
        s_{\lambda,i}(\blambda_h,\blambda_h)&=a_{e,i}(\bsigma_{ph,i}^*(\blambda_h), p_{ph,i}^*(\blambda_h);\bsigma_{ph,i}^*(\blambda_h), p_{ph,i}^*(\blambda_h))+a_{p,i}(\bu_{ph,i}^*(\blambda_h),\bu_{ph,i}^*(\blambda_h))\\
        &\hspace{2cm}+a_{p,i}^p(p_{ph,i}^*(\blambda_h),p_{ph,i}^*(\blambda_h))\\
        &= \|A^{1/2}(\bsigma_{ph,i}^*(\blambda_h)+\alpha p_{ph,i}^*(\blambda_h)\bI)\|_{\Omega_i}^2+\|\bK^{-1/2}\bu_{ph,i}^*(\blambda_h)\|_{\Omega_i}^2+\|s_0^{1/2}p_{ph,i}^*(\blambda_h)\|_{\Omega_i}^2.
        \end{split}
    \end{align}
Using the above, the proof follows as the proof of (3.47) in \cite[Theorem 3.1]{ManuBiot}.
\end{proof}

For the bound for the Stokes subproblems, we need the following auxiliary result.

\begin{lemma}\label{lemma4.3}
Under the assumptions of Lemma \ref{lemma4.1}, the bilinear form $a_{f,i}(\bu_\fhi,\bv_\fhi)$ is an inner product in $\bV_{fh,i}$, inducing the norm 
\begin{align*}
\|\bv_\fhi\|^2_{a_{f,i}} = a_{f,i}(\bv_\fhi,\bv_\fhi)\quad \forall\ \bv_\fhi\in\bV_\fhi.
\end{align*}
Moreover, there exist positive constants $C_1$ and $C_2$ independent of $h$ such that
\begin{align}\label{norm-equiv}
C_1\|\bv_\fhi\|_{1,\Omega_i}^2\leq \|\bv_\fhi\|^2_{a_{f,i}}\leq C_2\left(\|\bv_\fhi\|^2_{\Omega_i} +
\frac{1}{h}\|\bv_\fhi\|^2_{\Gamma_i^D}\right).
\end{align}
\end{lemma}
\begin{proof}
Since $a_{f,i}(\bu_\fhi,\bv_\fhi)$ is linear and symmetric, the lower bound in \eqref{norm-equiv}, which has been shown in Lemma \ref{lemma4.1}, implies that it is an inner product in $\bV_{fh,i}$ inducing the norm $\|\bv_\fhi\|_{a_{f,i}}$. For the upper bound in \eqref{norm-equiv}, using the definition of $a_{f,i}$, cf. \eqref{e3.9a}, we obtain
    \begin{align*}
        a_{f,i}(\bu_\fhi,\bu_\fhi) &\leq 2\mu\|\bD(\bu_\fhi)\|^2_{\Omega_i} +4\mu\| \bD(\bu_\fhi)\cdot \bn_f\|_{\Gamma_i^D}\| \bu_\fhi\|_{\Gamma_i^D}\nonumber+ \frac{\gamma}{h}\|\bu_\fhi\|^2_{\Gamma_i^D}\\
    &\leq 2\mu\left(1+C_I\right)\| \bu_\fhi\|^2_{1,\Omega_i} + \frac{1}{h}\left(\gamma+2\mu\right)\|\bu_\fhi\|^2_{\Gamma_i^D},
    \end{align*}
    where we used the trace inequality \eqref{e3.12} and Young's inequality.
\end{proof}
\begin{lemma}\label{lemma4.4}
    There exist positive constants $C_{S,1},\ C_{S,2}$ independent of $h$ and $s_0$ such that for all $\blambda_h\in\bLambda_h$
    \begin{align}
        C_{S,1}h\|\blambda_h^f\|_{\Gamma_\ff}^2&\leq \sum_{i = 1}^{N_f}s_{\lambda,i}(\blambda_h,\blambda_h)+\sum_{i = N_f+1}^{N}s_{\lambda,i}(\blambda_h,\blambda_h),\label{e4.29}\\
        \sum_{i = 1}^{N_f}s_{\lambda,i}(\blambda_h,\blambda_h)&\leq C_{S,2}\left(\|\blambda_h^f\|_{\Gamma_\ff}^2+\|\blambda_h^d\|_{\Gamma_\fp}^2+\|\lambda_h^p\|_{\Gamma_\fp}^2\right).\label{e4.30}
    \end{align}
    where $C_{S,1},C_{S,2}=C(\sigma,\nu,d,\mu,\alpha,\alpha_{\bjs},\bK,A,D_i)$.
\end{lemma}
\begin{proof}
We take $\bv_\fhi = \bu^*_\fhi(\blambda_h),\ w_\fhi = p_\fhi^*(\blambda_h)$ in the problems \eqref{e4.5} and, for $\Omega_i,\ 1\leq i\leq N_f$ we obtain
\begin{align*}
    s_{\lambda,i}(\blambda_h,\blambda_h) = a_{f,i}(\bu_\fhi^*(\blambda_h),\bu_\fhi^*(\blambda_h)) +a_{\bjs,i}(\bu_\fhi^*(\blambda_h),\blambda^d_h; \bu_\fhi^*(\blambda_h),\blambda^d_h),
\end{align*}
which implies 
\begin{align}\label{e4.31}
    \|\bu_\fhi^*(\blambda_h)\|^2_{a_{f,i}}\leq Cs_{\lambda,i}(\blambda_h,\blambda_h).
\end{align}

We begin by establishing the upper bound for $s_{\lambda,i}$. Using \eqref{e4.25}, we obtain
\begin{align*}
  s_{\lambda,i}(\blambda_h,\blambda_h) &\leq \|\bu_\fhi^*(\blambda_h)\|_{\Gamma_\ff^i}\|\blambda_h^f\|_{\Gamma_\ff^i}
  +\|\bu_\fhi^*(\blambda_h)\cdot\bn_f\|_{\Gamma_\fp^i}\|\lambda_h^p\|_{\Gamma_\fp^i}\\
&\qquad
+\mu\alpha_{\bjs}K_{\min}^{-1/2}\sum_{l=1}^{d-1}\left(\|\bu_\fhi^*(\blambda_h)\cdot\bt_{f,l}\|_{\Gamma_\fp^i}\|\blambda_h^d\cdot\bt_{f,l}\|_{\Gamma_\fp^i} + \|\blambda_h^d\cdot\bt_{f,l}\|_{\Gamma_\fp^i}^2 \right)\\
        &\leq C\|\bu_\fhi^*(\blambda_h)\|_{1,\Omega_i}\left(\|\blambda_h^f\|_{\Gamma_\ff^i}+\|\blambda_h^d\|_{\Gamma_\fp^i}+\|\lambda_h^p\|_{\Gamma_\fp^i} \right) + C \|\blambda_h^d\|_{\Gamma_\fp^i}^2.
\end{align*}
In the last inequality we have used the continuous trace inequality on $\Omega_i$; consequently the constant depends on $\text{diam}(\Omega_i) = D_i$. Using the coercivity of $a_f$ (see Lemma \ref{lemma4.1}), \eqref{e4.31}, and the above estimate, we obtain 
\begin{align*}
         s_{\lambda,i}(\blambda_h,\blambda_h)&\leq Cs^{1/2}_{\lambda,i}(\blambda_h,\blambda_h)\left(\|\blambda_h^f\|_{\Gamma_\ff^i}+\|\blambda_h^d\|_{\Gamma_\fp^i}+\|\lambda_h^p\|_{\Gamma_\fp^i}\right) + C \|\blambda_h^d\|_{\Gamma_\fp^i}^2,
\end{align*}
from which the bound \eqref{e4.30} follows, using Young's inequality for the first term on the right hand side. 
    
    Next, we establish the lower bound. Consider the subdomain $\Omega_i\subset\Omega_f$ such that $|\Gamma_\fp^i|>0$. Then, the inf-sup condition \eqref{e4.1} with $\gamma_i = \partial\Omega_i\setminus\Gamma_\fp$ and \eqref{e4.5a} give
    \begin{align}
    \begin{split}
      \|p_\fhi^*(\blambda_h)\|_{\Omega_i} &\leq C\sup_{\bv_\fhi\in\bV_\fhi^{0,\gamma_i}\setminus\{\mathbf{0}\}}\frac{1}{\|\bv_\fhi\|_{1,\Omega_i}}\Big(a_{f,i}(\bu_\fhi^*(\blambda_h),\bv_\fhi)+a_{\bjs,i}(\bu_\fhi^*(\blambda_h),\blambda_h^d;\bv_\fhi,\mathbf{0})
      \nonumber\\
        &\hspace{4cm} +\langle \lambda^p_h,\bv_\fhi\cdot\bn_f\rangle_{\Gamma_\fp^i}\Big)\nonumber
        \end{split}\\
          &\leq C\sup_{\bv_\fhi\in\bV_\fhi^{0,\gamma_i}\setminus\{\mathbf{0}\}}\frac{1}{\|\bv_\fhi\|_{1,\Omega_i}}
\Big(\|\bu_\fhi^*(\blambda_h)\|_{1,\Omega_i}\|\bv_\fhi\|_{1,\Omega_i} \nonumber \\
& \hspace{4cm}   +\big(\|\bu_\fhi^*(\blambda_h)\|_{\Gamma_\fp^i}
          +\|\blambda_h^d\|_{\Gamma_\fp^i}
          +\|\lambda_h^p\|_{\Gamma_\fp^i}\big)\|\bv_\fhi\|_{\Gamma_\fp^i}\Big)\nonumber\\
          &\leq C\left(\|\bu_\fhi^*(\blambda_h)\|_{1,\Omega_i}+\|\blambda_h^d\|_{\Gamma_\fp^i}+\|\lambda_h^p\|_{\Gamma_\fp^i}\right) \nonumber \\
          & \leq C\left(\|\bu_\fhi^*(\blambda_h)\|_{a_{f,i}}+\|\blambda_h^d\|_{\Gamma_\fp^i}+\|\lambda_h^p\|_{\Gamma_\fp^i}\right),\label{e4.32}
    \end{align}
where in the second inequality, we used that
\begin{align*}
    a_{f,i}(\bu_\fhi^*(\blambda_h),\bv_\fhi) = 2\mu(\bD(\bu_\fhi^*(\blambda_h)),\bD(\bv_\fhi))_{\Omega_i} \quad \text{for }\bv_\fhi\in\bV_\fhi^{0,\gamma_i},
\end{align*}
and in the last inequality we used \eqref{norm-equiv}.

We now consider an auxiliary problem on this Stokes subdomain $\Omega_i$.
We extend $\blambda_h^f$ from $\Gamma_\ff^i$ to $\partial\Omega_i$ such that
$ E_i\blambda_h^f\in \bH^{1/2}(\partial\Omega_i),$ $E_i\blambda_h^f=\blambda_h^f$ on $\Gamma_\ff^i$,
\begin{align}\label{e4.33}
    &\|E_i\blambda_h^f\|_{1/2,\partial\Omega_i}\leq C\|\blambda_h^f\|_{1/2,\Gamma_\ff^i},\quad \|E_i\blambda_h^f\|_{\partial\Omega_i}\leq C\|\blambda_h^f\|_{\Gamma_\ff^i},\ \text{and}\quad \int_{\partial\Omega_i}E_i\blambda_h^f = \mathbf{0}.
\end{align}
We refer the reader to \cite{Stein1970} for the existence of such extension. 

Let $P_\Lambda$ be the $L^2$-projection onto $\bV_{h,i}|_{\partial \Omega_i}$. It satisfies \cite{StokesDarcyMortar}
\begin{align}\label{e4.34}
    \|P_\Lambda \bvarphi\|_{\partial\Omega_i}\leq C\|\bvarphi\|_{\partial\Omega_i},\quad \|P_\Lambda\bvarphi\|_{1/2,\partial\Omega_i}\leq C\|\bvarphi\|_{1/2,\partial\Omega_i},\ \text{and }\int_{\partial\Omega_i}P_\Lambda E_i\blambda_h^f =\int_{\partial\Omega_i}E_i\blambda_h^f = \mathbf{0}.
\end{align}
Let $\bpsi_i \in \bH^1(\Omega_i)$ solve the auxiliary problem
\begin{align*}
    \nabla\cdot\bpsi_i &= 0\quad \text{in }\Omega_i,\\
    \bpsi_i &= P_\Lambda E_i\blambda_h^f\quad \text{on }\partial\Omega_i,
\end{align*}
which is well posed, due to \eqref{e4.34}. It holds that \cite{Galdi}
\begin{align}\label{e4.35}
    \|\bpsi_i\|_{1,\Omega_i}\leq C\|P_{\Lambda}E_i\blambda_h^f\|_{1/2,\partial\Omega_i}\leq C\|E_i\blambda_h^f\|_{1/2,\partial\Omega_i}\leq C\|\blambda_h^f\|_{1/2,\Gamma_\ff^i}.
\end{align}
Let $\bpsi_{h,i}\in\bV_{fh,i}$ be the Stokes finite element projection of $\bpsi_i$. It  satisfies
\begin{align}\label{e4.36}
    \bpsi_{h,i}|_{\Gamma_\ff^i} = \bpsi_i|_{\Gamma_\ff^i} = \blambda_h^f|_{\Gamma_\ff^i}\quad \text{and}\quad\|\bpsi_{h,i}\|_{1,\Omega_i}\leq C\|\bpsi_i\|_{1,\Omega_i}.
\end{align}
Now take $\bv_\fhi = -\chi\bpsi_{h,i}$ in \eqref{e4.5a}. Using \eqref{e4.32}, Lemma \ref{lemma4.3}, and the inverse inequality (see \cite{ciarlet}), it follows that
\begin{align*}
  \|\blambda_h^f\|_{\Gamma_\ff^i}^2 &= \langle\blambda_h^f,\bpsi_{h,i}\rangle_{\Gamma_\ff^i} \\
  & = a_{f,i}(\bu_\fhi^*(\blambda_h),\bpsi_{h,i})+a_{\bjs,i}(\bu_\fhi^*(\blambda_h)-\blambda_h^d,\bpsi_{h,i})
  +b_{f,i}(\bpsi_{h,i},p_\fhi^*(\blambda_h))
  +\langle \lambda^p_h,\bpsi_{h,i}\cdot\bn_f\rangle_{\Gamma_\fp^i}
  \\
    &\leq C\Big(\|\bu_\fhi^*(\blambda_h)\|_{a_{f,i}}\|\bpsi_{h,i}\|_{a_{f,i}}+\big(\|\bu_\fhi^*(\blambda_h)\|_{a_{f,i}}+\|\blambda_h^d\|_{\Gamma_\fp^i}+\|\lambda_h^p\|_{\Gamma_\fp^i}\big)\|\bpsi_{h,i}\|_{\Gamma_\fp^i}\\
    &\qquad+\|p_\fhi^*(\blambda_h)\|_{\Gamma_i^D}\|\bpsi_{h,i}\|_{\Gamma_i^D}\Big)\\
  &\leq C\Big(\|\bu_\fhi^*(\blambda_h)\|_{a_{f,i}}\big(\|\bpsi_{h,i}\|_{1,\Omega_i}+h^{-1/2}\|\bpsi_{h,i}\|_{\Gamma_i^D}\big)\\
  &\qquad
  +\big(\|\bu_\fhi^*(\blambda_h)\|_{\Gamma_\fp^i}+\|\blambda_h^d\|_{\Gamma_\fp^i}
    +\|\lambda_h^p\|_{\Gamma_\fp^i}\big)\|\bpsi_{h,i}\|_{\Gamma_\fp^i}
    +\|p_\fhi^*(\blambda_h^f)\|_{\Omega_i}h^{-1/2}\|\bpsi_{h,i}\|_{\Gamma_i^D}\Big)\\
    &\leq C\Big(\|\bu_\fhi^*(\blambda_h)\|_{a_{f,i}}\big(\|\bpsi_{h,i}\|_{1,\Omega_i}+h^{-1/2}\|\bpsi_{h,i}\|_{\Gamma_i^D}\big)\\
    &\qquad
    +\big(\|\bu_\fhi^*(\blambda_h)\|_{\Gamma_\fp^i}+\|\blambda_h^d\|_{\Gamma_\fp^i}
    +\|\lambda_h^p\|_{\Gamma_\fp^i}\big)\|\bpsi_{h,i}\|_{\Gamma_\fp^i}
    \\
    &\qquad
    +\big(\|\bu_\fhi^*(\blambda_h)\|_{a_{f,i}} +\|\blambda_h^d\|_{\Gamma_\fp^i}+\|\lambda_h^p\|_{\Gamma_\fp^i}\big)h^{-1/2}\|\bpsi_{h,i}\|_{\Gamma_i^D}\Big)\\
    &\leq C\big(\|\bu_\fhi^*(\blambda_h)\|_{\Gamma_\fp^i}+\|\blambda_h^d\|_{\Gamma_\fp^i}+\|\lambda_h^p\|_{\Gamma_\fp^i}\big)\big(\|\bpsi_{h,i}\|_{1,\Omega_i}+h^{-1/2}\|\bpsi_{h,i}\|_{\Gamma_i^D}\big)\\
    &\leq C\big(\|\bu_\fhi^*(\blambda_h)\|_{\Gamma_\fp^i}+\|\blambda_h^d\|_{\Gamma_\fp^i}+\|\lambda_h^p\|_{\Gamma_\fp^i}\big)\big(\|\bpsi_i\|_{1,\Omega_i}+h^{-1/2}\|P_\Lambda E_i\blambda_h^f\|_{\Gamma_i^D}\big),
\end{align*}
where we used the inverse inequality $\|p_\fhi^*(\blambda_h^f)\|_{\Gamma_i^D}\leq C h^{-1/2}\|p_\fhi^*(\blambda_h^f)\|_{\Omega_i}$, the continuous trace inequality \\$\|\bpsi_{h,i}\|_{\Gamma_\fp^i}\leq C\|\bpsi_{h,i}\|_{1,\Omega_i}$, and the estimate \eqref{e4.36}. Now, using \eqref{e4.35}, \eqref{e4.34}, \eqref{e4.33}, \eqref{e4.31}, and \eqref{Biot-bound} in the above estimate, we obtain
\begin{align}
    \|\blambda_h^f\|_{\Gamma_\ff^i}^2 &\leq C\big(\|\bu_\fhi^*(\blambda_h)\|_{\Gamma_\fp^i}+\|\blambda_h^d\|_{\Gamma_\fp^i}+\|\lambda_h^p\|_{\Gamma_\fp^i}\big)\big(\|\blambda_{h}^f\|_{1/2,\Gamma_\ff^i}+h^{-1/2}\|\blambda_h^f\|_{\Gamma_\ff^i}\big) \nonumber \\
    &\leq C h^{-1/2}\Big(s_{\lambda,i}^{1/2}(\blambda_h,\blambda_h)+\sum_{i = N_f+1}^N s_{\lambda,i}^{1/2}(\blambda_h,\blambda_h)\Big)\|\blambda_h^f\|_{\Gamma_\ff^i}. \label{interm-bound}
\end{align}
We next consider a subdomain $\Omega_j$ that has an interface with $\Omega_i$. We can use the inf-sup condition \eqref{e4.1} with $\gamma_i = \partial\Omega_i\setminus\Gamma_{i,j}$ to obtain a bound like \eqref{e4.32}. Continuing in a similar way as above, we obtain an extension of the bound \eqref{interm-bound} that includes $\|\blambda_h^f\|_{\Gamma_\ff^j}$. Proceeding iteratively, we obtain the upper bound \eqref{e4.30}.
\end{proof}

We are now ready to establish the main result. Let $\Gamma$ denote the union of all interfaces, $\Gamma = \Gamma_\fp\cup\Gamma_\ff\cup\Gamma_\pp$.

\begin{theorem}\label{theorem4.4}
    If there are no Stokes-Stokes interfaces, then
    \begin{align}
        C_{B,1}\|\blambda_h\|^2_{\Gamma} \leq s_\lambda(\blambda_h,\blambda_h)\leq \max\left\{C_{S,2},\frac{C_{B,2}}{h}\right\}\|\blambda_h\|^2_{\Gamma},\quad \forall\ \blambda_h\in\bLambda_h.
    \end{align}
    In the presence of Stokes-Stokes interfaces, then
    \begin{align}
        \min\left\{C_{S,1}h,\ C_{B,1}\right\}\|\blambda_h\|^2_{\Gamma} \leq s_\lambda(\blambda_h,\blambda_h)\leq \max\left\{C_{S,2},\frac{C_{B,2}}{h}\right\}\|\blambda_h\|^2_{\Gamma},\quad \forall\ \blambda_h\in\bLambda_h.
    \end{align}
\end{theorem}
\begin{proof}
    The assertion of the theorem follows from Lemma \ref{lemma4.2} and Lemma \ref{lemma4.4}.
\end{proof}
Theorem \ref{theorem4.3} shows that the interface operator $S_\lambda$ is positive definite, and Theorem \ref{theorem4.4} gives a bound on its field-of-values. We can therefore use GMRES to solve this linear problem. The following corollary follows from the field-of-values analysis in \cite{field_of_values}.

\begin{corollary}\label{c4.1}
Let $\br_k$ be the $k-$th GMRES residual for solving \eqref{e4.13} and let $|\cdot|$ denote the Euclidean vector norm. Then, it holds that
    \begin{align}
      |\br_k|\leq \left(\sqrt{1-Ch^2}\right)^k|\br_0|, \quad
      \text{if } \ |\Gamma_\ff|=0,\label{e4.43}\\
        |\br_k|\leq \left(\sqrt{1-Ch^4}\right)^k|\br_0|,\quad \text{if } \ |\Gamma_\ff|\neq 0.\label{e4.44}
    \end{align}
\end{corollary}
\begin{proof}
Using Theorem \ref{theorem4.4}, the statement follows by taking
\begin{align*}
&  C = \frac{C_{B,1}}{\max\left\{C_{S,2}h,C_{B,2}\right\}} \quad \text{if } \ |\Gamma_\ff| = 0,\\
  & C = \frac{\min\left\{C_{S,1},\ \frac{C_{B,1}}{h}\right\}}{\max\left\{C_{S,2}h,C_{B,2}\right\}} \quad
  \text{if } \ |\Gamma_\ff| \neq 0.
    \end{align*}
\end{proof}

\subsection{Discretization error estimate}

\begin{theorem}\label{theorem4.5}
    Let $(\bu_\fh,p_\fh,\bsigma_\ph,\disp_\ph,\bgamma_\ph,\bu_\ph,p_\ph,\blambda_h)\in\widetilde\bV_\fh\times\rW_\fh\times\widetilde{\mathbb{X}}_\ph\times\bV_{dh}\times\mathbb{Q}_\ph\times\widetilde{\bV}_\ph\times\rW_\ph\times\bLambda_h$ be the solution to \eqref{e4.4} and let $(\bu_f,p_f,\bsigma_p,\disp_p,\bgamma_p,\bu_p,p_p,\blambda)\in\bV_f\times\rW_f\times\mathbb{X}_p\times\bV_d\times\mathbb{Q}_p\times{\bV}_p\times\rW_p\times\bLambda_h$ of sufficient regularity, be the solution of the continuous problem. Under the assumptions of Lemma~\ref{lemma3.1}, there exists a positive constant $C$ independent of $h_f,\ h_p,\ s_0$, and $A_{min}$ such that
\begin{align*}
        &\|\bu_f-\bu_\fh\|_{1,\Omega_f}+\|(\bu_f-\blambda^d)-(\bu_\fh-\blambda_h^d)\|_{a_{\bjs}}+\|p_f-p_\fh\|_{\Omega_f}+\|A^{1/2}(\bsigma_p-\bsigma_\ph)\|_{\Omega_p}\\
  &\qquad +\|\nabla\cdot(\bsigma_p-\bsigma_\ph)\|_{\Omega_p}+\|A^{1/2}((\bsigma_p+\alpha p_p\mathbf{I})-(\bsigma_{ph}+\alpha p_\ph\mathbf{I}))\|_{\Omega_p}+\|\disp_p-\disp_\ph\|_{\Omega_p}
  \\
  &\qquad
  +\|\bgamma_p-\bgamma_\ph\|_{\Omega_p}
  +\|\bu_p-\bu_\ph\|_{\bH(\mathrm{div};\Omega_p)}+\|p_p-p_\ph\|_{\Omega_p}
  \\
        & \quad \leq C\Big(h_f^{s_{\bu_f}}\|\bu_f\|_{s_{\bu_f}+1,\Omega_f}+h_f^{s_{p_f}+1}\|p_f\|_{s_{p_f}+1,\Omega_p}+h_p^{s_{\bsigma_p}+1}\big(\|\bsigma_p\|_{s_{\bsigma_p}+1,\Omega_p}+\|\nabla\cdot\bsigma_p\|_{s_{\bsigma_p}+1,\Omega_p}\big)\\
        &\qquad+h_p^{s_{\disp_p}+1}\|\disp_p\|_{s_{\disp_p}+1,\Omega_p}+h_p^{s_{\bgamma_p}+1}\|\bgamma_p\|_{s_{\bgamma_p}+1,\Omega_p}+h_p^{s_{\bu_p}+1}\big(\|\bu_p\|_{s_{\bu_f}+1,\Omega_p}+\|\nabla\cdot\bu_p\|_{s_{\bu_f}+1,\Omega_p}\big)\\
  &\qquad+h_p^{s_{p_p}+1}\|p_p\|_{s_{p_p}+1,\Omega_p}
  \Big),
    \end{align*}
    where $s_{\bu_f},\ s_{p_f},\ s_{\bsigma_p},\ s_{\disp_p},\ s_{\bgamma_p},\ s_{\bu_p},\ s_{p_p}$ are the polynomial degrees corresponding to the finite element spaces $\bV_\fh,\ \rW_\fh,\ \mathbb{X}_\ph,\ \bV_\dh,\ \mathbb{Q}_\ph,\ \bV_\ph,\ \rW_\ph$, respectively, and $\|\bu_f\|_{a_{\bjs}}:=a_{\bjs}(\bu_f,\bu_f).$
\end{theorem}

\begin{proof}
The proof is based on an adaptation of the analysis developed \cite{fpsi-mixed-elast} for the quasistatic problem. The main difference is that here we use Nitsche's method for the Dirichlet boundary conditions in the Stokes subproblem. This can be handled by the use of the modified Stokes inf-sup condition \eqref{e3.1} and the coercivity of the bilinear form $a_f$ established in Lemma~\ref{lemma3.1}. We omit further details.
\end{proof}




\section{Numerical experiments}\label{sec:numer}

In this section we present the results from a series of numerical tests illustrating the performance of the proposed method. The numerical scheme is implemented on rectangular grids using the finite element package deal.II \cite{dealii2019design,dealii9.6}.

We use the finite element triple $\mathbb{X}_{ph}\times\bV_{dh}\times\mathbb{Q}_{ph}=(\mathcal{BDM}_1)^2\times (\mathcal{Q}_0)^2\times\mathcal{Q}_0$ for elasticity \cite{fe_elast_1,fe_elast_2}, the finite element pair $\bV_\ph\times\rW_\ph = \mathcal{BDM}_1\times\mathcal{Q}_0$ for Darcy, and the Taylor-Hood finite element $\bV_\fhi\times\rW_\fhi = (\mathcal{Q}_2)^2\times\mathcal{Q}_1$ for Stokes \cite{BrezziFortin}. According to \eqref{e3.7}, for the Lagrange multiplier spaces this choice yields piecewise discontinuous linears for $\Lambda_h^p,\ \bLambda_h^d$ and continuous quadratics for $\bLambda_h^f$. See Table \ref{tab:polynomial-degrees} for the associated polynomial degrees. 

We note that this choice of $\mathbb{X}_{ph}\times\bV_{dh}\times\mathbb{Q}_{ph}$ fits into the framework of the multipoint stress-flux mixed finite element method \cite{msmfe-quads,msfmfe-Biot,fpsi-msfmfe} where the stress, Darcy velocity, and rotation variables can be locally eliminated, resulting in a very efficient positive definite cell-centered scheme for the displacement and Darcy pressure in each Biot subdomain. For solving the interface problem we use non-restarted GMRES with a tolerance of $10^{-8}$ on the relative residual $\frac{|\mathbf{r}_k|}{|\mathbf{r}_0|}$ as the stopping criteria. We consider the isotropic case for elasticity, $\boldsymbol{\sigma}_e = \lambda_p(\nabla\cdot\boldsymbol{\eta}_p)\boldsymbol{\mathrm{I}}+2\mu_p\boldsymbol{\mathrm{D}}(\boldsymbol{\eta}_p)$, where $0 < \lambda_\text{min}\leq\lambda_p(x) \leq \lambda_\text{max}$ and $0 < \mu_\text{min} \leq \mu_p(x) \leq \mu_\text{max}$ are the Lam\'{e} parameters. In this case,
\begin{align*}
    A(\boldsymbol{\tau})
= \frac{1}{2\mu_p} \left(\boldsymbol{\tau} -
\frac{\lambda_p}{2\mu_p + d\lambda_p}
\,\operatorname{tr}(\boldsymbol\tau)\, \boldsymbol{\mathrm{I}}
\right),
\qquad
A^{-1}(\boldsymbol\tau)
=
2\mu_p\, \boldsymbol\tau
+
\lambda_p \operatorname{tr}(\boldsymbol{\tau})\, \boldsymbol{\mathrm{I}},
\end{align*}
with $A_\text{min}=1/(2\mu_\text{max}+d\lambda_\text{max})$ and $A_\text{max} = 1/(2\mu_\text{min})$. 

We investigate the rates of convergence, number of GMRES iterations, and robustness with respect to $s_0$ for several subdomain configurations, including a case with non-matching grids along the Stokes–Biot interface.
\renewcommand{\tabcolsep}{4.3pt}
\begin{table}[ht]
\centering
\caption{Degree of polynomials associated with the finite element spaces used in the numerical experiments.}
\label{tab:polynomial-degrees}
\begin{tabular}{cccccccccc}
\hline
$\bV_\fh:s_{\bu_f}$ & $\rW_\fh:s_{p_f}$ & $\mathbb{X}_p:s_{\bsigma_p}$ & $\bV_d:s_{\disp_p}$ & $\mathbb{Q}_p:s_{\bgamma_p}$ & $\bV_p:s_{\bu_p}$ & $\rW_p:s_{p_p}$ & $\bLambda_h^f:s_{\blambda^f}$ & $\Lambda_h^p:s_{\lambda^p}$ & $\bLambda_h^d:s_{\blambda^d}$\\
\hline
$2$ & $1$ & $1$ & $0$ & $0$ & $1$ & $0$ & $2$ & $1$ & $1$ \\
\hline
\end{tabular}
\end{table}

In three of the examples we consider a test case with domain $\Omega = (0, 2)^2$ and a known analytical solution and Dirichlet boundary conditions on $\partial \Omega$. We associate the left half $(0, 1) \times (0, 2)$ with the Stokes flow, while the right half $(1, 2)\times(0, 2)$ represents the flow in the poroelastic structure governed by the Biot system. In one of the examples the domain is $\Omega = (0, 2)\times(0,1)$ and it is split in a similar way. The physical parameters are $\boldsymbol{\mathrm{K}} = \boldsymbol{\mathrm{I}}, \mu = 1, \alpha = 1, \alpha_{\bjs} = 1,
s_0 = 1,\ 10^{-3}, \lambda_p = 1$, and $\mu_p = 1$. The solution in the Stokes region is
\begin{align*}
    \bu_f = \begin{pmatrix}
        -3x +\cos(\pi y)\\
        y+\sin(2\pi x)
    \end{pmatrix},\quad p_f = \sin(\pi x)\cos\left(2\pi y\right).
\end{align*} 
The Biot solution $(\boldsymbol{u}_p, p_p, \boldsymbol{\eta}_p)$ is chosen accordingly to satisfy the interface conditions at $x = 1$:
\begin{align*}
    &\bu_p = \pi\begin{pmatrix}
        -\cos(\pi x)\cos(2\pi y)\\
        2\sin(\pi x)\sin(2\pi y)
    \end{pmatrix},\quad
    p_p = \sin(\pi x)\cos(2\pi y) + 6,\\
    \disp_p = &\begin{pmatrix}
        -\frac{x}{3}-\frac{8}{3}+\cos(\pi y)+\pi\cos(\pi x)\cos(2\pi y)-\frac{2\pi}{3}\cos(\frac{\pi x}{2})\cos(\pi y)\\
        y+\sin(2\pi x) + 2\pi\sin(\pi x)\sin(2\pi y)+\pi(2\cos(2\pi x)-\sin(\pi y))
    \end{pmatrix}.
\end{align*}

From Corollary \ref{c4.1} we obtain the following bound for the iteration count for $h$ small enough:
\begin{align*}
    \text{iteration count} = \mathcal{O}(h^{-2})\ \text{if }|\Gamma_\ff| = 0 \quad \text{and} \quad \mathcal{O}(h^{-4})\ \text{if }|\Gamma_\ff| \neq 0.
\end{align*}
This bound is usually too pessimistic in practice. In particular, if the interface operator is normal the number of iterations required for GMRES to converge is $\mathcal{O}\left(\sqrt\frac{\lambda_{\max}}{\lambda_{\min}}\right)$ where $\lambda_{\max}$ and $\lambda_{\min}$ are the largest and smallest eigenvalues of $S_\lambda$ \cite{gmres_1,gmres_2}. The bounds in Theorem \ref{theorem4.4} in this case give
\begin{align}\label{e5.1}
  \text{iteration count} =   \mathcal{O}(h^{-0.5})\ \text{if }|\Gamma_\ff| = 0 \quad \text{and} \quad \mathcal{O}(h^{-1})\ \text{if }|\Gamma_\ff| \neq 0.
\end{align}

We present examples for four different domain decomposition and grid configurations, see  Figure~\ref{fig:four-examples}. The results from each example are reported in a table format. A graphical representation of the results can be found at the end of the section in Figure~\ref{fig:results}.

\begin{figure}
  \centering
  \includegraphics[width=.22\textwidth]{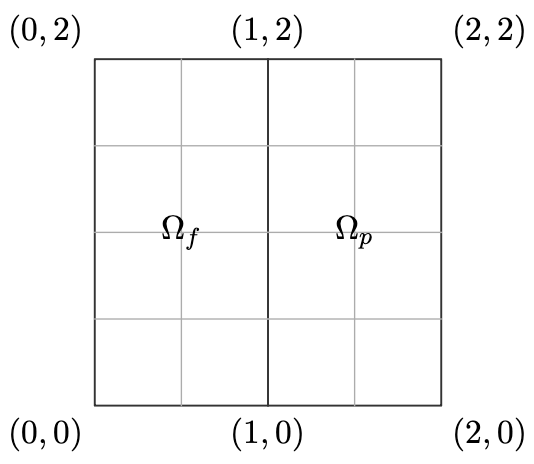}
  \includegraphics[width=.22\textwidth]{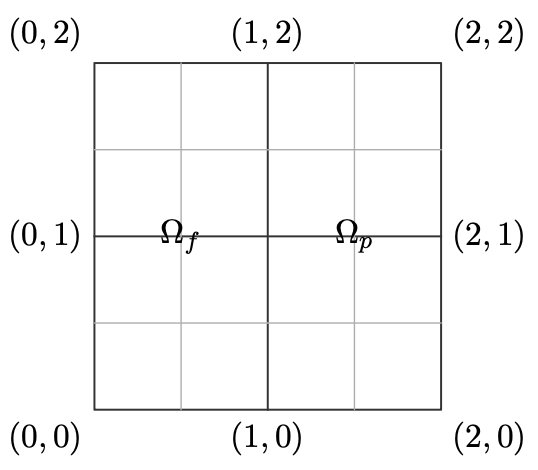}
  \includegraphics[width=.32\textwidth]{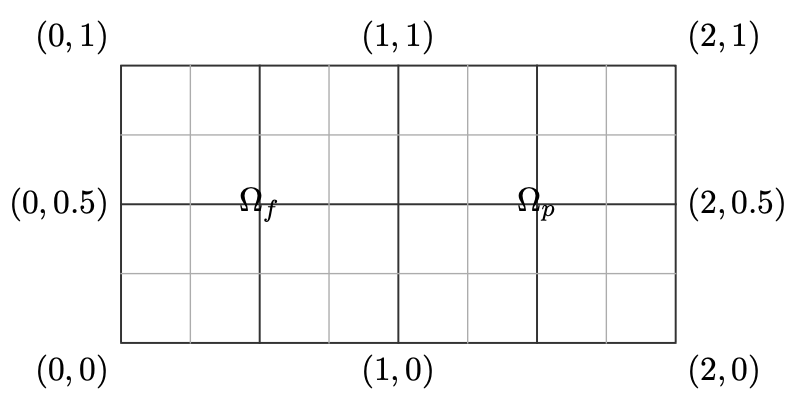}
  \includegraphics[width=.22\textwidth]{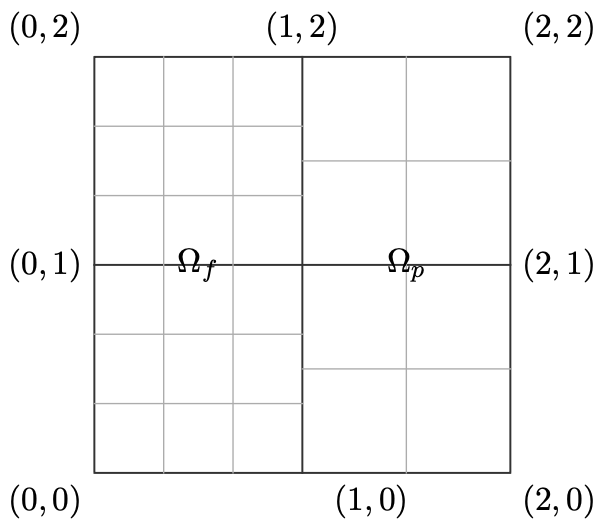}
  \caption{Domain decomposition and initial grids for Examples 1 to 4.}
  \label{fig:four-examples}
\end{figure}

The computed solution with the setup in Example 2 using $2\times 2$ domain decomposition and $h = 1/64$ is displayed in Figure~\ref{fig:solution}. In some of the pots, different color map ranges are used in the Stokes and Biot regions in order to better display the variation in the solution.

\begin{figure}
  \centering
  \includegraphics[width=.33\textwidth]{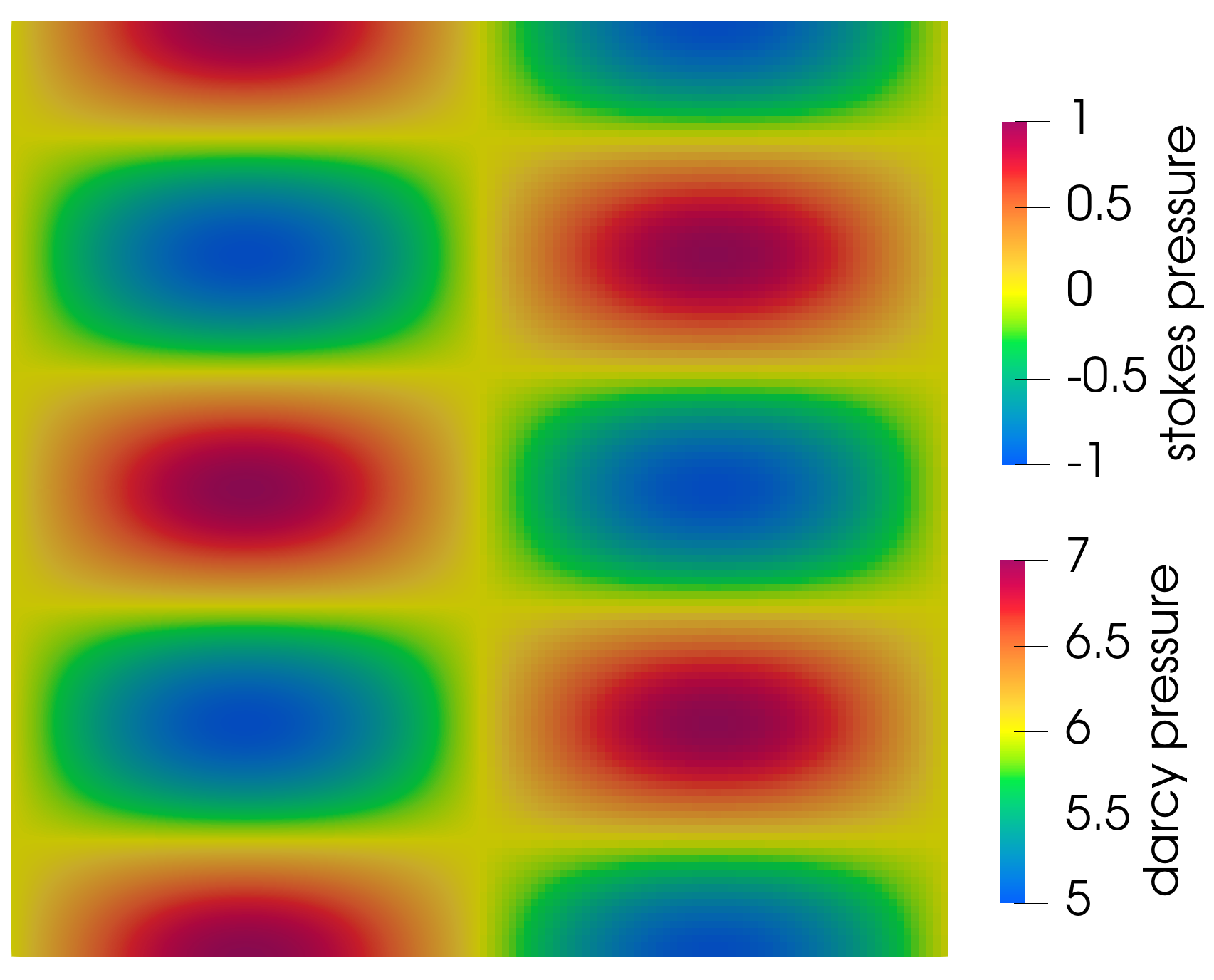}
  \includegraphics[width=.33\textwidth]{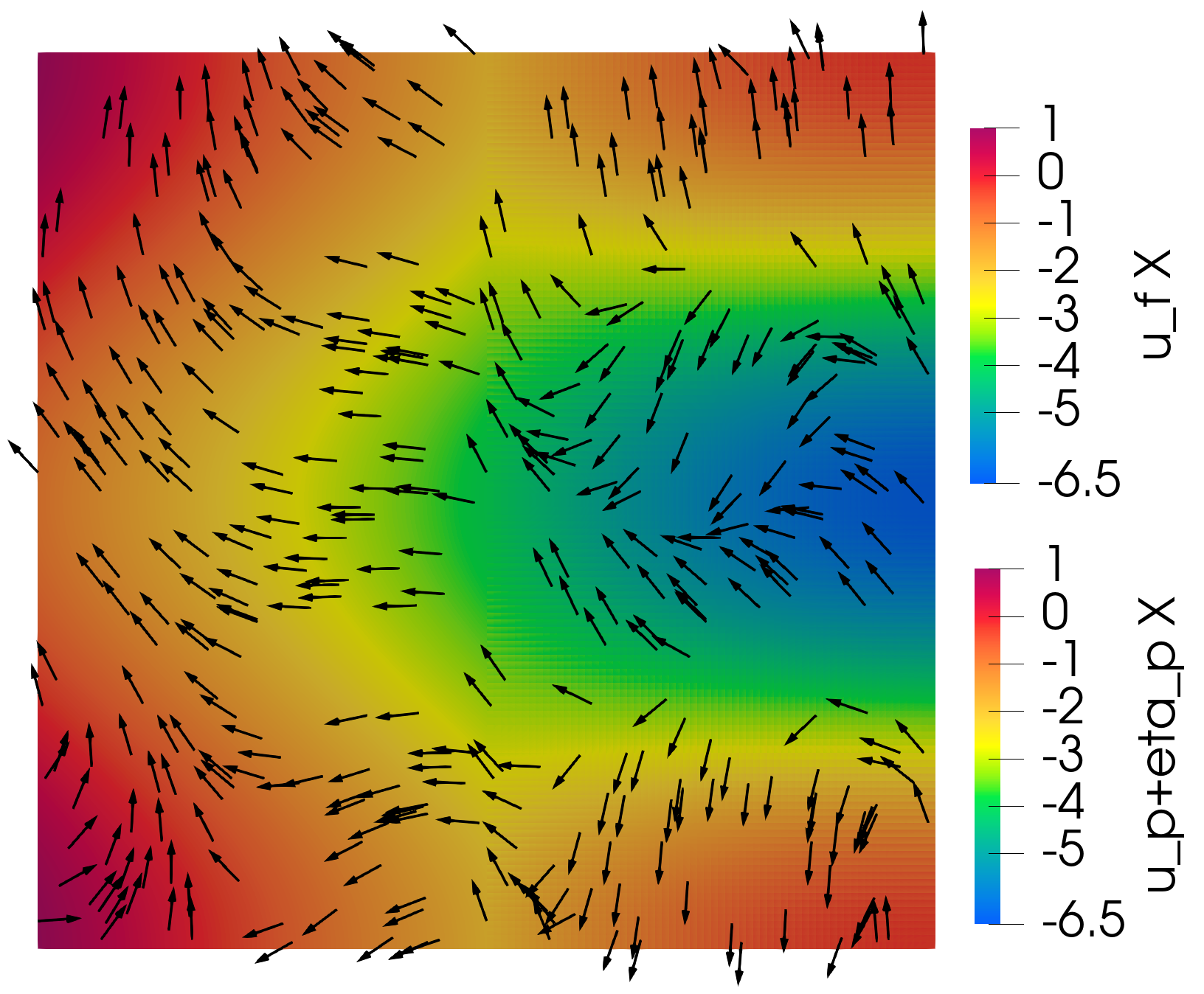}
  \includegraphics[width=.33\textwidth]{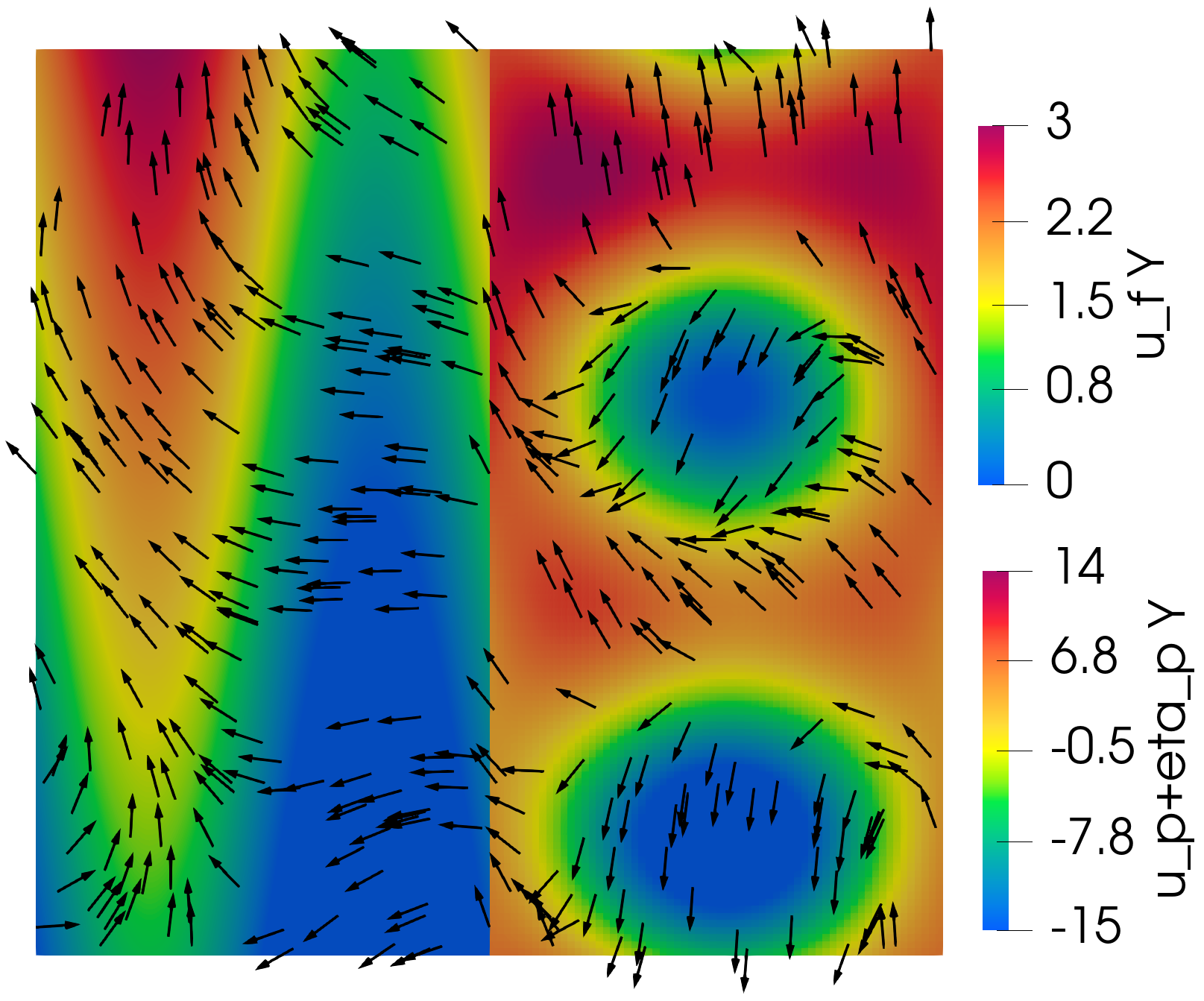}

  \includegraphics[width=.33\textwidth]{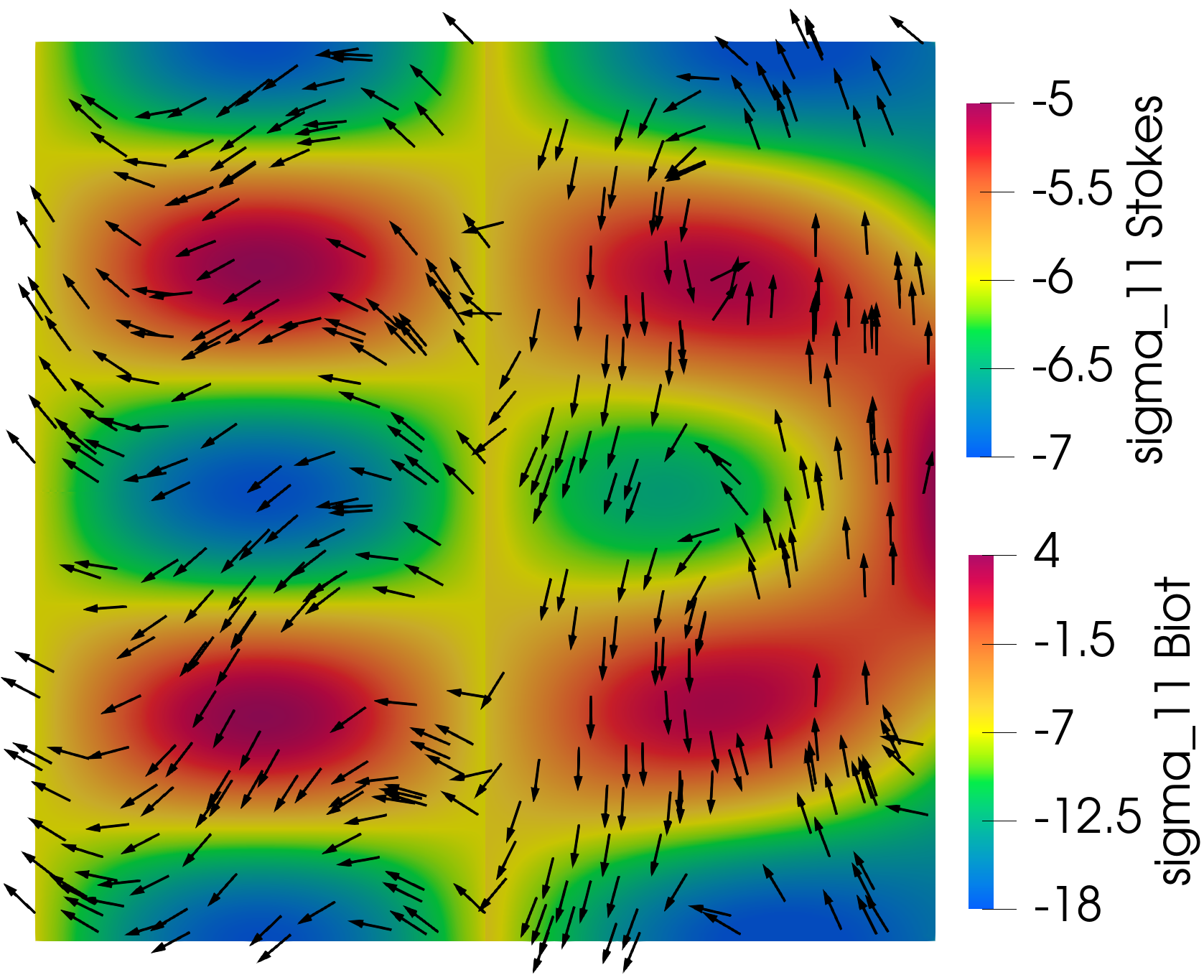}
  \includegraphics[width=.33\textwidth]{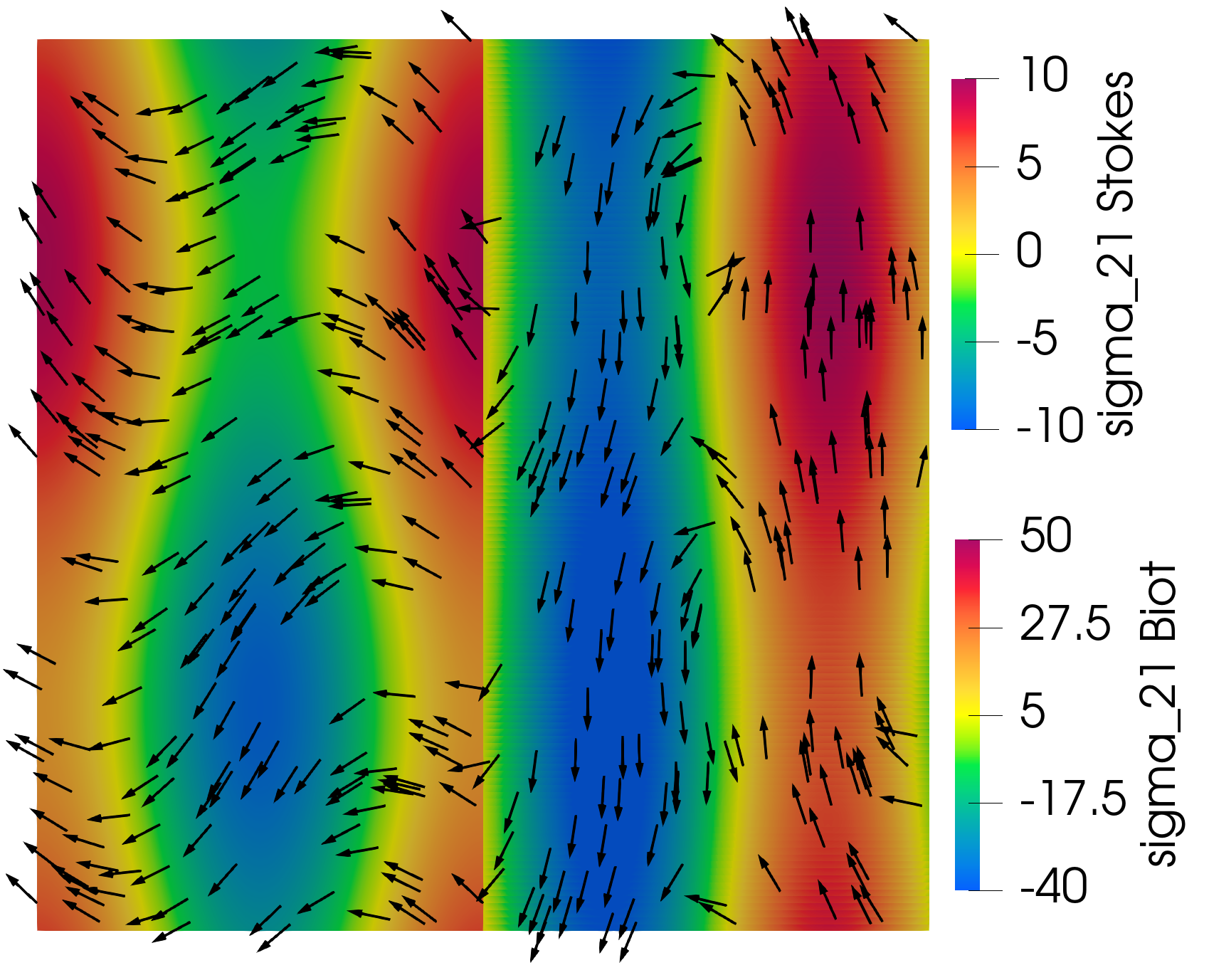}
  \caption{Computed solution in Example 2 with $2\times 2$ domain decomposition and $h = 1/64$. Top row: pressure: $p_f$, $p_p$; $x$-component of velocity: $\bu_{f,x}$, $\bu_{p,x} + \partial_t \disp_{p,x}$; $y$-component of velocity: $\bu_{f,y}$, $\bu_{p,y} + \partial_t\disp_{p,y}$. Bottom row: $xx$-component of stress: $\bsigma_{f,xx}$, $\bsigma_{p,xx}$; $yx$-component of stress: $\bsigma_{f,yx}$, $\bsigma_{p,yx}$. The arrows represent the velocity or stress vectors and are not scaled.}
  \label{fig:solution}
\end{figure}

\subsection{Example 1}

We use two subdomains as shown in the first panel in Figure~\ref{fig:four-examples}. We run the test on a sequence of mesh refinements, starting with an initial grid of $2\times 4$ in each subdomain. In Tables \ref{table2} and \ref{table3} we report the number of GMRES iterations, numerical errors, and convergence rates for $s_0 = 1,\ 10^{-3}$, respectively. The method exhibits growth in the number of iterations at the rate of $O(h^{-0.5})$ as predicted by \eqref{e5.1} and we observe at least $O(h)$ convergence for all norms as predicted by Theorem \ref{theorem4.5}. Furthermore, we note that there is no change in the number of GMRES iterations and rates of convergence for different values of $s_0$.

\begin{table}
\caption{Convergence for 2x1 domain decomposition with each initial subdomain refinement 2x4 and $s_0 = 1$.}
\label{table2}
\centering
\begin{tabular*}{\textwidth}{@{\extracolsep{\fill}}cccccccccccc} 
\toprule
$h$ & dofs & 
\multicolumn{2}{c}{\# GMRES} & 
\multicolumn{2}{c}{$||\nabla(u_f-u_{fh})||_{\Omega_f}$} & 
\multicolumn{2}{c}{$||p_f-p_{fh}||_{\Omega_f}$} &
\multicolumn{2}{c}{$||\eta_p-\eta_{ph}||_{\Omega_p}$} & 
\multicolumn{2}{c}{$||\gamma-\gamma_{h}||_{\Omega_p}$}\\ \hline
1/2 & 24 & 19 & rate & 2.72e-01 & rate & 3.36e+00 & rate & 8.06e-01 & rate & 7.93e-01 & rate\\ 
1/4 & 48 & 29 & -0.61 & 1.10e-01 & 1.31 & 7.57e-01 & 2.15 & 4.24e-01 & 0.93 & 5.11e-01 & 0.63\\ 
1/8 & 96 & 43 & -0.57 & 2.92e-02 & 1.92 & 2.04e-01 & 1.89 & 2.10e-01 & 1.01 & 2.60e-01 & 0.98\\
1/16 & 192 & 61 & -0.50 & 7.46e-03 & 1.97 & 5.33e-02 & 1.93 & 1.04e-01 & 1.01 & 1.30e-01 & 1.00\\ 
1/32 & 384 & 86 & -0.50 & 1.88e-03 & 1.99 & 1.36e-02 & 1.97 & 5.19e-02 & 1.00 & 6.47e-02 & 1.00\\ 
1/64 & 768 & 121 & -0.49 & 4.73e-04 & 1.99 & 3.43e-03 & 1.99 & 2.59e-02 & 1.00 & 3.23e-02 & 1.00 \\
\hline
\end{tabular*}
\vspace{0.2em}
\begin{tabular*}{\textwidth}{@{\extracolsep{\fill}}ccccccccccc} \hline
$h$ & 
\multicolumn{2}{c}{$||p_p-p_{ph}||_{\Omega_p}$} & 
\multicolumn{2}{c}{$||u_p-u_{ph}||_{\Omega_p}$} & 
\multicolumn{2}{c}{$||\nabla\cdot(u_p-u_{ph})||_{\Omega_p}$} & 
\multicolumn{2}{c}{$||\sigma_p-\sigma_{ph}||_{\Omega_p}$} & 
\multicolumn{2}{c}{$||\nabla\cdot(\sigma_p-\sigma_{ph})||_{\Omega_p}$}\\ \hline
1/2 & 1.04e-01 & rate & 1.01e+00 & rate & 1.01e+00 & rate & 6.12e-01 & rate & 6.59e-01 & rate\\ 
1/4 & 4.47e-02 & 1.21 & 2.35e-01 & 2.11 & 5.51e-01 & 0.87 & 3.21e-01 & 0.93 & 3.45e-01 & 0.93\\ 
1/8 & 2.31e-02 & 0.95 & 6.04e-02 & 1.96 & 2.80e-01 & 0.98 & 1.37e-01 & 1.23 & 1.79e-01 & 0.95\\ 
1/16 & 1.16e-02 & 0.99 & 1.54e-02 & 1.97 & 1.40e-01 & 1.00 & 6.48e-02 & 1.08 & 9.00e-02 & 0.99\\ 
1/32 & 5.82e-03 & 1.00 & 3.88e-03 & 1.99 & 7.01e-02 & 1.00 & 3.20e-02 & 1.02 & 4.51e-02 & 1.00\\ 
1/64 & 2.91e-03 & 1.00 & 9.75e-04 & 1.99 & 3.50e-02 & 1.00 & 1.59e-02 & 1.00 & 2.26e-02 & 1.00\\ 
\bottomrule
\end{tabular*}

\end{table}

\begin{table}
\caption{Convergence for 2x1 domain decomposition with each initial subdomain refinement 2x4 and $s_0 = 10^{-3}$.}
\label{table3}
\centering
\begin{tabular*}{\textwidth}{@{\extracolsep{\fill}}cccccccccccc} 
\toprule
h & dofs & 
\multicolumn{2}{c}{\# GMRES} & 
\multicolumn{2}{c}{$||\nabla(u_f-u_{fh})||_{\Omega_f}$} & 
\multicolumn{2}{c}{$||p_f-p_{fh}||_{\Omega_f}$} &
\multicolumn{2}{c}{$||\eta_p-\eta_{ph}||_{\Omega_p}$} & 
\multicolumn{2}{c}{$||\gamma-\gamma_{h}||_{\Omega_p}$}\\ \hline
1/2 & 24 & 19 & rate & 2.73e-01 & rate & 3.37e+00 & rate & 8.06e-01 & - & 7.93e-01 & rate\\ 
1/4 & 48 & 30 & -0.66 & 1.10e-01 & 1.31 & 7.55e-01 & 2.16 & 4.24e-01 & 0.93 & 5.11e-01 & 0.63\\
1/8 & 96 & 42 & -0.49 & 2.92e-02 & 1.92 & 2.04e-01 & 1.89 & 2.10e-01 & 1.01 & 2.60e-01 & 0.98\\ 
1/16 & 192 & 61 & -0.54 & 7.46e-03 & 1.97 & 5.33e-02 & 1.93 & 1.04e-01 & 1.01 & 1.30e-01 & 1.00\\ 
1/32 & 384 & 86 & -0.50 & 1.88e-03 & 1.99 & 1.36e-02 & 1.97 & 5.19e-02 & 1.00 & 6.47e-02 & 1.00\\ 
1/64 & 768 & 120 & -0.48 & 4.73e-04 & 1.99 & 3.44e-03 & 1.99 & 2.59e-02 & 1.00 & 3.23e-02 & 1.00\\ 
\hline
\end{tabular*}
\vspace{0.2em}
\centering
\begin{tabular*}{\textwidth}{@{\extracolsep{\fill}}ccccccccccc} \hline
h & 
\multicolumn{2}{c}{$||p_p-p_{ph}||_{\Omega_p}$} & 
\multicolumn{2}{c}{$||u_p-u_{ph}||_{\Omega_p}$} & 
\multicolumn{2}{c}{$||\nabla\cdot(u_p-u_{ph})||_{\Omega_p}$} & 
\multicolumn{2}{c}{$||\sigma_p-\sigma_{ph}||_{\Omega_p}$} & 
\multicolumn{2}{c}{$||\nabla\cdot(\sigma_p-\sigma_{ph})||_{\Omega_p}$}\\ \hline
1/2 & 1.07e-01 & rate & 1.02e+00 & rate & 1.01e+00 & rate & 6.13e-01 & rate & 6.59e-01 & rate\\ 
1/4 & 4.48e-02 & 1.25 & 2.36e-01 & 2.11 & 5.51e-01 & 0.87 & 3.21e-01 & 0.93 & 3.45e-01 & 0.93\\ 
1/8 & 2.31e-02 & 0.96 & 6.09e-02 & 1.96 & 2.80e-01 & 0.98 & 1.37e-01 & 1.23 & 1.79e-01 & 0.95\\ 
1/16 & 1.16e-02 & 0.99 & 1.55e-02 & 1.97 & 1.40e-01 & 1.00 & 6.48e-02 & 1.08 & 9.00e-02 & 0.99\\ 
1/32 & 5.82e-03 & 1.00 & 3.91e-03 & 1.99 & 7.01e-02 & 1.00 & 3.20e-02 & 1.02 & 4.51e-02 & 1.00\\ 
1/64 & 2.91e-03 & 1.00 & 9.83e-04 & 1.99 & 3.50e-02 & 1.00 & 1.59e-02 & 1.00 & 2.26e-02 & 1.00\\ 
\bottomrule
\end{tabular*}
\end{table}

\subsection{Example 2}

In this and the subsequent examples, the interfaces $\Gamma_{ff}$ and $\Gamma_{pp}$ have non-zero measures. Here we use 4 subdomains as shown in the second panel in Figure~\ref{fig:four-examples}. For each subdomain, we start with an initial refinement of $2\times 2$. In Tables \ref{table4} and \ref{table5} we report the errors for $s_0 = 1,\ 10^{-3}$, respectively. The method exhibits growth in the number of iterations less than the rate of $O(h^{-1})$ as predicted by \eqref{e5.1} and we observe at least $O(h)$ convergence for all norms as predicted by Theorem \ref{theorem4.5}. Again, there is no change in the number of GMRES iterations and rates of convergence for different values of $s_0$.

\begin{table}
\caption{Convergence for 2x2 domain decomposition with each initial subdomain refinement 2x2 and $s_0 = 1$.}
\label{table4}
\begin{tabular*}{\textwidth}{@{\extracolsep{\fill}}cccccccccccc} 
\toprule
$h$ & dofs & 
\multicolumn{2}{c}{\# GMRES} & 
\multicolumn{2}{c}{$||\nabla(u_f-u_{fh})||_{\Omega_f}$} & 
\multicolumn{2}{c}{$||p_f-p_{fh}||_{\Omega_f}$} &
\multicolumn{2}{c}{$||\eta_p-\eta_{ph}||_{\Omega_p}$} & 
\multicolumn{2}{c}{$||\gamma-\gamma_{h}||_{\Omega_p}$}\\ \hline
1/2 & 46 & 46 & rate & 2.72e-01 & rate & 3.33e+00 & rate & 8.07e-01 & rate & 7.90e-01 & rate\\ 
1/4 & 90 & 89 & -0.95 & 1.10e-01 & 1.30 & 7.56e-01 & 2.14 & 4.24e-01 & 0.93 & 5.11e-01 & 0.63\\ 
1/8 & 178 & 165 & -0.89 & 2.91e-02 & 1.92 & 2.04e-01 & 1.89 & 2.11e-01 & 1.01 & 2.60e-01 & 0.98\\ 
1/16 & 354 & 297 & -0.85 & 7.45e-03 & 1.97 & 5.33e-02 & 1.93 & 1.04e-01 & 1.01 & 1.30e-01 & 1.00\\ 
1/32 & 706 & 484 & -0.70 & 1.88e-03 & 1.99 & 1.36e-02 & 1.97 & 5.20e-02 & 1.00 & 6.47e-02 & 1.00\\ 
1/64 & 1410 & 778 & -0.68 & 4.73e-04 & 1.99 & 3.44e-03 & 1.98 & 2.60e-02 & 1.00 & 3.23e-02 & 1.00\\ \hline
\end{tabular*}
\vspace{0.2em}
\begin{tabular*}{\textwidth}{@{\extracolsep{\fill}}ccccccccccc}\hline 
$h$ & 
\multicolumn{2}{c}{$||p_p-p_{ph}||_{\Omega_p}$} & 
\multicolumn{2}{c}{$||u_p-u_{ph}||_{\Omega_p}$} & 
\multicolumn{2}{c}{$||\nabla\cdot(u_p-u_{ph})||_{\Omega_p}$} & 
\multicolumn{2}{c}{$||\sigma_p-\sigma_{ph}||_{\Omega_p}$} & 
\multicolumn{2}{c}{$||\nabla\cdot(\sigma_p-\sigma_{ph})||_{\Omega_p}$}\\ \hline
1/2 & 1.03e-01 & rate & 1.01e+00 & rate & 1.01e+00 & rate & 6.12e-01 & rate & 6.59e-01 & rate\\ 
1/4 & 4.47e-02 & 1.20 & 2.35e-01 & 2.10 & 5.51e-01 & 0.87 & 3.21e-01 & 0.93 & 3.45e-01 & 0.93\\ 
1/8 & 2.31e-02 & 0.95 & 6.04e-02 & 1.96 & 2.80e-01 & 0.98 & 1.37e-01 & 1.23 & 1.79e-01 & 0.95\\ 
1/16 & 1.16e-02 & 0.99 & 1.54e-02 & 1.97 & 1.40e-01 & 1.00 & 6.48e-02 & 1.08 & 9.00e-02 & 0.99\\ 
1/32 & 5.82e-03 & 1.00 & 3.88e-03 & 1.99 & 7.01e-02 & 1.00 & 3.20e-02 & 1.02 & 4.51e-02 & 1.00\\ 
1/64 & 2.91e-03 & 1.00 & 9.74e-04 & 1.99 & 3.50e-02 & 1.00 & 1.59e-02 & 1.00 & 2.26e-02 & 1.00\\ 
\bottomrule
\end{tabular*}
\end{table}

\begin{table}
\caption{Convergence for 2x2 domain decomposition with each initial subdomain refinement 2x2 and $s_0 = 10^{-3}$.}
\label{table5}
\centering
\begin{tabular*}{\textwidth}{@{\extracolsep{\fill}}cccccccccccc} 
\toprule
$h$ & dofs & 
\multicolumn{2}{c}{\# GMRES} & 
\multicolumn{2}{c}{$||\nabla(u_f-u_{fh})||_{\Omega_f}$} & 
\multicolumn{2}{c}{$||p_f-p_{fh}||_{\Omega_f}$} &
\multicolumn{2}{c}{$||\eta_p-\eta_{ph}||_{\Omega_p}$} & 
\multicolumn{2}{c}{$||\gamma-\gamma_{h}||_{\Omega_p}$}\\ \hline
1/2 & 46 & 46 & rate & 2.72e-01 & rate & 3.32e+00 & rate & 8.07e-01 & rate & 7.90e-01 & rate\\ 
1/4 & 90 & 89 & -0.95 & 1.10e-01 & 1.31 & 7.54e-01 & 2.14 & 4.24e-01 & 0.93 & 5.11e-01 & 0.63\\ 
1/8 & 178 & 165 & -0.89 & 2.91e-02 & 1.92 & 2.03e-01 & 1.89 & 2.11e-01 & 1.01 & 2.60e-01 & 0.98\\ 
1/16 & 354 & 297 & -0.85 & 7.45e-03 & 1.97 & 5.33e-02 & 1.93 & 1.04e-01 & 1.01 & 1.30e-01 & 1.00\\ 
1/32 & 706 & 485 & -0.71 & 1.88e-03 & 1.99 & 1.36e-02 & 1.97 & 5.20e-02 & 1.00 & 6.47e-02 & 1.00\\ 
1/64 & 1410 & 785 & -0.69 & 4.73e-04 & 1.99 & 3.44e-03 & 1.99 & 2.60e-02 & 1.00 & 3.23e-02 & 1.00\\ 
\hline
\end{tabular*}
\vspace{0.2em}
\begin{tabular*}{\textwidth}{@{\extracolsep{\fill}}ccccccccccc} \hline
$h$ & 
\multicolumn{2}{c}{$||p_p-p_{ph}||_{\Omega_p}$} & 
\multicolumn{2}{c}{$||u_p-u_{ph}||_{\Omega_p}$} & 
\multicolumn{2}{c}{$||\nabla\cdot(u_p-u_{ph})||_{\Omega_p}$} & 
\multicolumn{2}{c}{$||\sigma_p-\sigma_{ph}||_{\Omega_p}$} & 
\multicolumn{2}{c}{$||\nabla\cdot(\sigma_p-\sigma_{ph})||_{\Omega_p}$}\\ \hline
1/2 & 1.05e-01 & rate & 1.02e+00 & rate & 1.01e+00 & rate & 6.12e-01 & rate & 6.59e-01 & rate\\ 
1/4 & 4.48e-02 & 1.24 & 2.36e-01 & 2.10 & 5.51e-01 & 0.87 & 3.21e-01 & 0.93 & 3.45e-01 & 0.93\\ 
1/8 & 2.31e-02 & 0.96 & 6.08e-02 & 1.96 & 2.80e-01 & 0.98 & 1.37e-01 & 1.23 & 1.79e-01 & 0.95\\ 
1/16 & 1.16e-02 & 0.99 & 1.55e-02 & 1.97 & 1.40e-01 & 1.00 & 6.48e-02 & 1.08 & 9.00e-02 & 0.99\\ 
1/32 & 5.82e-03 & 1.00 & 3.91e-03 & 1.99 & 7.01e-02 & 1.00 & 3.20e-02 & 1.02 & 4.51e-02 & 1.00\\ 
1/64 & 2.91e-03 & 1.00 & 9.82e-04 & 1.99 & 3.50e-02 & 1.00 & 1.59e-02 & 1.00 & 2.26e-02 & 1.00\\ 
\bottomrule
\end{tabular*}
\end{table}

\subsection{Example 3}

In this example the domain is $\Omega = (0,2)\times (0,1)$. We use 8 subdomains as shown in the third panel in Figure~\ref{fig:four-examples}.. For each subdomain, we start with an initial refinement of $2\times 2$. In Tables \ref{table6} and \ref{table7} we report the errors for $s_0 = 1,\ 10^{-3}$, respectively. The method exhibits growth in the number of iterations less than the predicted rate of $O(h^{-1})$ (see \eqref{e5.1}) and we observe at least $O(h)$ convergence for all norms as predicted by Theorem \ref{theorem4.5}. Furthermore, we note that there is no change in the rates of convergence for different values of $s_0$.

\begin{table}
\caption{Convergence for 4x2 domain decomposition with each initial subdomain refinement 2x2 and $s_0 = 1$.}
\label{table6}
\centering
\begin{tabular*}{\textwidth}{@{\extracolsep{\fill}}cccccccccccc} 
\toprule
$h$ & dofs & 
\multicolumn{2}{c}{\# GMRES} & 
\multicolumn{2}{c}{$||\nabla(u_f-u_{fh})||_{\Omega_f}$} & 
\multicolumn{2}{c}{$||p_f-p_{fh}||_{\Omega_f}$} &
\multicolumn{2}{c}{$||\eta_p-\eta_{ph}||_{\Omega_p}$} & 
\multicolumn{2}{c}{$||\gamma-\gamma_{h}||_{\Omega_p}$}\\ \hline
1/4 & 112 & 96 & rate & 1.03e-01 & rate & 5.86e-01 & - & 4.24e-01 & rate & 5.26e-01 & rate\\ 
1/8 & 216 & 172 & -0.84 & 2.71e-02 & 1.93 & 1.36e-01 & 2.11 & 2.16e-01 & 0.97 & 2.73e-01 & 0.95\\ 
1/16 & 424 & 319 & -0.89 & 6.91e-03 & 1.97 & 3.39e-02 & 2.00 & 1.08e-01 & 1.00 & 1.37e-01 & 1.00\\ 
1/32 & 840 & 567 & -0.83 & 1.74e-03 & 1.99 & 8.57e-03 & 1.98 & 5.38e-02 & 1.00 & 6.83e-02 & 1.00\\ 
1/64 & 1672 & 880 & -0.63 & 4.39e-04 & 1.99 & 2.23e-03 & 1.94 & 2.69e-02 & 1.00 & 3.41e-02 & 1.00\\ 
\hline
\end{tabular*}
\vspace{0.2em}
\begin{tabular*}{\textwidth}{@{\extracolsep{\fill}}ccccccccccc} \hline
$h$ & 
\multicolumn{2}{c}{$||p_p-p_{ph}||_{\Omega_p}$} & 
\multicolumn{2}{c}{$||u_p-u_{ph}||_{\Omega_p}$} & 
\multicolumn{2}{c}{$||\nabla\cdot(u_p-u_{ph})||_{\Omega_p}$} & 
\multicolumn{2}{c}{$||\sigma_p-\sigma_{ph}||_{\Omega_p}$} & 
\multicolumn{2}{c}{$||\nabla\cdot(\sigma_p-\sigma_{ph})||_{\Omega_p}$}\\ \hline
1/4 & 4.47e-02 & rate & 2.29e-01 & rate & 5.53e-01 & rate & 3.26e-01 & rate & 3.33e-01 & rate\\ 
1/8 & 2.31e-02 & 0.95 & 5.93e-02 & 1.95 & 2.80e-01 & 0.98 & 1.39e-01 & 1.23 & 1.74e-01 & 0.94\\ 
1/16 & 1.16e-02 & 0.99 & 1.51e-02 & 1.97 & 1.40e-01 & 1.00 & 6.57e-02 & 1.08 & 8.78e-02 & 0.98\\ 
1/32 & 5.82e-03 & 1.00 & 3.81e-03 & 1.99 & 7.01e-02 & 1.00 & 3.24e-02 & 1.02 & 4.40e-02 & 1.00\\ 
1/64 & 2.91e-03 & 1.00 & 9.56e-04 & 1.99 & 3.50e-02 & 1.00 & 1.61e-02 & 1.00 & 2.20e-02 & 1.00\\ 
\bottomrule
\end{tabular*}
\end{table}

\begin{table}
\caption{Convergence for 4x2 domain decomposition with each initial subdomain refinement 2x2 and $s_0 = 10^{-3}$.}
\label{table7}
\centering
    \begin{tabular*}{\textwidth}{@{\extracolsep{\fill}}cccccccccccc} 
\toprule
$h$ & dofs & 
\multicolumn{2}{c}{\# GMRES} & 
\multicolumn{2}{c}{$||\nabla(u_f-u_{fh})||_{\Omega_f}$} & 
\multicolumn{2}{c}{$||p_f-p_{fh}||_{\Omega_f}$} &
\multicolumn{2}{c}{$||\eta_p-\eta_{ph}||_{\Omega_p}$} & 
\multicolumn{2}{c}{$||\gamma-\gamma_{h}||_{\Omega_p}$}\\ \hline
1/4 & 112 & 96 & rate & 1.03e-01 & rate & 5.86e-01 & rate & 4.24e-01 & rate & 5.26e-01 & rate\\ 
1/8 & 216 & 172 & -0.84 & 2.71e-02 & 1.93 & 1.36e-01 & 2.10 & 2.16e-01 & 0.97 & 2.73e-01 & 0.95\\ 
1/16 & 424 & 319 & -0.89 & 6.90e-03 & 1.97 & 3.40e-02 & 2.00 & 1.08e-01 & 1.00 & 1.37e-01 & 1.00\\ 
1/32 & 840 & 565 & -0.82 & 1.74e-03 & 1.99 & 8.56e-03 & 1.99 & 5.38e-02 & 1.00 & 6.83e-02 & 1.00\\ 
1/64 & 1672 & 868 & -0.62 & 4.41e-04 & 1.98 & 2.27e-03 & 1.91 & 2.69e-02 & 1.00 & 3.41e-02 & 1.00\\ 
\hline
\end{tabular*}

\begin{tabular*}{\textwidth}{@{\extracolsep{\fill}}ccccccccccc} \hline
$h$ & 
\multicolumn{2}{c}{$||p_p-p_{ph}||_{\Omega_p}$} & 
\multicolumn{2}{c}{$||u_p-u_{ph}||_{\Omega_p}$} & 
\multicolumn{2}{c}{$||\nabla\cdot(u_p-u_{ph})||_{\Omega_p}$} & 
\multicolumn{2}{c}{$||\sigma_p-\sigma_{ph}||_{\Omega_p}$} & 
\multicolumn{2}{c}{$||\nabla\cdot(\sigma_p-\sigma_{ph})||_{\Omega_p}$}\\ \hline
1/4 & 4.47e-02 & rate & 2.30e-01 & rate & 5.53e-01 & rate & 3.26e-01 & rate & 3.33e-01 & rate\\ 
1/8 & 2.31e-02 & 0.96 & 5.96e-02 & 1.95 & 2.80e-01 & 0.98 & 1.39e-01 & 1.23 & 1.74e-01 & 0.94\\ 
1/16 & 1.16e-02 & 0.99 & 1.52e-02 & 1.97 & 1.40e-01 & 1.00 & 6.57e-02 & 1.08 & 8.78e-02 & 0.98\\ 
1/32 & 5.82e-03 & 1.00 & 3.84e-03 & 1.99 & 7.01e-02 & 1.00 & 3.24e-02 & 1.02 & 4.40e-02 & 1.00\\ 
1/64 & 2.91e-03 & 1.00 & 9.63e-04 & 1.99 & 3.50e-02 & 1.00 & 1.61e-02 & 1.00 & 2.20e-02 & 1.00\\ 
\bottomrule
\end{tabular*}

\end{table}

\subsection{Example 4}

In this example we use 4 subdomains with a nonmatching mesh along the Stokes-Biot interface as shown in the fourth panel in Figure~\ref{fig:four-examples}.. For the Stokes subdomains, we start with an initial grid of $3\times 3$ and for Biot with $2\times 2$. In Table \ref{table8} we report the errors. The method exhibits growth in the number of iterations less than the rate of $O(h^{-1})$ as predicted by \eqref{e5.1} and we observe at least $O(h)$ convergence for all norms as predicted by Theorem \ref{theorem4.5}.

\begin{table}
\caption{Convergence for 2x2 domain decomposition with each initial subdomain refinement 3x3 for Stokes and 2x2 for Biot subdomains, and $s_0 = 1$.}
\label{table8}
\centering
\begin{tabular*}{\textwidth}{@{\extracolsep{\fill}}cccccccccccc} 
\toprule
$h$ & dofs & 
\multicolumn{2}{c}{\# GMRES} & 
\multicolumn{2}{c}{$||\nabla(u_f-u_{fh})||_{\Omega_f}$} & 
\multicolumn{2}{c}{$||p_f-p_{fh}||_{\Omega_f}$} &
\multicolumn{2}{c}{$||\eta_p-\eta_{ph}||_{\Omega_p}$} & 
\multicolumn{2}{c}{$||\gamma-\gamma_{h}||_{\Omega_p}$}\\ \hline
1/2 & 50 & 50 & rate & 2.85e-01 & rate & 3.44e+00 & rate & 8.07e-01 & rate & 7.90e-01 & rate\\ 
1/4 & 98 & 92 & -0.88 & 6.85e-02 & 2.06 & 7.17e-01 & 2.26 & 4.24e-01 & 0.93 & 5.11e-01 & 0.63\\ 
1/8 & 194 & 168 & -0.87 & 1.89e-02 & 1.86 & 1.98e-01 & 1.86 & 2.11e-01 & 1.01 & 2.60e-01 & 0.98\\ 
1/16 & 386 & 308 & -0.87 & 4.94e-03 & 1.93 & 5.22e-02 & 1.92 & 1.04e-01 & 1.01 & 1.30e-01 & 1.0\\ 
1/32 & 770 & 499 & -0.70 & 1.26e-03 & 1.97 & 1.33e-02 & 1.97 & 5.20e-02 & 1.00 & 6.47e-02 & 1.00\\ 
1/64 & 1538 & 818 & -0.71 & 3.33e-04 & 1.92 & 3.41e-03 & 1.97 & 2.60e-02 & 1.00 & 3.23e-02 & 1.00\\ 
\hline
\end{tabular*}
\vspace{0.2em}
\begin{tabular*}{\textwidth}{@{\extracolsep{\fill}}ccccccccccc} \hline
$h$ & 
\multicolumn{2}{c}{$||p_p-p_{ph}||_{\Omega_p}$} & 
\multicolumn{2}{c}{$||u_p-u_{ph}||_{\Omega_p}$} & 
\multicolumn{2}{c}{$||\nabla\cdot(u_p-u_{ph})||_{\Omega_p}$} & 
\multicolumn{2}{c}{$||\sigma_p-\sigma_{ph}||_{\Omega_p}$} & 
\multicolumn{2}{c}{$||\nabla\cdot(\sigma_p-\sigma_{ph})||_{\Omega_p}$}\\ \hline
1/2 & 1.02e-01 & rate & 1.01e+00 & rate & 1.01e+00 & rate & 6.12e-01 & rate & 6.59e-01 & rate\\ 
1/4 & 4.47e-02 & 1.20 & 2.35e-01 & 2.10 & 5.51e-01 & 0.87 & 3.21e-01 & 0.93 & 3.45e-01 & 0.93\\ 
1/8 & 2.31e-02 & 0.95 & 6.04e-02 & 1.96 & 2.80e-01 & 0.98 & 1.37e-01 & 1.23 & 1.79e-01 & 0.95\\ 
1/16 & 1.16e-02 & 0.99 & 1.54e-02 & 1.97 & 1.40e-01 & 1.00 & 6.48e-02 & 1.08 & 9.00e-02 & 0.99\\ 
1/32 & 5.82e-03 & 1.00 & 3.88e-03 & 1.99 & 7.01e-02 & 1.00 & 3.20e-02 & 1.02 & 4.51e-02 & 1.00\\ 
1/64 & 2.91e-03 & 1.00 & 9.74e-04 & 1.99 & 3.50e-02 & 1.00 & 1.59e-02 & 1.00 & 2.26e-02 & 1.00\\
\bottomrule
\end{tabular*}

\end{table}

In summary, the numerical experiments show that the method exhibits number of GMRES iterations consistent with the theoretical results in the case of a normal interface operator. The discretization convergence rates for all variables in their corresponding norms are $O(h)$, as predicted by the theory for this choice of finite element spaces.


    


\begin{figure}
  \centerline{\includegraphics[width = \textwidth]{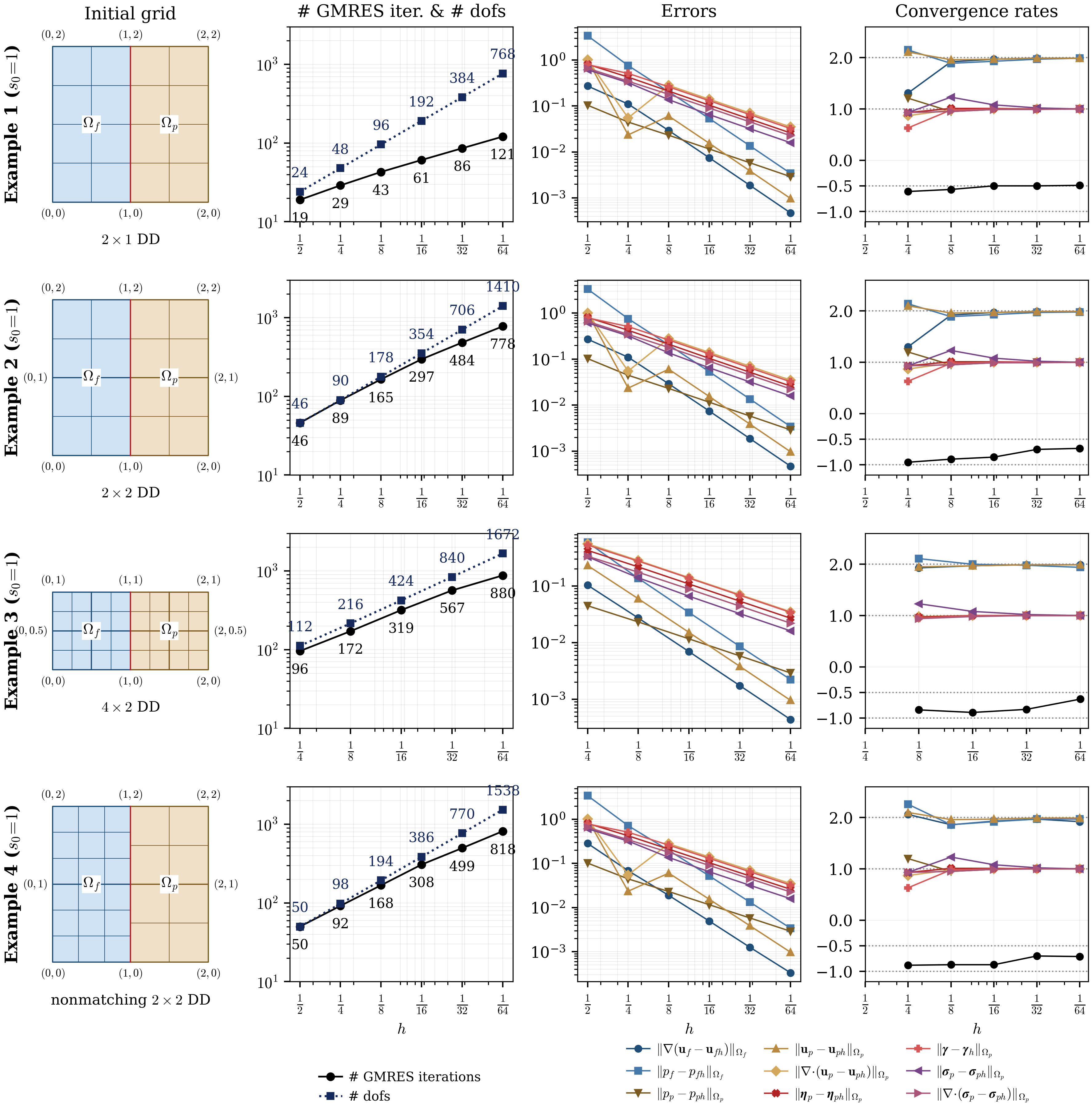}}
  \caption{Number of GMRES iterations, numerical errors, and convergence rates for Examples 1 to 4.}
  \label{fig:results}
\end{figure}

\section{Conclusions}\label{sec:concl}

We developed a non-overlapping domain decomposition method for the coupled Stokes-Biot model of fluid-poroelastic structure interaction. The method provides a scalable algorithm for the solution of the problem on massively parallel computers with distributed memory. By computing the Schur complement, the coupled system is reduced to an interface problem for the Lagrange multipliers that model the normal fluid stress on Stokes-Stokes interfaces, and the displacement and pressure on Stokes-Biot and Biot-Biot interfaces. The interface problem is solved with a Krylov space iterative method -- GMRES. Each iteration requires solving simpler single-physics Stokes or Biot subdomain problems. Furthermore, the method allows for non-matching subdomain grids at the Stokes-Biot interfaces. We proved that subdomain problems are well posed and that the interface operator is positive definite, and established spectral bounds to analyze the convergence of the interface GMRES. A series of numerical experiments illustrate the efficiency, accuracy, and robustness of the method.

Several extensions of the presented work are possible. These include an extension to a multiscale mortar framework \cite{arbogast2000mixed,APWY} allowing for coarse scale approximation of the Lagrange multipliers with improved computational efficiency, as well as non-matching Stokes-Stokes and Biot-Biot interfaces, as e.g. in \cite{ManuMortarBiot,StokesDarcyMortar,Kim-elast-BDDC}, the use of offline multiscale stress-flux basis \cite{ganis2009implementation,EldarElast,ManuMortarBiot}, the development of optimal interface preconditioners of balancing type \cite{pencheva2003balancing,flux-mortar-mpfa}, and the application of machine learning techniques for the solution of the subdomain problems \cite{ML-DD}. These will be investigated in a future work.

\bibliographystyle{abbrv}
\bibliography{references}
\end{document}